\documentclass[11pt,twoside,reqno,letter]{amsart}
\usepackage[margin=1in]{geometry}
\usepackage{amsmath,graphicx,varioref,amscd,amssymb,color,bm,amsthm,epsfig,color,enumerate,fancybox,bbm,esint,upref,xcolor,latexsym,mathtools,breqn,appendix,setspace,physics,tikz}
\usepackage[pdftex, colorlinks=true, linktocpage=true]{hyperref}

\makeatletter
\def\@tocline#1#2#3#4#5#6#7{\relax
  \ifnum #1>\c@tocdepth 
  \else
    \par \addpenalty\@secpenalty\addvspace{#2}%
    \begingroup \hyphenpenalty\@M
    \@ifempty{#4}{%
      \@tempdima\csname r@tocindent\number#1\endcsname\relax
    }{%
      \@tempdima#4\relax
    }%
    \parindent\z@ \leftskip#3\relax \advance\leftskip\@tempdima\relax
    \rightskip\@pnumwidth plus4em \parfillskip-\@pnumwidth
    #5\leavevmode\hskip-\@tempdima
      \ifcase #1
      \or\or \hskip 1em \or \hskip 2em \else \hskip 3em \fi%
      #6\nobreak\relax
    \dotfill\hbox to\@pnumwidth{\@tocpagenum{#7}}\par
    \nobreak
    \endgroup
  \fi}
  
  \def\l@subsection{\@tocline{2}{0pt}{2pc}{5pc}{}}
\makeatother

\numberwithin{equation}{section}
\newtheorem{thm}{Theorem}[section]
\newtheorem*{thm*}{Theorem}
\newtheorem{mydef}[thm]{Definition}
\newtheorem{lem}[thm]{Lemma}
\newtheorem*{lem*}{Lemma}
\newtheorem{cor}[thm]{Corollary}
\newtheorem{prop}[thm]{Proposition}
\newtheorem*{prop*}{Proposition}

\newtheorem{rem}[thm]{Remark}
\newcommand{\R}{\mathbb{R}}
\newcommand{\C}{\mathbb{C}}
\newcommand{\N}{\mathbb{N}}

\newcommand{\Leray}{\mathcal{P}}

\begin{document}


\title[Local weak solutions of an NSE-NLS model of superfluids]{Local weak solutions to a Navier-Stokes-nonlinear-Schr\"odinger model of superfluidity}

\author[Jayanti]{Pranava Chaitanya Jayanti}
\address[Jayanti]{\newline
Department of Physics \\ University of Maryland \\ College Park, MD 20742, USA.}
\email[]{\href{jayantip@umd.edu}{jayantip@umd.edu}}

\author[Trivisa]{Konstantina Trivisa}
\address[Trivisa]{\newline
Department of Mathematics \\ University of Maryland \\ College Park, MD 20742, USA.}
\email[]{\href{trivisa@math.umd.edu}{trivisa@math.umd.edu}}

\date{\today}


\keywords{Superfluids; Navier-Stokes equation; Nonlinear Schr\"odinger equation; Local weak solutions; Existence}

\thanks{P.C.J. was partially supported by the Ann Wylie Fellowship at UMD. Both P.C.J. and K.T. gratefully acknowledge the support of the National Science Foundation under the awards DMS-1614964 and DMS-2008568.}

\maketitle

\begin{abstract}
In \cite{Pitaevskii1959PhenomenologicalPoint}, a micro-scale  model of superfluidity was derived from first principles, to describe the interacting dynamics between the superfluid and normal fluid phases of Helium-4. The model couples two of the most fundamental PDEs in mathematics: the nonlinear Schr\"odinger equation (NLS) and the Navier-Stokes equations (NSE). In this article, we show the local existence of weak solutions to this system (in a smooth bounded domain in 3D), by deriving the required a priori estimates. (We will also establish an energy inequality obeyed by the weak solutions constructed in \cite{Kim1987WeakDensity} for the incompressible, inhomogeneous NSE.) To the best of our knowledge, this is the first rigorous mathematical analysis of a bidirectionally coupled system of the NLS and NSE.
\end{abstract}

\setcounter{tocdepth}{2} 
\tableofcontents

\pagebreak

\section{Introduction}  \label{intro}

Superfluidity is a quantum mechanical phenomenon that is not as well-understood as it is well-known. Upon isobaric cooling at low pressures, helium-4 gas first liquefies before giving rise to a superfluid phase below 2.17K. As the temperature drops, the amount of helium in the superfluid phase increases (and that in the normal fluid phase decreases), until eventually at 0K, we get a pure superfluid phase. Since its experimental discovery \cite{Kapitza1938Viscosity-Point,Allen1938FlowII} over 80 years ago, this phenomenon has evolved into an important sub-field of condensed matter physics research. Despite serious and persistent efforts over several decades by some of the most renowned theoretical physicists, we do not have a unique theory that explains reasonably well all of the observed properties. \medskip

The most striking features of superfluid He-4 are the absence of viscosity and the tendency to flow against a thermal gradient, which can be observed in quite dramatic experiments of ``anti-gravity film flows" and the ``fountain effect" respectively \cite{Vinen:808382,Allen1938New2}. Andronikashvili's experiment \cite{Vinen:808382} (attenuated damping of rotating discs as the surrounding He-4 was cooled) showed evidence of the presence of two fluids, giving credence to Landau's two-fluid model \cite{Landau1941TheoryII}. The latter is a semi-microscopic theory that treats the normal fluid as the excitations of a ground-state superfluid, and notes that the two fluids cannot really be compared to a (classical) multiphase flow where each point in spacetime can be \textit{uniquely} identified with a given phase. Using this model, Landau was able to make some remarkably accurate quantitative estimates (for example, the critical velocity). The success of the two-fluid model led to a search for microscopic theories, based on quantum mechanics; these efforts were spearheaded by Onsager, Feynman, Tisza and London, among others. Onsager \cite{Onsager1953IntroductoryTalk} and Feynman \cite{Feynman1955ApplicationHelium} proposed that the excitations described by Landau are manifested as vortex lines (with quantized circulation) in the superfluid, and this was experimentally confirmed by Bewley et. al. \cite{Bewley2006VisualisationVortices} in 2006. \medskip

The transition from normal He-4 to superfluid He-4 is an example of \textit{order-disorder transitions} \cite{Vinen:808382}. Arguing that this transition is quantum mechanical in nature (given the extremely low temperatures and the absence of a solid phase even at absolute zero), and taking into consideration the bosonic nature of He-4, a reasonable approach was to describe the superfluid phase using a weakly interacting model of Bose-Einstein condensates\footnote{Parallels between superfluidity and superconductivity had been drawn for quite some time: the quantized vortex filaments in the former were analogous to the quantized magnetic flux tubes in the latter, and both phenomena were characterized as order-disorder transitions. Furthermore, following the success of the BCS theory of superconductivity, it became clear that the same explanation (Cooper pairing) can be extended to the superfluidity of the \textit{fermionic} He-3.}. Such an approach led to the Gross-Pitaevskii equation (GPE), also known in the mathematics community as the nonlinear Schr\"odinger equation (or NLS), which soon became the most popular superfluid model. It describes low-energy scattering of the condensate particles (at absolute zero), leading to the well-known cubic nonlinearity. Over the last few decades, the NLS has grown to become one of the most studied PDE models in mathematics. It has been studied for well-posedness in a multitude of scenarios \cite{CollianderWell-posednessEquations}, while also being investigated for scattering solutions \cite{Tao2006NonlinearAnalysis,Dodson2016Global2}. The NLS (including a non-local potential) has also been used to model dipolar quantum gases \cite{Carles2008OnGases,Sohinger2011BoundsEquations}. \medskip

By making a simple transformation of variables, the NLS can be recast as a system of compressible Euler equations (referred to as \textit{quantum hydrodynamics} or QHD\footnote{Interestingly, this formulation is used in David Bohm's \textit{pilot wave theory}, a deterministic yet complicated interpretation of quantum mechanics. This posits that a pilot wave (whose dynamics are governed by QHD) \textit{guides} quantum particles in a classical manner, at odds with other descriptions, like the inherently random Copenhagen interpretation or the fantastical multiverse theories.}) with an additional ``quantum pressure" term \cite{Carles2012MadelungKorteweg} of the form $\rho\nabla\left( \frac{\Delta\sqrt{\rho}}{\sqrt{\rho}} \right)$. This system is a member of the class of Korteweg models and has been extensively studied. Hattori and Li \cite{Hattori1994SolutionsType} established the local well-posedness of the 2D viscous QHD equations for high-regularity data, and upgraded this result to global well-posedness in the case of small data \cite{HattoriHarumi1996GlobalMaterials}. After including an external potential which solves the Poisson equation, the resulting QHD-Poisson system was shown to have local strong solutions \cite{JungelAnsgar2002LocalEquations}. Local unique classical solutions were shown to exist for the same model starting from very regular data in 1D \cite{JungelAnsgar2004QuantumDecay}. Furthermore, this result was made global in time for initial conditions that are sufficiently close to a stationary state, also ensuring the solution's exponential convergence to this stationary state. Wang and Guo \cite{Wang2020AModel} derived a blow-up criterion for strong solutions of the QHD equations, and improved it in \cite{Wang2021AModel}. Meanwhile, there has been a lot of interest in weak solutions to QHD-type models. Antonelli and Marcati \cite{Antonelli2009OnDynamics,Antonelli2012TheDimensions} showed the existence of finite-energy global weak solutions for the QHD-Poisson system, by reverting to the Schr\"odinger formulation. In both these works (among others), a novel fractional step method was used: the NLS was solved and the solution was then periodically updated to account for a collision-induced momentum transfer between constituent particles (macroscopically, a drag force). The irrotationality of the velocity field (except at regions of vacuum) was also implemented to characterize the occurrence of quantum vortices. In \cite{Jungel2010GlobalFluids}, the existence of weak solutions to the viscous QHD system in 2D was proven by J\"ungel (but with test functions that vanish at vacuum). These solutions were global in time if the viscosity was smaller than $\hbar$ (the reduced Planck's constant). The proof involved the use of the Bresch-Desjardin entropy functional \cite{Bresch2004QuelquesKorteweg}, and a redefined velocity to convert the continuity equation into a parabolic type. Vasseur and Yu \cite{Vasseur2016GlobalDamping} improved this result to include more standard test functions, adding some physically-motivated drag terms to gain the required compactness properties for $\sqrt{\rho}$. For the QHD-Poisson system (with linear drag) in $\mathbb{T}^3$, non-uniqueness of the global weak solutions was dealt with in \cite{Donatelli2015Well/IllProblems} using convex integration. The same paper also established weak-strong uniqueness when there are no vacuum zones. \medskip

All of the above discussion on the NLS is valid only at absolute zero. At non-zero (and small) temperatures, as mentioned before, there is a normal fluid as well. This prompts the question of modeling the interactions between the two fluids. There exist models at various length scales: micro-, meso- and macroscopic (see \cite{Jayanti2021GlobalEquations} for a brief introduction, and references therein for more details). The basic idea in these models is to intertwine the dynamics of both the fluids, keeping in mind that they can transfer mass and momentum between themselves. Previously, the authors established global well-posedness of strong solutions in 2D for a macro-scale model of superfluidity known as the HVBK equations, which are a modified version of the Navier-Stokes equations (NSE) \cite{Jayanti2021GlobalEquations}. In this article, we will consider weak solutions for a micro-scale model derived by Pitaevskii \cite{Pitaevskii1959PhenomenologicalPoint} which couples the NLS (for the superfluid) and the NSE (for the normal fluid) via a nonlinear interaction that provides a kind of relaxation mechanism. To the best of our knowledge, this is the first investigation of a model that bidirectionally (see Remark \ref{Antonelli Marcati 2-fluid paper}) couples the NLS and the NSE. At this stage, it is obligatory to comment on the latter. On the one hand, the study of the incompressible limit is arguably the most active area of research in applied mathematics (see \cite{Temam1977Navier-StokesAnalysis,Majda2002VorticityFlow,Robinson2016TheEquations} for classical results). At the other end of the spectrum are compressible flows (a little more realistic in some scenarios), which have also been subjected to intense scrutiny in mathematical literature \cite{Feireisl2004DynamicsFluids,Lions1996MathematicalMechanicsb}. In this work, we will occupy a middle ground between the two extremes: an inhomogeneous, incompressible flow, which consists of the compressible NSE appended with the ``divergence-free velocity" condition. In 3D, the existence of (local) weak solutions when the initial density is bounded below was first established by Kazhikov \cite{Kazhikov1974Fluid}, and this was extended to allow for vacuum (regions of zero density) by Kim \cite{Kim1987WeakDensity}. Further improvements were made by Simon \cite{Simon1990NonhomogeneousPressure}, by analyzing the weak and strong continuity at $t=0$, and also proving global weak solutions in a larger function space (similar to Kazhikov) than the previous works. In the case of strong solutions, when the density is bounded below, Ladyzhenskaya and Solonnikov \cite{Ladyzhenskaya1978UniqueFluids} investigated local (global, respectively) unique solvability in 3D (in 2D, respectively) and global uniqueness for small data. Using a compatibility criterion on the initial data, Choe and Kim \cite{Choe2003StrongFluids} showed local existence of a unique strong solution when the density is not bounded below. More recently, Boldrini et al \cite{Boldrini2003Semi-GalerkinFluids} proved the local existence of a unique strong solution for a model of inhomogeneous, incompressible and asymmetric flow (with density bounded below); this was also extended to a global solution for sufficiently small data. \medskip

This work is most closely related to that of \cite{Kim1987WeakDensity} in that we use a similar approach while deriving the a priori estimates for the NSE. However, we work with a density field that is bounded below (positively) and not governed by a simple, homogeneous transport equation. The inhomogeneity is a relaxation mechanism that allows for mass and momentum transfer between the two fluids, as will be seen later on. The presence of a source term that is not non-negative almost-everywhere forces us to account for the unphysical possibility of the density becoming negative in a set of positive measure. To avoid this, we must accordingly limit our existence time. In a departure from \cite{Kim1987WeakDensity}, we also need a bound on $\norm{\partial_t u}_{L^2_{t,x}}$, obtained from the lower bound on $\rho$ and the estimate on $\norm{\sqrt{\rho}\partial_t u}_{L^2_{t,x}}$. This problem was recognized and addressed using higher order a priori estimates based on more regular data, and necessitates the stopping of the evolution of the system before the density reaches zero somewhere in the domain. Furthermore, as a consequence of the nature of the coupling between the two fluids, we also begin from data that is more regular than in \cite{Kim1987WeakDensity} (but less regular than in \cite{Choe2003StrongFluids}), so that we may get an $L^{\infty}_{t,x}$ bound on the normal fluid velocity. The analysis also entails the use of higher-order boundary conditions on the velocity and the wavefunction. In turn, these dictate the choice of basis functions used in constructing the approximations in the semi-Galerkin scheme. We will now discuss the notation used in the article, before describing the model and stating the results. \medskip

\begin{rem}\label{Antonelli Marcati 2-fluid paper}
    After the preparation of this manuscript, it was pointed out to us by Pierangelo Marcati that a coupled 2-fluid model was already used in \cite{AntonelliPaolo2015FiniteSuperfluidity} to analyze superfluidity. We would like to highlight some key differences in the models which result in a significant departure in the approaches used and ultimately, the results. In their work, the authors do not permit any mass transfer between the two fluids, which allows for global-in-time solutions. Moreover, the momentum transfer is unidirectional and linear, affecting only the superfluid phase (as opposed to the bidirectional and nonlinear nature of the coupling in this work). Finally, we consider the problem on a smooth bounded domain and require certain higher-order boundary conditions, while in \cite{AntonelliPaolo2015FiniteSuperfluidity} the problem is set in $\R^3$.
\end{rem} \medskip

\subsection{Notation} \label{notation}
Let $\mathfrak{D}(\Omega)$ be the space of smooth, compactly-supported functions on $\Omega$. Then, $H^s_0(\Omega)$ is the completion of $\mathfrak{D}$ under the Sobolev norm $H^s$. The more general Sobolev spaces are denoted by $W^{s,p}(\Omega)$, where $s\in\R$ is the derivative index and $1\le p\le\infty$ is the integrability index. A dot on top, like $\Dot{H}^s(\Omega)$ or $\Dot{W}^{s,p}$, is used when referring to the homogeneous Sobolev spaces. \medskip

Consider a 3D vector-valued function $u\equiv (u_1,u_2,u_3)$, where $u_i\in \mathfrak{D}(\Omega), i=1,2,3$. The collection of all divergence-free functions $u$ defines $\mathfrak{D}_d(\Omega)$. Then, $H^s_d(\Omega)$ is the completion of $\mathfrak{D}_d(\Omega)$ under the $H^s$ norm. In addition, to say that a complex-valued wavefunction $\psi\in H^s(\Omega)$ means that its real and imaginary parts are the limits (in the $H^s$ norm) of functions in $\mathfrak{D}(\Omega)$. \medskip




For $s\in\R$, $s^-$ is defined to be the set $\{q\in\R \ \lvert \ q<s \}$. For instance, $H^{2^-}$ denotes all Sobolev spaces $H^s$ for $s<2$. Also, the indices may also be specified as a range. Example: $L^{[1,6)}(\Omega) := \{L^p(\Omega)\ | \ 1\le p<6\}$ and $H^{[0,2)}(\Omega) := \{H^s(\Omega)\ | \ 0\le s<2 \}$. \medskip

The $L^2$ inner product, denoted by $\langle \cdot,\cdot \rangle$, is sesquilinear (the first argument is complex conjugated, indicated by an overbar) to accommodate the complex nature of the Schr\"odinger equation. Thus, for example, $\langle \psi,B\psi \rangle = \int_{\Omega} \Bar{\psi}B\psi \ dx$. Needless to say, since the velocity and density are real-valued functions, we will ignore the complex conjugation when they constitute the first argument of the inner product. \medskip

We use the subscript $x$ on a Banach space to denote the Banach space is defined over $\Omega$. For instance, $L^p_x$ stands for the Lebesgue space $L^p(\Omega)$, and similarly for the Sobolev spaces: $H^s_{d,x} := H^s_d(\Omega)$. For spaces/norms over time, the subscript $t$ will denote the time interval in consideration, such as $L^p_t := L^p_{[0,T]}$, where $T$ stands for the local existence time unless mentioned otherwise. The Bochner spaces $L^p(0,T;X)$ and $C([0,T];X)$ have their usual meanings, as ($L^p$ and continuous, respectively) maps from $[0,T]$ to a Banach space $X$. The notation $C_w([0,T];X)$ is the space of weakly continuous functions over $X$, i.e., the set of all $f\in L^{\infty}(0,T;X)$ such that the map $t\mapsto \langle g,f(t) \rangle_{X'\times X}$ is continuous for all $g\in X'$ (the dual of $X$). \medskip

We also use the notation $X\lesssim Y$ to imply that there exists a positive constant $C$ such that $X\le CY$. The dependence of the constant on various parameters (including the initial data), will be denoted using a subscript as $X\lesssim_{k_1,k_2} Y$ or $X\le C_{k_1,k_2}Y$. \bigskip

\subsection{Organization of the paper}

In Section \ref{mathematical model}, we present and discuss the mathematical model, along with statements of the main results. Several a priori estimates are derived in Section \ref{a priori estimates}. The proofs of the local existence of weak solutions, and that of the energy equality, constitute Section \ref{local existence proof}. Finally, in Section \ref{energy inequality proof - Kim}, we establish the energy inequality for the weak solutions constructed in \cite{Kim1987WeakDensity}, which have lower regularity than the ones in this work. \bigskip

\section{Mathematical model and main results} \label{mathematical model}
The superfluid phase is described by a complex wavefunction, whose dynamics are governed by the nonlinear Schr\"odinger equation (NLS), while the normal fluid is modeled using the compressible Navier-Stokes equations (NSE). The full set of equations, in all generality, can be found in Section 2 of \cite{Pitaevskii1959PhenomenologicalPoint}. In what follows, we use a slightly simplified version of the equations, arrived at by making the following assumptions. \medskip

\begin{enumerate}
    \item We consider the commonly used cubic nonlinearity for the NLS. This is done by choosing the internal energy of the system to be $\frac{\mu}{2}\lvert \psi \rvert^4$. We also assume the internal energy is independent of the density of the normal fluid. \medskip
    
    \item We work in the limit of a divergence-free normal fluid velocity. This means that the pressure is a Lagrange multiplier, and renders the equations of state and entropy unnecessary. However, due to the nature of the coupling between the two phases, the density of the normal fluid is not constant. \medskip
    
    \item Planck's constant $(\hbar)$ and mass of the Helium atom $(m)$ have both been set to unity for simplicity. \medskip
    
    \item For the boundary conditions of the velocity and the wavefunction, apart from the fields being zero on the boundary of the domain, we also need vanishing derivatives (up to a certain order). This requirement is purely mathematical in nature, and stems from the nature of the higher-order a priori estimates. This will be more clear once the estimates are derived. \medskip
\end{enumerate}

We now state the equations used in this article.

\begin{align}
    \partial_t \psi + \Lambda B\psi &= -\frac{1}{2i}\Delta\psi + \frac{\mu}{i}\lvert\psi\rvert^2 \psi \tag{NLS} \label{NLS} \\
    B = \frac{1}{2}\left(-i\nabla - u \right)^2 + \mu \lvert \psi \rvert^2 &= -\frac{1}{2}\Delta + iu\cdot\nabla + \frac{1}{2}\lvert u \rvert^2 + \mu \lvert \psi \rvert^2 \tag{CPL} \label{coupling} \\
    \partial_t \rho + \nabla\cdot(\rho u) &= 2\Lambda\Re(\Bar{\psi}B\psi) \tag{CON} \label{continuity} \\
    \partial_t (\rho u) + \nabla\cdot (\rho u \otimes u) + \nabla p - \nu \Delta u &= \!\begin{multlined}[t]
    -2\Lambda \Im(\nabla \Bar{\psi}B\psi) + \Lambda\nabla \Im(\Bar{\psi}B\psi) \tag{NSE} \label{NSE} + \frac{\mu}{2}\nabla\lvert\psi\rvert^4 
    \end{multlined} \\
    \nabla\cdot u &= 0 \tag{DIV} \label{divergence-free}
\end{align}

\medskip

These equations are supplemented with the required initial and boundary conditions\footnote{For a justification of the exclusion of $t=0$ in the boundary conditions for the wavefunction, see Remark \ref{initial data violating boundary conditions}.} on the wavefunction, velocity and density.

\begin{equation*} \tag{INI} \label{initial conditions}
        \psi(0,x) = \psi_0(x) \qquad u(0,x) = u_0(x) \qquad \rho(0,x) = \rho_0(x) \quad a.e. \ x\in\Omega
\end{equation*}

\medskip

\begin{equation} \tag{BC} \label{boundary conditions}
    \begin{aligned}
    u = \frac{\partial u}{\partial n} = 0 \quad &a.e. \ (t,x)\in[0,T]\times\partial\Omega \\
    \psi = \frac{\partial \psi}{\partial n} = \frac{\partial^2 \psi}{\partial n^2} = \frac{\partial^3 \psi}{\partial n^3} = 0 \quad &a.e. \ (t,x)\in(0,T]\times\partial\Omega \notag
    \end{aligned}
\end{equation}

\medskip

\noindent where $n$ is the outward normal direction on the boundary, and $T$ is the local existence time.

\medskip

Here, $\psi$ is the wavefunction describing the superfluid phase, while $\rho$, $u$ and $p$ are the density, velocity and pressure (respectively) of the normal fluid. The normal fluid has viscosity $\nu$, and $\mu$ (positive constant) is the strength of the dipole-dipole scattering interactions within the superfluid\footnote{$\mu>0$ (resp. $\mu<0$) is called the defocusing (resp. focusing) NLS.}. Finally, $\Lambda$ is a positive constant that indicates the strength of the coupling between the two phases. The coupling is itself denoted by the nonlinear operator $B$. \medskip

According to the Schr\"odinger equation, the wavefunction's evolution in time is generated by the Hamiltonian (roughly, the energy) of the system. Indeed, the coupling term $B$ is seen to have the structure of relative kinetic energy\footnote{There is also the cubic nonlinearity term, which is to say that the relaxation to equilibrium also depends on the potential energy of the superfluid.} between the two phases. This is perhaps made clear by recalling that the quantum mechanical momentum operator (in the position basis) is $-i\hbar\nabla$. Since the mass has been set to unity, this is also the velocity of the superfluid phase. The purpose of this coupling is to allow for momentum/energy transfer between the two phases as a means of relaxation or dissipation. \medskip

Having stated the model, the notion of weak solutions to \eqref{NLS}, \eqref{NSE}, \eqref{continuity} and \eqref{divergence-free} [with initial conditions \eqref{initial conditions} and boundary conditions \eqref{boundary conditions}], henceforth referred to as the ``\textbf{Pitaevskii model}", is as follows. \bigskip

\begin{mydef}[Weak solutions] \label{definition of weak solutions}
    Let $\Omega \subset \R^3$ be a bounded set with a smooth boundary $\partial\Omega$. For a given time $T>0$, consider the following test functions:
    \begin{enumerate}
        \item a complex-valued scalar field $\varphi \in H^1(0,T;L^2(\Omega))\cap L^2(0,T;H_0^1(\Omega))$,
        
        \item a real-valued, divergence-free (3D) vector field $\Phi \in H^1(0,T;L^2_d(\Omega)) \cap L^2(0,T;H^1_d(\Omega))$, and
        
        \item a real-valued scalar field $\sigma \in H^1(0,T;L^2(\Omega))\cap L^2(0,T;H^1(\Omega))$.
    \end{enumerate}
    
    A triplet $(\psi,u,\rho)$ is called a \textbf{weak solution} to the Pitaevskii model if:
    
    \begin{enumerate} [(i)]
        \item 
        
        \begin{equation}
        \begin{gathered}
            \psi\in L^2(0,T;H^{\frac{7}{2}+\delta}_0(\Omega)) \\
            u\in L^2(0,T;H^{\frac{3}{2}+\delta}_d(\Omega)) \\
            \rho \in L^{\infty}([0,T]\times\Omega)
        \end{gathered}
        \end{equation}\medskip
        
        \item and they satisfy the governing equations in the sense of distributions for all test functions, i.e.,
        
        \begin{multline} \label{weak solution wavefunction}
            -\int_0^T \int_{\Omega} \left[ \psi\partial_t\Bar{\varphi} + \frac{1}{2i}\nabla\psi\cdot\nabla\Bar{\varphi} - \Lambda\Bar{\varphi}B\psi - i\mu\Bar{\varphi}\lvert \psi \rvert^2\psi \right] dx \ dt \\ = \int_{\Omega} \left[ \psi_0\Bar{\varphi}(t=0) - \psi(T)\Bar{\varphi}(T) \right] dx
        \end{multline}
    
        \begin{multline} \label{weak solution velocity}
            -\int_0^T \int_{\Omega} \left[ \rho u\cdot \partial_t \Phi + \rho u\otimes u:\nabla\Phi - \nu\nabla u:\nabla\Phi - 2\Lambda\Phi\cdot\Im(\nabla\Bar{\psi}B\psi) \right] dx \ dt \\ = \int_{\Omega} \left[ \rho_0 u_0 \Phi(t=0) - \rho(T)u(T)\Phi(T) \right] dx
        \end{multline}
    
        \begin{equation} \label{weak solution density}
            -\int_0^T \int_{\Omega} \left[ \rho \partial_t \sigma + \rho u\cdot\nabla\sigma + 2\Lambda \sigma \Re(\Bar{\psi}B\psi) \right] dx \ dt = \int_{\Omega} \left[ \rho_0 \sigma(t=0) - \rho(T)\sigma(T) \right] dx
        \end{equation}
    
        \noindent where (the initial data) $\psi_0 \in H^{\frac{5}{2}+\delta}_0(\Omega)$, $u_0 \in H^{\frac{3}{2}+\delta}_d(\Omega)$ and $\rho_0 \in L^{\infty}(\Omega)$.
    \end{enumerate}
    
\end{mydef}

\medskip

\begin{rem} \label{gradient terms in NSE}
    We note that the last two terms in \eqref{NSE} are gradients, just like the pressure term, and thus, vanish in the definition of the weak solution (since the test function is divergence-free). Henceforth, we will absorb these two gradient terms into a modified pressure, denoted by $\Tilde{p}$ wherever necessary.
\end{rem}

\medskip

The operator $B$ is seen from \eqref{coupling} to be a second-order elliptic operator, with time-dependent coefficients. This causes a few problems:

\begin{enumerate}
    \item Even though $B$ is non-negative (shown in Lemma \ref{coupling symmetric and non-negative}), and dissipative in nature, the eigenfunctions of its linear part cannot be used as a basis for the semi-Galerkin scheme employed here. This is because the eigenvalues and eigenfunctions depend on time, requiring severe assumptions on their time-regularity. Moreover, the linear part of $B$ does not have a spectral gap at $0$: its eigenvalues are not known to be bounded away from zero. \medskip
    
    \item We also do not expand $B$ into its constituent terms (to transfer some derivatives to the test functions, for instance), because as will be shown later on, the a priori estimates contain dissipative terms like $\lVert B\psi \rVert_{L^2(0,T;H^s(\Omega))}$. Thus, separating $B$ will make it impossible to make use of the energy structure of the model. \bigskip
\end{enumerate}

We are now ready to state our main results.

\begin{thm} [Local existence] \label{local existence}
    For any $\delta \in (0,\frac{1}{2})$, let $\psi_0 \in H^{\frac{5}{2}+\delta}_0(\Omega)$ and $u_0\in H^{\frac{3}{2}+\delta}_d(\Omega)$. Suppose $\rho_0$ is bounded both above and below a.e. in $\Omega$, i.e., $0<m\le\rho_0\le M<\infty$. Then, there exists a local existence time $T$ and at least one weak solution $(\psi,u,\rho)$ to the Pitaevskii model. In particular, the weak solutions have the following regularity:
    
    \begin{gather}
        \psi\in C(0,T;H^{\frac{5}{2}+\delta}_0(\Omega)) \cap L^2(0,T;H^{\frac{7}{2}+\delta}_0(\Omega)) \label{weak solution psi regularity} \\
        u\in C(0,T;H^{\frac{3}{2}+\delta}_d(\Omega)) \cap L^2(0,T;H^{2}_d(\Omega)) \label{weak solution u regularity} \\
        \rho\in L^{\infty}([0,T]\times \Omega)\cap C(0,T;L^2(\Omega)) \label{weak solution rho regularity}
    \end{gather}\medskip
    
    \noindent where $T$ depends on $\varepsilon\in (0,m)$, the allowed infimum of the density field (see Definition \ref{local existence time definition} below). In addition, the weak solutions (not necessarily unique) also satisfy the following energy equality:
    
    \begin{equation} \label{energy equality for weak solutions}
    \begin{multlined}
        \left( \frac{1}{2}\norm{\sqrt{\rho}u}_{L^2_x}^2 + \frac{1}{2}\norm{\nabla\psi}_{L^2_x}^2 + \frac{\mu}{2}\norm{\psi}_{L^4_x}^4 \right)(t) + \nu\norm{\nabla u}_{L^2_{[0,t]}L^2_x}^2 + 2\Lambda\norm{B\psi}_{L^2_{[0,t]}L^2_x}^2 \\ = \frac{1}{2}\norm{\sqrt{\rho_0}u_0}_{L^2_x}^2 + \frac{1}{2}\norm{\nabla\psi_0}_{L^2_x}^2 + \frac{\mu}{2}\norm{\psi_0}_{L^4_x}^4 \quad a.e. \ t\in [0,T]
    \end{multlined}
    \end{equation}\medskip
    
\end{thm}

The proof of Theorem \ref{local existence} will utilize a semi-Galerkin scheme and the Aubin-Lions-Simon lemma for the required compactness argument. The approach is motivated by that of \cite{Kim1987WeakDensity}, but we begin with more regular data. This is because the presence of $u$ in the nonlinear coupling means we are required to control it in $L^{\infty}(\Omega)$ to prevent the formation of vacuum (and even regions of negative density), as opposed to $H^1(\Omega)$ in \cite{Kim1987WeakDensity}. The local existence time will be determined by fixing a positive lower bound on the density. Several a priori estimates for the Schr\"odinger equation will be established sequentially, starting from the standard mass and energy estimates to those of higher orders. \medskip

\begin{rem} \label{initial data violating boundary conditions}
    While the boundary conditions for the wavefunction include a vanishing third derivative, one may observe that the initial condition only belongs to $H^{\frac{5}{2}+\delta}_0$. This means that the regularity of the initial condition can only ensure a vanishing second derivative on the boundary, and is the reason for not including $t=0$ in the boundary conditions. Of course, the boundary conditions are enforced for $t>0$ by using an appropriate eigenfunction expansion for the semi-Galerkin scheme.
\end{rem} \medskip

\begin{rem} \label{uniqueness in another article}
    In an accompanying article \cite{Jayanti2021UniquenessSuperfluidity}, we also address the uniqueness of the above weak solutions. We demonstrate weak-strong uniqueness, i.e., starting from the same data, if there is a weak solution and a strong solution, then they are identical. We also establish ``weak-moderate" uniqueness when the stronger solution has a regularity intermediate to the weak and strong solutions, provided the data and the existence time (and/or the initial energy) are small enough.
\end{rem} \medskip

As is the case with weak solutions in general, the energy estimate derived for the smooth approximations holds as an inequality when we pass to the limit, due to lower semi-continuity of the norms and weak convergences. Owing to the regularity of the initial data, we can actually obtain an equality. In the case of the (less-regular) weak solutions obtained in \cite{Kim1987WeakDensity}, we will briefly explore the energy \textit{inequality}, something that was not discussed in the original work. \medskip

\begin{prop} [Energy inequality for weak solutions of incompressible, inhomogeneous fluids] \label{energy inequality - Kim}
    Consider an incompressible, inhomogeneous fluid in a smooth and bounded domain $\Omega\subset\R^3$.\medskip
    
    \begin{equation}
        \begin{gathered}
            \partial_t (\rho u) + \nabla\cdot (\rho u \otimes u) + \nabla p - \nu \Delta u = 0 \\
            \partial_t \rho + \nabla\cdot(\rho u) = 0 \\
            \nabla\cdot u = 0 \\
            \\
            u(t,\partial\Omega) = 0 \quad a.e. \ t>0 \\
            u(0,x) = u_0(x) \quad , \quad \rho(0,x) = \rho_0(x)
        \end{gathered}
    \end{equation}\medskip
    
    In \cite{Kim1987WeakDensity}, local weak solutions were constructed starting from initial data $u_0 \in H^1_d(\Omega)$, $\rho_0 \in L^{\infty}(\Omega)$ and $0\le \rho_0(x) \le M<\infty$ a.e. in $\Omega$. More precisely, it was shown that $u\in L^{\infty}(0,T;H^1_d (\Omega))\cap L^2(0,T; H^2(\Omega))$ and $\rho \in L^{\infty}(0,T\times\Omega)$, for a time $T$ that depends only on the norm of the initial velocity and the size of the domain. \medskip

    It also holds that these solutions satisfy an energy inequality, i.e., for a.e. $t\in [0,T]$, \medskip
    
    \begin{equation} \label{energy inequation - Kim}
        \frac{1}{2}\norm{\sqrt{\rho}u}_{L^2_x}^2 (t) + \nu\norm{\nabla u}_{L^2_{[0,t]}L^2_x}^2 \le \frac{1}{2}\norm{\sqrt{\rho_0}u_0}_{L^2_x}^2 \quad a.e. \ t\in [0,T]
    \end{equation}\medskip
\end{prop}

The main achievement in \cite{Kim1987WeakDensity} is not requiring the density to be bounded below by a positive value. Eventually, this manifests itself as a lack of compactness for the velocity, and an inequality in the energy equation results due to lower-semicontinuity of weakly-convergent norms. \bigskip

\subsection{The strategy} \label{the strategy}
The nonlinear coupling terms in \eqref{NLS} and \eqref{NSE} may be the most conspicuous differences between this model and other standard fluid dynamics models, but the source term in \eqref{continuity} is the most pernicious. We will motivate and discuss the strategy towards proving Theorem \ref{local existence}, beginning with a simple observation. Henceforth, we will refer to the linear (in $\psi$) part of $B$ as $B_L$. Thus, 
\begin{equation} \label{defining B_L}
    B_L = B - \mu\lvert\psi\rvert^2 = -\frac{1}{2}\Delta + \frac{1}{2}\lvert u\rvert^2 + iu\cdot\nabla
\end{equation}

\begin{lem} [$B_L$ is symmetric and $B$ is non-negative] \label{coupling symmetric and non-negative}
    
    \begin{enumerate}
        \item $\langle \phi,B_L\psi \rangle = \langle B_L\phi,\psi \rangle \ \forall \ \phi,\psi\in H^1_0(\Omega)$
        
        \item $\langle \psi,B\psi \rangle \ge 0 \ \forall \ \psi\in H^1_0(\Omega)$
    \end{enumerate}
\end{lem}

\proof 

\begin{enumerate}
    \item Starting from \eqref{coupling}, we integrate by parts and use the fact that the wavefunction vanishes on the boundary, and that the velocity is divergence-free.
    
    \begin{align*}
        \langle \phi,B_L\psi \rangle = \int_{\Omega} \Bar{\phi} B_L\psi &= \int_{\Omega} \Bar{\phi} \left[ -\frac{1}{2}\Delta\psi + \frac{1}{2}\lvert u\rvert^2\psi + iu\cdot\nabla\psi \right] \\
        &= \int_{\Omega} \left[ -\frac{1}{2}\Delta\Bar{\phi} + \frac{1}{2}\lvert u\rvert^2\Bar{\phi} - iu\cdot\nabla\Bar{\phi} \right]\psi \\
        &= \int_{\Omega} (\overline{B_L\phi})\psi = \langle B_L\phi,\psi \rangle
    \end{align*}
    
    \item Similarly,
    \begin{align*}
        \langle \psi,B\psi \rangle = \int_{\Omega} \Bar{\psi} B\psi &= \int_{\Omega} \Bar{\psi} \left[ -\frac{1}{2}\Delta\psi + \frac{1}{2}\lvert u\rvert^2\psi + iu\cdot\nabla\psi + \mu\lvert\psi\rvert^2\psi \right] \\
        &= \frac{1}{2}\lVert\nabla\psi\rVert^2_{L^2_x} + \frac{1}{2}\int_{\Omega}\lvert u\rvert^2 \lvert\psi\rvert^2 + \int_{\Omega} iu\Bar{\psi}\cdot\nabla\psi +  \mu\lVert\psi\rVert^4_{L^4_x} \\
        &\ge \mu\lVert\psi\rVert^4_{L^4_x} \ge 0
    \end{align*}\medskip

    In going from the second line to the third, we used H\"older's and Young's inequalities to cancel the third term with the first two terms:
    \begin{gather*}
        \Bigg\lvert\int_{\Omega} iu\Bar{\psi}\cdot\nabla\psi \Bigg\rvert \le \lVert u\psi \rVert_{L^2_x}\lVert\nabla\psi\rVert_{L^2_x} \le \frac{1}{2}\lVert u\psi \rVert^2_{L^2_x} + \frac{1}{2}\lVert\nabla\psi\rVert^2_{L^2_x} \\
        \Rightarrow \int_{\Omega} iu\Bar{\psi}\cdot\nabla\psi \ge -\frac{1}{2}\lVert u\psi \rVert^2_{L^2_x} - \frac{1}{2}\lVert\nabla\psi\rVert^2_{L^2_x}
    \end{gather*}
\end{enumerate}

\qed \medskip

\begin{rem} \label{B has continuous spectrum}
    Note that there is no positive lower bound on $\langle \psi,B_L \psi \rangle$, so the spectrum of $B_L$ need not be strictly positive.
\end{rem} \medskip

Thus, by integrating \eqref{continuity} over the domain, the advective term vanishes and using Lemma \ref{coupling symmetric and non-negative}:

\begin{equation}
    \frac{d}{dt} \int_{\Omega}\rho \ dx = 2\Lambda \Re\int_{\Omega}\Bar{\psi}B\psi \ge 0
\end{equation} \medskip

This implies that the overall mass of the normal fluid does not decrease with time. In other words, the coupling causes superfluid to be converted into normal fluid. However, the RHS of \eqref{continuity} need not be non-negative pointwise, i.e., we are not guaranteed that $2\Lambda \Re\Bar{\psi}B\psi \ge 0 \ a.e. \ x\in\Omega$. This means that the density of the normal fluid may locally decrease to zero, or even negative values. To prevent this physically unrealistic scenario, we choose our existence time for the solution so as to ensure that the density is bounded below. \bigskip

\begin{mydef} [Local existence time] \label{local existence time definition}
    Start with an initial density field $\rho_0$ such that $$0<m\le\rho_0(x)\le M<\infty.$$ 
    Given $0<\varepsilon<m$, we define the local existence time\footnote{Of course, the local existence time depends on the choice of $\varepsilon$ and should ideally be written as $T_{\varepsilon}$. However, we will assume that the value of $\varepsilon$ is fixed throughout this article, and for brevity, drop the subscript.} for the solution as:
    
    \begin{equation} \label{abstract definition existence time}
        T := \inf \{t>0 \ \lvert \ \inf_{\Omega}\rho(t,x) = \varepsilon\}
    \end{equation}
\end{mydef} \bigskip

A formal solution to the continuity equation can be written using the method of characteristics. Let $X_{\alpha}(t)$ be the characteristic starting at $\alpha\in\Omega$. To wit, the characteristic solves the following differential equation:

\begin{equation} \label{characteristics}
    \begin{gathered}
        \frac{d}{dt}X_{\alpha}(t) = u(t,X_{\alpha}(t)) \\
        X_{\alpha}(0) = \alpha\in\Omega
    \end{gathered}
\end{equation} \medskip

Here, $u$ is the velocity of the normal fluid. So, along such characteristics,

\begin{equation} \label{density along characteristic}
    \rho(t,X_{\alpha}(t)) = \rho_0(\alpha) + 2\Lambda\Re\int_0^t \Bar{\psi}B\psi (\tau,X_{\alpha}(\tau)) \ d\tau
\end{equation} \medskip

From \eqref{abstract definition existence time} and \eqref{density along characteristic}, it is clear that a sufficient condition to ensure the density is bounded below by $\varepsilon$ is:

\begin{equation}
    2\Lambda\int_0^T \lvert\Bar{\psi}B\psi\rvert (\tau,X_{\alpha}(\tau)) \ d\tau < m-\varepsilon
\end{equation}\medskip

This can be, in turn, be ensured through the following sufficiency:

\begin{equation} \label{constraint to choose existence time}
    2\Lambda T^{\frac{1}{2}}\lVert\psi\rVert_{L^{\infty}_{[0,T]}L^{\infty}_x} \lVert B\psi\rVert_{L^2_{[0,T]}L^{\infty}_x} < m-\varepsilon
\end{equation}\medskip

So, $T$ is chosen small enough that \eqref{constraint to choose existence time} is satisfied. The boundedness in space of $B\psi$ in the above condition is what leads to the requirement of rather high-regularity data. The momentum equation \eqref{NSE} is itself handled in a manner similar to \cite{Kim1987WeakDensity}. The \eqref{NLS}, on the other hand, is used to derive increasingly higher-order a priori estimates, in order to achieve the required bound on $B\psi$. \bigskip


\subsection{Some useful results and properties}

In the proofs of our main results, we will be using (repeatedly, in some cases) the following lemmas.

\begin{lem}[Poincar\'e inequality] \label{poincare inequality}
    For $k\in\N$ and $f\in H_0^k(\Omega)$, we have $\lVert f \rVert_{L^2(\Omega)} \lesssim \lVert \nabla^k f \rVert_{L^2(\Omega)}$. In particular, the homogeneous norm is equivalent to the standard Sobolev norm: 
    \begin{equation*}
        \lVert f \rVert_{H^k(\Omega)} \lesssim \lVert f \rVert_{\Dot{H}^k(\Omega)} \lesssim \lVert f \rVert_{H^k(\Omega)}
    \end{equation*}
\end{lem}\medskip

This follows from the standard Poincar\'e inequality (see Section 5.6.1 in \cite{Evans2010PartialEquations}) and an induction argument. \bigskip

Next, we will list an analogous result to the above lemma, except that the derivatives are replaced by fractional powers of the negative Laplacian operator. On a torus (with periodic boundary conditions), the action of $(-\Delta)^s$ on functions with zero mean can be described using Fourier series, to show that $\lVert (-\Delta)^s f\rVert_{L^2_x} \equiv \lVert f \rVert_{\Dot{H}^{2s}_x} \equiv \lVert f \rVert_{H^{2s}_x}$ (see Section 2.3 in \cite{Robinson2016TheEquations}). Similarly, on the whole space, one can use the Fourier transform. The case of a (smooth) bounded domain is different $-$ the equivalence between the homogeneous and regular Sobolev norms doesn't hold for all indices. \bigskip

First, we define fractional powers of positive, self-adjoint operators with compact inverses.

\begin{mydef} [Fractional operator spaces] \label{fractional operator spaces definition}
    As described in Section 2.1 of \cite{Fefferman2019SimultaneousEigenspaces}, for a positive, self-adjoint operator $A$ (defined on a separable Hilbert space $H$) with a compact inverse, we will define spaces of its fractional powers using an eigenfunction expansion. Such an operator $A$ has (see Chapter 6 and Appendix D in \cite{Evans2010PartialEquations}) a discrete set of positive and non-decreasing eigenvalues (say $0<e_1\le e_2\le e_3 \dots \rightarrow\infty$), and the corresponding eigenfunctions $(\{ w_j \}\in C^{\infty}(\Bar{\Omega}))$ can be chosen to be orthonormal in the $H$ norm. For $\alpha \ge 0$:

    \begin{equation*}
        D(A^{\alpha}) = \left\{ u = \sum_{j=1}^{\infty} \Hat{u}_j w_j : \sum_{j=1}^{\infty} e_j^{2\alpha} \abs{\Hat{u}_j}^2 < \infty \right\}
    \end{equation*}\medskip

    \noindent where $\hat{u}_j = \langle u,w_j \rangle_H$ (the inner product on $H$). For $\alpha<0$, $D(A^{\alpha})$ is defined as the dual space of $D(A^{-\alpha})$, via the inner product $\langle u,v \rangle_{D(A^{\alpha})} = \sum_{j=1}^{\infty} e_j^{2\alpha}\Hat{u}_j\Hat{v}_j$.
\end{mydef}\medskip

\begin{lem}[Poincar\'e inequality for fractional derivatives] \label{poincare inequality fractional}
    For $s\in\R,s>0$ and $f\in H^s_0(\Omega)$, we have: 
    $$\lVert f \rVert_{H^s(\Omega)} \lesssim \lVert (-\Delta)^{\frac{s}{2}} f \rVert_{L^2(\Omega)} \lesssim \lVert f \rVert_{H^s(\Omega)}$$
\end{lem}\medskip

\begin{proof}
    The statement is actually true for $s>-\frac{1}{2}$ (see discussion in \cite{Guermond2011APowers}), but we are only concerned with positive values of $s$. In Section 3 of \cite{Fefferman2019SimultaneousEigenspaces}, Fefferman et al prove both the inequalities for $0<s\le 1$. For the case of $s>1$, the authors use an induction argument to establish only the first inequality. The argument can be easily used to prove the second inequality, as shown below. \medskip
    
    Let us denote the negative Dirichlet Laplacian by $\mathcal{L}$. Consider $k\in\N,0<r\le 1$. We already have that $\norm{\mathcal{L}^r u}_{L^2_x} \lesssim \norm{u}_{H^{2r}_x}$. Assume this holds for all powers of $\mathcal{L}$ in the range $(0,k]$ for some $k\in\N$. Since $k+r>1$, thus $D(\mathcal{L}^{k+r})\subset D(\mathcal{L})$; this is because of Definition \ref{fractional operator spaces definition} and the strictly positive spectrum of $\mathcal{L}$ (see \cite{Fefferman2019SimultaneousEigenspaces} for details). This means any $u\in D(\mathcal{L}^{k+r})$ also belongs to $D(\mathcal{L})$, implying that $\mathcal{L}u = -\Delta u$. Therefore,
    
    \begin{equation*}
        \norm{\mathcal{L}^{k+r}u}_{L^2_x} = \norm{\mathcal{L}^{k-1+r}\mathcal{L}u}_{L^2_x} \lesssim \norm{\mathcal{L}u}_{H^{2(k-1+r)}_x} = \norm{-\Delta u}_{H^{2(k-1+r)}_x} \lesssim \norm{u}_{H^{2(k+r)}_x}
    \end{equation*}\medskip
    
    The first inequality is due to the inductive assumption, and the last inequality follows from the Poincar\'e inequality in Lemma \ref{poincare inequality}. The inductive step has thus been established, and the proof is completed.
\end{proof} \bigskip

While deriving the highest-order a priori estimate for the wavefunction, we require the following lemma as an abstract integration-by-parts.

\begin{lem} \label{fractional powers self-adjoint property}
    Let $A$ be a positive, self-adjoint operator with a compact inverse, defined on a separable Hilbert space $H$. For $s_1,s_2\in\R$, $s_1,s_2 \ge 0$, and $u,v \in D(A^{s_1+s_2})$,
    
    \begin{equation*}
        \langle A^{s_1}u,A^{s_2}v \rangle_H = \langle u,A^{s_1+s_2}v \rangle_H
    \end{equation*}
\end{lem}\medskip

\begin{proof}
    Using the notation in Definition \ref{fractional operator spaces definition},
    
    \begin{multline*}
        \langle A^{s_1}u,A^{s_2}v \rangle_H = \sum_{j,k=1}^{\infty} e_j^{s_1}e_k^{s_2}\hat{u}_j\hat{v}_j\langle w_j,w_k \rangle_H = \sum_{j,k=1}^{\infty} e_j^{s_1}e_k^{s_2}\hat{u}_j\hat{v}_j \delta_{jk} \\ = \sum_{j,k=1}^{\infty} e_k^{s_1+s_2}\hat{u}_j\hat{v}_k \delta_{jk} = \sum_{j,k=1}^{\infty} e_k^{s_1+s_2}\hat{u}_j\hat{v}_k \langle w_j,w_k \rangle_H = \langle u,A^{s_1+s_2}v \rangle_H
    \end{multline*}
\end{proof}

\medskip

The following well-known Sobolev embeddings (see Chapter 5 of \cite{Evans2010PartialEquations}) will also be useful.

\begin{lem}[Sobolev embeddings] \label{sobolev embeddings}
    For $\Omega$ a smooth, bounded subset of $\R^3$,
    \begin{enumerate}
        \item $H^1(\Omega)\subset L^6(\Omega)$ \ ; \ $H^1(\Omega)\Subset L^p(\Omega), p\in [1,6)$
        \item $H^s(\Omega) \subset L^{\infty}(\Omega) \ \forall \ s>\frac{3}{2}$ \ ; \ $H^s(\R) \subset L^{\infty}(\R) \ \forall \ s>\frac{1}{2}$
        \item $H^s(\Omega) \Subset H^{s'}(\Omega) \ \forall \ s,s'\in\R, s>s'$
    \end{enumerate}
\end{lem}\bigskip

We will use the Aubin-Lions-Simon compactness argument to extract a strongly-converging subsequence, after obtaining uniform a priori bounds on the approximating sequence of solutions. The Lions-Magenes lemma will prove useful in the final a priori estimate, to bound $B\psi$ in $L^{\infty}(\Omega)$.

\begin{lem}[Aubin-Lions-Simon lemma] \label{aubin-lions}
    Let $X_0,X,X_1$ be three Banach spaces such that $X_0\Subset X\subset X_1$. For $1\le p,q\le\infty$, define
    \begin{equation*}
        V := \{u\in L^p(0,T;X_0) , \ \partial_t u\in L^q(0,T;X_1) \}
    \end{equation*}
    Then, $V\Subset L^p(0,T;X)$ when $p<\infty$, and $V\Subset C(0,T;X)$ when $p=\infty$ and $q>1$.
\end{lem}\medskip

\begin{lem}[Lions-Magenes lemma] \label{lions-magenes}
    Let $X,Y,X'$ be three Hilbert spaces such that $X\subset Y\subset X'$, and $X'$ is the dual of $X$ (with the dual pairing denoted by $\langle \cdot ,\cdot\rangle$). If $u\in L^2(0,T;X)$ and its time derivative $u'\in L^2(0,T;X')$, then $u$ is a.e. equal to a function in $C([0,T];Y)$. Moreover, the following equality holds in the sense of scalar distributions:
    
    \begin{equation*}
        \frac{d}{dt}\lVert u\rVert_Y^2 = 2 \langle u',u \rangle_{X'\times X}
    \end{equation*}
\end{lem}\medskip

The first of the two lemmas can be found as Corollary 4 of \cite{Simon1986CompactLpOTB}, while the second one is proved in \cite{Temam1977Navier-StokesAnalysis} (see Lemma 1.2 in Chapter 3). \medskip

Finally, we will also make use of another compactness argument in the context of weakly continuous (in time) maps to Banach spaces, the proof of which is in Appendix C of \cite{Lions1996MathematicalMechanics}. This result will be especially useful while upgrading the regularity of the density field, from that obtained by the application of the Aubin-Lions-Simon lemma.

\begin{lem}[Weak-continuity in time] \label{weak-continuity in time}
    Let $X$ be a separable reflexive Banach space such that $X\subset Y$, where $Y'$ (the dual of $Y$) is separable and dense in $X'$. For a time $T\in (0,\infty)$, consider a sequence of functions $\{f_n\}$ such that: \medskip
    
    \begin{enumerate} [(i)]
        \item $\{f_n\}$ are bounded in $L^{\infty}(0,T;X)$,
        \item $f_n \in C([0,T];Y)$, and
        \item $\langle \omega,f_n(t) \rangle_{Y'\times Y}$ is uniformly continuous in $t\in [0,T]$ uniformly in $n$, for all $\omega \in Y'$.
    \end{enumerate}\medskip
    
    Then, $f_n$ is relatively compact in $C_w([0,T];X)$.
\end{lem}\bigskip

\section{A priori estimates} \label{a priori estimates}

In this section, we will derive all the required a priori estimates, using formal calculations. We will assume the wavefunction and velocity are smooth functions up to the local existence time (with the first four derivatives of the wavefunction and the first derivative of the velocity vanishing on the boundary), such that the density is bounded below by $\varepsilon>0$.

\begin{rem}[Madelung transformation] \label{quantum mechanics interpretation}
    For completeness, and to make for easier understanding of the labels of ``mass" and ``energy", we would like to point out the following: 
    \begin{enumerate}
        \item By substituting $\psi = A e^{iS}$ (polar form) in the Schr\"odinger equation, we are led to a pair of equations that closely resemble the compressible Navier-Stokes equations, if we identify $A = \sqrt{\frac{\rho_s}{m}}$ and $v = \frac{1}{m}\nabla S$. Here, $\rho_s$ is the density of the superfluid density and $v$ is the superfluid velocity, while $m$ is the mass of the superfluid atom. This is known as the Madelung transformation and the resulting system, the equations of quantum hydrodynamics. This motivates the appearance of $\lVert\nabla\psi\rVert_{L^2_x}^2$ in the energy estimate.
        
        \item The absolute square of the wavefunction is the probability density of finding the excitation of the quantum field (a ``particle") at a given point in space-time. This is known as the Copenhagen interpretation of quantum mechanics. The physical density of the superfluid is thus proportional to the probability density (the constant of proportionality being the mass of the superfluid atom, which we have set to be unity).
    \end{enumerate}
\end{rem}\bigskip

\subsection{Superfluid mass estimate}
Multiplying \eqref{NLS} by $\Bar{\psi}$, taking the real part, and integrating over $\Omega$ gives us:

\begin{equation} \label{superfluid mass estimate}
    \frac{d}{dt} \frac{1}{2}\norm{\psi}_{L^2_x}^2 + \Lambda \int_{\Omega}\Re \Bar{\psi}B\psi = 0
\end{equation}\medskip

The Laplacian term on the RHS of \eqref{NLS} vanishes due to the boundary conditions: 

$$\Im\int_{\Omega}\Bar{\psi}\Delta\psi = \Im\int_{\partial\Omega} \Bar{\psi}\frac{\partial\psi}{\partial n} - \Im\int_{\Omega} \abs{\nabla\psi}^2 = 0$$\medskip 

Using Lemma \ref{coupling symmetric and non-negative}, the second term in \eqref{superfluid mass estimate} is non-negative, so we conclude that the mass of superfluid (using the quantum mechanical interpretation of the wavefunction) is uniformly bounded in time:

\begin{equation} \label{superfluid mass bound}
    \norm{\psi}_{L^2_x}(t) \le \norm{\psi_0}_{L^2_x} \quad a.e. \  t \in [0,T]
\end{equation}\bigskip

\subsection{Energy estimate} \label{energy estimate}

Acting the gradient operator on \eqref{NLS}, multiplying by $\nabla\Bar{\psi}$, and taking the real part gives:

\begin{equation} \label{energy estimate first step}
    \partial_t \abs{\nabla\psi}^2 = -\Im(\nabla\Bar{\psi}\cdot\Delta\nabla\psi) - 2\Lambda\Re(\nabla\Bar{\psi}\cdot\nabla(B\psi)) - 2\mu\nabla\abs{\psi}^2\cdot\Im(\Bar{\psi}\nabla\psi)
\end{equation}\medskip

Integrating over $\Omega$, we notice that the first term on the RHS vanishes upon integration by parts due to the boundary conditions\footnote{Both the normal and tangential derivatives of $\psi$ are zero on the boundary, the latter because $\psi$ is zero on a smooth boundary.}. 

$$\Im\int_{\Omega}\nabla\Bar{\psi}\Delta\nabla\psi = \Im\int_{\partial\Omega} \nabla\Bar{\psi}\nabla\frac{\partial\psi}{\partial n} - \Im\int_{\Omega} \abs{\nabla\nabla\psi}^2 = 0 $$\medskip

The second term on the RHS of \eqref{energy estimate first step} is similarly integrated by parts to yield:

\begin{equation} \label{nabla psi norm equation}
    \frac{d}{dt}\frac{1}{2}\norm{\nabla\psi}_{L^2_x}^2 = \Lambda\Re\int_{\Omega}\Delta\Bar{\psi}B\psi - \mu\Im\int_{\Omega}\nabla\abs{\psi}^2\cdot\Bar{\psi}\nabla\psi
\end{equation}\medskip

Now, we rewrite the first term on the RHS by replacing\footnote{This trick will be used again for deriving the higher-order a priori estimates.} the Laplacian in terms of the $B$ operator, giving us a dissipative contribution to the energy estimate. Thus,

\begin{align} 
    \Lambda\Re\int_{\Omega}\Delta\Bar{\psi}B\psi &= -2\Lambda\Re\int_{\Omega} \left( -\frac{1}{2}\Delta\Bar{\psi} \right)B\psi \notag \\
    &= -2\Lambda\Re\int_{\Omega} \left( \overline{B\psi} - \frac{1}{2}\abs{u}^2\Bar{\psi} + iu\cdot\nabla\Bar{\psi} - \mu\abs{\psi}^2\Bar{\psi} \right)B\psi \notag \\
    &= \!\begin{multlined}[t]
    -2\Lambda\norm{B\psi}_{L^2_x}^2 + \Lambda\int_{\Omega}\abs{u}^2\Re(\Bar{\psi}B\psi) + 2\Lambda\int_{\Omega}u\cdot\Im(\nabla\Bar{\psi}B\psi) \\ + 2\mu\Lambda\int_{\Omega}\abs{\psi}^2\Re(\Bar{\psi}B\psi)
    \end{multlined} \label{exchanging derivatives for B}
\end{align}\medskip

We have chosen the \eqref{NLS} to have a cubic nonlinearity, and this contributes a quartic term to the (potential) energy. Multiply \eqref{NLS} by $\Bar{\psi}$ and take the real part to obtain:

\begin{equation*}
    \partial_t \abs{\psi}^2 + \nabla\cdot\Im(\Bar{\psi}\nabla\psi) = -2\Lambda\Re(\Bar{\psi}B\psi)
\end{equation*}\medskip

Multiplying the above equation with $\mu\abs{\psi}^2$ and integrating over $\Omega$ leads to:

\begin{equation} \label{cubic nonlinearity potential energy}
    \frac{d}{dt}\frac{\mu}{2}\norm{\psi}_{L^4_x}^4 - \mu\int_{\Omega}\nabla\abs{\psi}^2\cdot\Im(\Bar{\psi}\nabla\psi) = -2\mu\Lambda\int_{\Omega}\abs{\psi}^2\Re(\Bar{\psi}B\psi)
\end{equation}\medskip

Combining \eqref{nabla psi norm equation}, \eqref{exchanging derivatives for B} and \eqref{cubic nonlinearity potential energy} gives the energy equation for the superfluid phase.

\begin{equation} \label{superfluid energy equation}
    \frac{d}{dt}\left( \frac{1}{2}\norm{\nabla\psi}_{L^2_x}^2 + \frac{\mu}{2}\norm{\psi}_{L^4_x}^4 \right) + 2\Lambda\norm{B\psi}_{L^2_x}^2 = \Lambda\int_{\Omega}\abs{u}^2\Re(\Bar{\psi}B\psi) + 2\Lambda\int_{\Omega}u\cdot\Im(\nabla\Bar{\psi}B\psi)
\end{equation}\medskip

Now, to cancel the terms on the RHS, we need to include the energy equation for the normal fluid. To achieve this, we first rewrite \eqref{NSE} in the \textit{non-conservative form} (see Remark \ref{gradient terms in NSE}).

\begin{equation}  \label{NSE'}
    \rho\partial_t u + \rho u\cdot\nabla u + \nabla \Tilde{p} - \nu \Delta u = -2\Lambda \Im(\nabla \Bar{\psi}B\psi) - 2\Lambda u\Re(\Bar{\psi}B\psi) \tag{NSE'}
\end{equation}\medskip

Taking the inner product of both \eqref{NSE} and \eqref{NSE'} with $u$, using incompressibility, and adding them, we arrive at the energy equation for the normal fluid.

\begin{equation} \label{normal fluid energy equation}
    \frac{d}{dt}\frac{1}{2}\norm{\sqrt{\rho}u}_{L^2_x}^2 + \nu\norm{\nabla u}_{L^2_x}^2 = - 2\Lambda\int_{\Omega}u\cdot\Im(\nabla \Bar{\psi}B\psi) - \Lambda\int_{\Omega}\abs{u}^2\Re(\Bar{\psi}B\psi)
\end{equation}\medskip

Therefore, by adding \eqref{superfluid energy equation} and \eqref{normal fluid energy equation}, we obtain the energy equation for the Pitaevskii model.

\begin{equation} \label{Energy equation}
    \frac{d}{dt}\left( \frac{1}{2}\norm{\sqrt{\rho}u}_{L^2_x}^2 + \frac{1}{2}\norm{\nabla\psi}_{L^2_x}^2 + \frac{\mu}{2}\norm{\psi}_{L^4_x}^4 \right) + \nu\norm{\nabla u}_{L^2_x}^2 + 2\Lambda\norm{B\psi}_{L^2_x}^2 = 0
\end{equation}\medskip

Integrating over $[0,T]$,

\begin{multline} \label{energy bound E0}
    \left( \frac{1}{2}\norm{\sqrt{\rho}u}_{L^2_x}^2 + \frac{1}{2}\norm{\nabla\psi}_{L^2_x}^2 + \frac{\mu}{2}\norm{\psi}_{L^4_x}^4 \right)(t) + \nu\norm{\nabla u}_{L^2_{[0,t]}L^2_x}^2 + 2\Lambda\norm{B\psi}_{L^2_{[0,t]}L^2_x}^2 \\ = E_0 \quad a.e. \ t\in [0,T]
\end{multline}\medskip

\noindent where the $E_0$ is the initial energy of the system, defined as

\begin{equation} \label{E0 definition}
    E_0 := \frac{1}{2}\norm{\sqrt{\rho_0}u_0}_{L^2_x}^2 + \frac{1}{2}\norm{\nabla\psi_0}_{L^2_x}^2 + \frac{\mu}{2}\norm{\psi_0}_{L^4_x}^4
\end{equation}\medskip

Since we have assumed that the density is bounded both above and below, the energy equation implies:

\begin{equation} \label{energy equation inclusion implication}
    u \in L^{\infty}_{[0,T]}L^2_{d,x} \cap L^2_{[0,T]}\dot{H}^1_{d,x} \quad , \quad \psi\in L^{\infty}_{[0,T]}\dot{H}^1_{0,x} \cap L^{\infty}_{[0,T]}L^4_x \quad , \quad B\psi\in L^2_{[0,T]}L^2_x
\end{equation}\medskip

\subsubsection{What does ``$B\psi\in L^2_{[0,T]}L^2_x$" imply for $\psi$?} \label{B psi is in L^2}

The coupling is a second order differential operator, so $B\psi$ being square-integrable should intuitively mean that $\psi\in H^2_x$. We will confirm this with a simple calculation. From \eqref{coupling},

\begin{align}
    -\frac{1}{2}\Delta\psi &= B\psi - iu\cdot\nabla\psi - \frac{1}{2}\abs{u}^2\psi - \mu\abs{\psi}^2\psi \notag \\
    \Rightarrow \norm{\Delta\psi}_{L^2_x} &\lesssim \norm{B\psi}_{L^2_x} + \norm{u\cdot\nabla\psi}_{L^2_x} + \norm{\abs{u}^2\psi}_{L^2_x} + \mu\norm{\abs{\psi}^2\psi}_{L^2_x} \label{B psi what does it mean}
\end{align}\medskip

Now, using Lemma \ref{poincare inequality}, the LHS is equivalent to $\norm{\psi}_{H^2_x}$. We will the last three terms on the RHS, based on \eqref{energy equation inclusion implication}.

\begin{align*}
    \norm{u\cdot\nabla\psi}_{L^2_x} &\lesssim \norm{u}_{L^6_x} \norm{\nabla\psi}_{L^3_x} \tag{H\"older} \\
    &\lesssim \norm{u}_{H^1_x} \norm{\nabla\psi}_{L^2_x}^{\frac{1}{2}} \norm{\nabla\psi}_{L^6_x}^{\frac{1}{2}} \tag{Lemma \ref{sobolev embeddings} + interpolation} \\
    &\lesssim \norm{u}_{H^1_x} \norm{\psi}_{H^1_x}^{\frac{1}{2}} \norm{\psi}_{H^2_x}^{\frac{1}{2}} \tag{Lemma \ref{sobolev embeddings}} \\
    &\lesssim \frac{1}{c}\norm{u}_{H^1_x}^2 \norm{\psi}_{H^1_x} + c\norm{\psi}_{H^2_x} \tag{Young's inequality}
\end{align*}\medskip

The constant $c$ is chosen to be small enough that the second term in the final step can be absorbed into the LHS of \eqref{B psi what does it mean}. Similarly, using H\"older's inequality and Sobolev embedding,

\begin{gather*}
    \norm{\abs{u}^2\psi}_{L^2_x} \lesssim \norm{u}_{L^6_x}^2 \norm{\psi}_{L^6_x} \lesssim \norm{u}_{H^1_x}^2 \norm{\psi}_{H^1_x} \\
    \norm{\abs{\psi}^2\psi}_{L^2_x} \lesssim \norm{\psi}_{L^6_x}^3 \lesssim \norm{\psi}_{H^1_x}^3
\end{gather*}\medskip

Substituting these into \eqref{B psi what does it mean}, and integrating over $[0,T]$,

\begin{equation*}
    \norm{\psi}_{L^1_{[0,T]}H^2_x} \lesssim T^{\frac{1}{2}}\norm{B\psi}_{L^2_{[0,T]}L^2_x} + \norm{u}_{L^2_{[0,T]}H^1_x}^2 \norm{\psi}_{L^{\infty}_{[0,T]}H^1_x} + \mu T \norm{\psi}_{L^{\infty}_{[0,T]}H^1_x}^3 < \infty
\end{equation*}\medskip

Each of the terms on the RHS is finite, according to \eqref{energy equation inclusion implication}. Thus, given the a priori estimates, $B\psi\in L^2_{[0,T]}L^2_x\Rightarrow \psi\in L^1_{[0,T]}H^2_x$. \medskip

A further consequence of this is that the time derivative of the wavefunction can be bounded in a suitable space, permitting the use of the Aubin-Lions-Simon lemma. From \eqref{NLS},

\begin{gather*}
    \norm{\partial_t \psi}_{L^1_{[0,T]}L^2_x} \lesssim \Lambda T^{\frac{1}{2}}\norm{B\psi}_{L^2_{[0,T]}L^2_x} + \norm{\psi}_{L^1_{[0,T]}H^2_x} + \mu T \norm{\psi}_{L^{\infty}_{[0,T]}H^1_x}^3 < \infty \\
    \Rightarrow \norm{\psi}_{L^2_{[0,T]}H^1_x} + \norm{\psi}_{L^1_{[0,T]}H^2_x} + \norm{\partial_t \psi}_{L^1_{[0,T]}L^2_x} < \infty
\end{gather*}\medskip

Thus, from Lemma \ref{aubin-lions}, we can conclude that a sequence of wavefunctions that satisfy the above finiteness condition contain a subsequence that converges strongly in both $L^1_{[0,T]} H^{[0,2)}_{0,x} \cap L^2_{[0,T]} H^{[0,1)}_{0,x} \cap L^2_{[0,T]} L^{[1,6)}_x$. \bigskip

\subsection{Higher-order ``energy" estimate} \label{higher order estimate}

In this subsection, we will utilize the approach in \cite{Kim1987WeakDensity} to derive a higher-order a priori estimate, involving all the three fields $(\psi,u,\rho)$.

\subsubsection{The Schr\"odinger equation} \label{NLS higher order estimate}
Similar to the derivation of the energy equation, we act upon \eqref{NLS} with the Laplacian $\Delta$, multiply by $\Delta\Bar{\psi}$, take the real part and integrate over the domain:

\begin{equation} \label{schrodinger equation higher order first step}
    \frac{d}{dt}\frac{1}{2}\norm{\Delta\psi}_{L^2_x}^2 = - \Lambda\Re\int_{\Omega}(\Delta^2\Bar{\psi})B\psi + \mu\Im\int_{\Omega}(\Delta^2\Bar{\psi})\abs{\psi}^2\psi
\end{equation}\medskip

The first term on the RHS of \eqref{NLS} vanishes as a result of the boundary conditions on $\psi$ (namely, that all tangential derivatives vanish due to the smooth boundary, and that the first three normal derivatives are zero):

$$\Im\int_{\Omega} \Delta\Bar{\psi} \Delta^2\psi = \Im\int_{\partial\Omega} \Delta\Bar{\psi} \Delta\nabla\frac{\partial\psi}{\partial_n} - \Im\int_{\Omega} \abs{\Delta\nabla\Bar{\psi}}^2 = 0$$\medskip

As done before, we will express the Laplacian in \eqref{schrodinger equation higher order first step} in terms of the $B$ operator to obtain a non-negative term and estimate the RHS.\medskip

\begin{enumerate}
    \item \begin{align*}
        -\Lambda\Re\int_{\Omega}(\Delta^2\Bar{\psi})B\psi &= \Lambda\Re\int_{\Omega}\nabla(\Delta\Bar{\psi})\cdot\nabla(B\psi) \\
        &= -2\Lambda\Re\int_{\Omega}\nabla\left[\overline{B\psi} + iu\cdot\nabla\Bar{\psi} - \frac{1}{2}\abs{u}^2\Bar{\psi} - \mu\abs{\psi}^2\Bar{\psi}\right]\cdot\nabla(B\psi) \\
        &= \begin{multlined}[t] -2\Lambda\norm{\nabla(B\psi)}_{L^2_x}^2 + 2\Lambda\Im\int_{\Omega}\nabla(u\cdot\nabla\Bar{\psi})\cdot\nabla(B\psi) \\ + 2\Lambda\Re\int_{\Omega}\nabla(\frac{1}{2}\abs{u}^2\Bar{\psi})\cdot\nabla(B\psi) \\ + 2\Lambda\mu\Re\int_{\Omega} \nabla(\abs{\psi}^2\Bar{\psi})\cdot\nabla(B\psi) \end{multlined}
    \end{align*}\medskip
    
    \item \begin{align*}
        \mu\Im\int_{\Omega}(\Delta^2\Bar{\psi})\abs{\psi}^2\psi &= -\mu\Im\int_{\Omega}\nabla(\Delta\Bar{\psi})\cdot\nabla(\abs{\psi}^2\psi) \\
        &= 2\mu\Im\int_{\Omega} \nabla\left[ \overline{B\psi} + iu\cdot\nabla\Bar{\psi} - \frac{1}{2}\abs{u}^2\Bar{\psi} - \mu\abs{\psi}^2\Bar{\psi} \right]\cdot\nabla(\abs{\psi}^2\psi) \\
        &= \begin{multlined}[t] 2\mu\Im\int_{\Omega} \nabla\left( \overline{B\psi} \right)\cdot\nabla(\abs{\psi}^2\psi) + 2\mu\Re\int_{\Omega} \nabla\left( u\cdot\nabla\Bar{\psi} \right)\cdot\nabla(\abs{\psi}^2\psi) \\ - 2\mu\Im\int_{\Omega} \nabla\left(\frac{1}{2}\abs{u}^2\Bar{\psi} \right)\cdot\nabla(\abs{\psi}^2\psi) \end{multlined}
    \end{align*}\medskip
\end{enumerate}

Thus, \eqref{schrodinger equation higher order first step} becomes:

\begin{equation} \label{schrodinger equation higher order second step}
    \frac{d}{dt}\frac{1}{2}\norm{\Delta\psi}_{L^2_x}^2 + 2\Lambda\norm{\nabla(B\psi)}_{L^2_x}^2 \begin{multlined}[t] = 2\Lambda\Im\int_{\Omega}\nabla(u\cdot\nabla\Bar{\psi})\cdot\nabla(B\psi) \\ + 2\Lambda\Re\int_{\Omega}\nabla\left(\frac{1}{2}\abs{u}^2\Bar{\psi}\right)\cdot\nabla(B\psi) \\ + 2\Lambda\mu\Re\int_{\Omega} \nabla(\abs{\psi}^2\Bar{\psi})\cdot\nabla(B\psi) \\ + 2\mu\Im\int_{\Omega} \nabla\left( \overline{B\psi} \right)\cdot\nabla(\abs{\psi}^2\psi) \\ + 2\mu\Re\int_{\Omega} \nabla\left( u\cdot\nabla\Bar{\psi} \right)\cdot\nabla(\abs{\psi}^2\psi) \\ - 2\mu\Im\int_{\Omega} \nabla\left(\frac{1}{2}\abs{u}^2\Bar{\psi} \right)\cdot\nabla(\abs{\psi}^2\psi) \end{multlined}
\end{equation}\medskip

We can now repeatedly use H\"older's and Young's inequalities to extract out $\norm{\nabla(B\psi)}_{L^2_x}^2$ from each of the first four terms on the RHS (with small enough constants in front to be able to be absorbed into the dissipation term on the LHS). We use the same inequalities with the last two terms. Combining all this, we end up with:

\begin{equation} \label{schrodinger equation higher order third step}
    \frac{d}{dt}\norm{\Delta\psi}_{L^2_x}^2 + \Lambda\norm{\nabla(B\psi)}_{L^2_x}^2 \lesssim \Lambda\norm{\nabla(u\cdot\nabla\psi)}_{L^2_x}^2 + \Lambda\norm{\nabla\left( \frac{1}{2}\abs{u}^2\psi \right)}_{L^2_x}^2 + \gamma\norm{\abs{\psi}^2\psi}_{L^2_x}^2
\end{equation}\medskip

\noindent where $\gamma := \mu^2\left(\Lambda+\frac{1}{\Lambda}\right)$.\medskip

Each of the three terms on the RHS has to be estimated.

\begin{enumerate}
    \item \begin{align*}
        \norm{\nabla(u\cdot\nabla\psi)}_{L^2_x}^2 &\lesssim \norm{\nabla u \cdot\nabla\psi}_{L^2_x}^2 + \norm{u \cdot\nabla\nabla\psi}_{L^2_x}^2 \\
        &\lesssim \norm{\nabla u}_{L^3_x}^2 \norm{\nabla\psi}_{L^6_x}^2 + \norm{u}_{L^{\infty}_x}^2 \norm{\Delta\psi}_{L^2_x}^2 \tag{H\"older + Lemma \ref{poincare inequality}} \\
        &\lesssim \norm{\nabla u}_{L^2_x}\norm{\Delta u}_{L^2_x} \norm{\Delta\psi}_{L^2_x}^2 + \norm{\nabla u}_{L^2_x}^{\frac{1}{2}}\norm{\Delta u}_{L^2_x}^{\frac{3}{2}} \norm{\Delta\psi}_{L^2_x}^2 \tag{Lemma \ref{sobolev embeddings} + Interpolation + Lemma \ref{poincare inequality}}
    \end{align*}\medskip
    
    (The term $\norm{u}_{L^{\infty}_x}$ is bounded above by $\norm{u}_{H^{\frac{7}{4}}_x}$ using Lemma \ref{sobolev embeddings}, which is in turn interpolated between $H^1_x$ and $H^2_x$.) \medskip
    
    \item We will use the property of a 3D vector field, not necessarily divergence-free:
    
    \begin{equation*}
        \nabla\left( \frac{1}{2}\abs{u}^2 \right) = u\cdot\nabla u - \omega\times u
    \end{equation*}\medskip
    
    \noindent where $\omega = \nabla\times u$ is the vorticity. We will also note that $\norm{\omega}_{L^p_x}\le\norm{\nabla u}_{L^p_x}$ for all $1\le p\le\infty$, thus rendering the second term of the RHS in the above inequality equivalent to the first (insofar as $L^p_x$ norms are concerned).
    
    \begin{align*}
        \norm{\nabla\left( \frac{1}{2}\abs{u}^2\psi \right)}_{L^2_x}^2 &\lesssim \norm{\nabla\left( \frac{1}{2}\abs{u}^2\right)\psi}_{L^2_x}^2 + \norm{\frac{1}{2}\abs{u}^2\nabla\psi}_{L^2_x}^2 \\
        &\lesssim \norm{u}_{L^6_x}^2\norm{\nabla u}_{L^3_x}^2\norm{\psi}_{L^{\infty}_x}^2 + \norm{u}_{L^6_x}^4\norm{\nabla\psi}_{L^6_x}^2 \tag{H\"older} \\
        &\lesssim \norm{\nabla u}_{L^2_x}^3\norm{\Delta u}_{L^2_x}\norm{\Delta\psi}_{L^2_x}^2 + \norm{\nabla u}_{L^2_x}^4\norm{\Delta\psi}_{L^2_x}^2 \tag{Lemma \ref{sobolev embeddings} + Interpolation}
    \end{align*}\medskip

    \item \begin{align*}
        \norm{\nabla(\abs{\psi}^2\psi)}_{L^2_x}^2 \lesssim \norm{\psi}_{L^6_x}^4 \norm{\nabla\psi}_{L^6_x}^2 \lesssim \norm{\nabla\psi}_{L^2_x}^4 \norm{\Delta\psi}_{L^2_x}^2
    \end{align*}\medskip
\end{enumerate}

Combining all these into \eqref{schrodinger equation higher order third step} results in:

\begin{multline} \label{schrodinger equation higher order fourth step}
    \frac{d}{dt}\norm{\Delta\psi}_{L^2_x}^2 + \norm{\nabla(B\psi)}_{L^2_x}^2 \lesssim \Lambda\left[\norm{\nabla u}_{L^2_x}\norm{\Delta u}_{L^2_x} + \norm{\nabla u}_{L^2_x}^{\frac{1}{2}} \norm{\Delta u}_{L^2_x}^{\frac{3}{2}}\right. \\ \left. + \norm{\nabla u}_{L^2_x}^3 \norm{\Delta u}_{L^2_x} + \norm{\nabla u}_{L^2_x}^4 + \frac{\gamma}{\Lambda}\norm{\nabla\psi}_{L^2_x}^4 \right]\norm{\Delta\psi}_{L^2_x}^2
\end{multline}

Now, we will use Young's inequality to extract out $\norm{\Delta u}_{L^2_x}^2$ with a certain sufficiently small coefficient (whose choice will be justified later in this subsection). We also recall from \eqref{energy bound E0} that $\norm{\nabla\psi}_{L^2_x} \le 2E_0 \ a.e. \ t\in [0,T]$. In the following, recall that the density is bounded below by $\varepsilon$ and above by $M+m-\varepsilon$ (see discussion following Definition \ref{local existence time definition}). Let us refer to this upper bound on the density as $M'$, for brevity.

\begin{align} \label{delta psi equation}
    \frac{d}{dt}\norm{\Delta\psi}_{L^2_x}^2 + \Lambda\norm{\nabla(B\psi)}_{L^2_x}^2 \lesssim &\frac{\Lambda^2 M'}{\nu^2}\norm{\nabla u}_{L^2_x}^2\norm{\Delta\psi}_{L^2_x}^4 + \frac{\Lambda^4 {M'}^{3}}{\nu^6}\norm{\nabla u}_{L^2_x}^2\norm{\Delta\psi}_{L^2_x}^8 \notag \\ 
    &+ \frac{\Lambda^2 M'}{\nu^2}\norm{\nabla u}_{L^2_x}^6 \norm{\Delta\psi}_{L^2_x}^4 + \Lambda\norm{\nabla u}_{L^2_x}^4\norm{\Delta\psi}_{L^2_x}^2 \notag \\
    &+ \gamma E_0^4 \norm{\Delta\psi}_{L^2_x}^2 + \frac{\nu^2}{CM'} \norm{\Delta u}_{L^2_x}^2
\end{align}\medskip

(The $C$ in the denominator of the last term is a large constant, and simply to ensure the term is small enough to be absorbed into a similar dissipative expression on the LHS that will appear from the normal fluid's estimates.) \medskip

\subsubsection{The Navier-Stokes equation} \label{NSE higher order estimate}

We will now follow the approach in \cite{Kim1987WeakDensity} to derive a higher order estimate for the velocity field, which will be combined with \eqref{delta psi equation} to arrive at a Gr\"onwall inequality argument. Starting with \eqref{NSE'}, we first multiply it by $\partial_t u$ and integrate over the domain:

\begin{multline} \label{NSE higher order first step}
    \int_{\Omega}\rho\abs{\partial_t u}^2 + \frac{\nu}{2}\frac{d}{dt}\norm{\nabla u}_{L^2_x}^2 = -\int_{\Omega}\rho u\cdot\nabla u\cdot\partial_t u - 2\Lambda\int_{\Omega}\partial_t u\cdot\Im(\nabla\Bar{\psi}B\psi) \\ - 2\Lambda\int_{\Omega}\partial_t u\cdot u\Re(\Bar{\psi}B\psi)
\end{multline}\medskip

The three terms on the RHS are estimated as follows. Recall that $\varepsilon$ is the lower bound on the density, and $M'=M+m-\varepsilon$ is the upper bound.\medskip

\begin{enumerate}
    \item \begin{equation*}
        -\int_{\Omega}\rho u\cdot\nabla u\cdot\partial_t u \le \frac{1}{6}\int_{\Omega}\rho\abs{\partial_t u}^2 + \frac{3}{2}M' \int_{\Omega}\abs{u}^2\abs{\nabla u}^2
    \end{equation*}\medskip
    
    The last term is bounded using Lemma \ref{sobolev embeddings} and Sobolev interpolation: $\norm{u}_{L^{\infty}_x} \lesssim \norm{u}_{H^{\frac{7}{4}}_x} \lesssim \norm{u}_{H^1_x}^{\frac{1}{4}}\norm{u}_{H^2_x}^{\frac{3}{4}}$. Thus,
    
    \begin{equation} \label{example of bounding u}
        \int_{\Omega}\abs{u}^2\abs{\nabla u}^2 \lesssim \norm{\nabla u}_{L^2_x}^{\frac{5}{2}}\norm{\Delta u}_{L^2_x}^{\frac{3}{2}}
    \end{equation}\medskip
    
    \item \begin{equation*}
        - 2\Lambda\int_{\Omega}\partial_t u\cdot\Im(\nabla\Bar{\psi}B\psi) \le \frac{1}{6}\int_{\Omega}\rho\abs{\partial_t u}^2 + \frac{6\Lambda^2}{\varepsilon}\norm{\nabla\psi}_{L^6_x}^2\norm{B\psi}_{L^3_x}^2
    \end{equation*}\medskip
    
    The $B\psi$ term is handled via interpolation, while the $\nabla\psi$ term is bounded using Sobolev embedding. We finally use Young's inequality to extract the dissipative term, with a sufficiently small coefficient ($C$ is sufficiently large).
    
    \begin{equation*}
        \norm{\nabla\psi}_{L^6_x}^2\norm{B\psi}_{L^3_x}^2 \lesssim \frac{C\Lambda}{\varepsilon}\norm{B\psi}_{L^2_x}^2\norm{\Delta\psi}_{L^2_x}^4 + \frac{\varepsilon}{C\Lambda}\norm{\nabla(B\psi)}_{L^2_x}^2
    \end{equation*}\medskip
    
    \item \begin{equation*}
        - 2\Lambda\int_{\Omega}\partial_t u\cdot u\Re(\Bar{\psi}B\psi) \le \frac{1}{6}\int_{\Omega}\rho\abs{\partial_t u}^2 + \frac{6\Lambda^2}{\varepsilon}\norm{u}_{L^6_x}^2\norm{\psi}_{L^{\infty}_x}^2\norm{B\psi}_{L^3_x}^2
    \end{equation*}\medskip

    Just as in the previous case,
    
    \begin{equation*}
        \norm{u}_{L^6_x}^2\norm{\psi}_{L^{\infty}_x}^2\norm{B\psi}_{L^3_x}^2 \lesssim \frac{C\Lambda}{\varepsilon}\norm{B\psi}_{L^2_x}^2\norm{\nabla u}_{L^2_x}^4\norm{\Delta\psi}_{L^2_x}^4 + \frac{\varepsilon}{C\Lambda}\norm{\nabla(B\psi)}_{L^2_x}^2
    \end{equation*}\medskip
\end{enumerate}

Substituting the above estimates into \eqref{NSE higher order first step},

\begin{multline} \label{NSE higher order second step}
    \frac{\nu}{2}\frac{d}{dt}\norm{\nabla u}_{L^2_x}^2 + \frac{1}{2}\int_{\Omega}\rho\abs{\partial_t u}^2 \lesssim M' \norm{\nabla u}_{L^2_x}^{\frac{5}{2}}\norm{\Delta u}_{L^2_x}^{\frac{3}{2}} + \frac{\Lambda}{C}\norm{\nabla(B\psi)}_{L^2_x}^2 \\ + \frac{\Lambda^3}{\varepsilon^2}\norm{B\psi}_{L^2_x}^2\left( 1 + \norm{\nabla u}_{L^2_x}^4 \right)\norm{\Delta\psi}_{L^2_x}^4
\end{multline}\medskip

Having obtained equations for the rate of change of $\norm{\nabla u}_{L^2_x}$ and $\norm{\Delta\psi}_{L^2_x}$, what remains is to consider the ``higher-order dissipative" term $\norm{\Delta u}_{L^2_x}^2$ in these estimates. Having this on the LHS will allow us to absorb such terms from the RHS, and set up the required Gr\"onwall inequality. (Note that the higer-order dissipative term for the wavefunction is $\norm{\nabla(B\psi)}_{L^2_x}^2$, which is present in \eqref{delta psi equation}.) To this end, we multiply \eqref{NSE'} by $-\theta\Delta u$ and integrate over the domain, where $\theta$ is a positive constant whose value will be fixed shortly.

\begin{multline} \label{NSE higher order third step}
    \theta\nu\norm{\Delta u}_{L^2_x}^2 = \theta\int_{\Omega}\rho\partial_t u\cdot\Delta u + \theta\int_{\Omega}\rho u\cdot\nabla u\cdot\Delta u + 2\Lambda\theta\int_{\Omega}\Im(\nabla\Bar{\psi}B\psi)\cdot\Delta u \\ + 2\Lambda\theta\int_{\Omega}u\Re(\Bar{\psi}B\psi)\cdot\Delta u
\end{multline}\medskip

Once again, we estimate each term on the RHS.

\begin{enumerate}
    \item \begin{equation*}
        \theta\int_{\Omega}\rho\partial_t u\cdot\Delta u \le \frac{\theta\nu}{8}\norm{\Delta u}_{L^2_x}^2 + \frac{2\theta M'}{\nu}\int_{\Omega}\rho\abs{\partial_t u}^2
    \end{equation*}\medskip
    
    \item \begin{equation*}
        \theta\int_{\Omega}\rho u\cdot\nabla u\cdot\Delta u \le \frac{\theta\nu}{8}\norm{\Delta u}_{L^2_x}^2 + \frac{2\theta {M'}^2}{\nu}\int_{\Omega}\abs{u}^2\abs{\nabla u}^2
    \end{equation*}\medskip
    
    The last term is manipulated just as in \eqref{example of bounding u}. \medskip

    \item \begin{equation*}
        2\Lambda\theta\int_{\Omega}\Im(\nabla\Bar{\psi}B\psi)\cdot\Delta u \le \frac{\theta\nu}{8}\norm{\Delta u}_{L^2_x}^2 + \frac{8\Lambda^2\theta}{\nu}\norm{\nabla\psi}_{L^2_x}^2\norm{B\psi}_{L^2_x}^2
    \end{equation*}\medskip
    
    \item \begin{equation*}
        2\Lambda\theta\int_{\Omega}u\Re(\Bar{\psi}B\psi)\cdot\Delta u \le \frac{\theta\nu}{8}\norm{\Delta u}_{L^2_x}^2 + \frac{8\Lambda^2\theta}{\nu}\norm{u}_{L^6_x}^2\norm{\psi}_{L^{\infty}_x}^2\norm{B\psi}_{L^2_x}^2
    \end{equation*}\medskip
\end{enumerate}

Thus, \eqref{NSE higher order third step} becomes ($c$ is a positive constant that depends only on $\Omega$):

\begin{multline} \label{NSE higher order fourth step}
    \frac{\theta\nu}{2}\norm{\Delta u}_{L^2_x}^2 \le \frac{2\theta M'}{\nu}\int_{\Omega}\rho\abs{\partial_t u}^2 + \frac{2\theta {M'}^2 c}{\nu}\norm{\nabla u}_{L^2_x}^{\frac{5}{2}}\norm{\Delta u}_{L^2_x}^{\frac{3}{2}} + \frac{\Lambda}{2}\norm{\nabla(B\psi)}_{L^2_x}^2 \\ + \frac{\Lambda^3\theta^2 c}{\nu^2}\norm{B\psi}_{L^2_x}^2\left(1+\norm{\nabla u}_{L^2_x}^4 \right)\norm{\Delta\psi}_{L^2_x}^4
\end{multline}\medskip

We now add \eqref{NSE higher order second step} and \eqref{NSE higher order fourth step}. Choosing $\theta = \frac{\nu}{8M'}$, and extracting $\norm{\Delta u}_{L^2_x}^2$ with sufficiently small coefficients, we absorb into the corresponding term on the LHS. Finally, what remains is:

\begin{multline} \label{nabla u equation}
    \frac{\nu}{2}\frac{d}{dt}\norm{\nabla u}_{L^2_x}^2 + \frac{1}{4}\norm{\sqrt{\rho}\partial_t u}_{L^2_x}^2 + \frac{\nu^2}{M'}\norm{\Delta u}_{L^2_x}^2 \lesssim \frac{{M'}^7}{\nu^6}\norm{\nabla u}_{L^2_x}^{10} + \frac{\Lambda}{2}\norm{\nabla(B\psi)}_{L^2_x}^2 \\ + \frac{\Lambda^3}{\varepsilon^2}\norm{B\psi}_{L^2_x}^2\left( 1+\norm{\nabla u}_{L^2_x}^4 \right)\norm{\Delta\psi}_{L^2_x}^4
\end{multline}\medskip

\subsubsection{The Gr\"onwall inequality step} \label{gronwall inequality step for higher order estimate}

Having derived the equations for the higher-order norms of $u$ and $\psi$, and also accounted for the relevant dissipative terms, we now add \eqref{delta psi equation} and \eqref{nabla u equation}:

\begin{multline} \label{delta psi nabla u equations combined}
    \frac{d}{dt}\left[ \norm{\Delta\psi}_{L^2_x}^2 + \nu\norm{\nabla u}_{L^2_x}^2 \right] + \Lambda\norm{\nabla(B\psi)}_{L^2_x}^2 + \norm{\sqrt{\rho}\partial_t u}_{L^2_x}^2 + \frac{\nu^2}{M'}\norm{\Delta u}_{L^2_x}^2 \\ \lesssim \left[ \frac{\Lambda^2 M'}{\nu^2}\norm{\nabla u}_{L^2_x}^2\norm{\Delta\psi}_{L^2_x}^4 + \frac{\Lambda^4 {M'}^{3}}{\nu^6}\norm{\nabla u}_{L^2_x}^2\norm{\Delta\psi}_{L^2_x}^8  
    + \frac{\Lambda^2 M'}{\nu^2}\norm{\nabla u}_{L^2_x}^6 \norm{\Delta\psi}_{L^2_x}^4 \right. \\  + \Lambda\norm{\nabla u}_{L^2_x}^4\norm{\Delta\psi}_{L^2_x}^2
    + \gamma E_0^4 \norm{\Delta\psi}_{L^2_x}^2 \Bigg] \\ + \left[ \frac{{M'}^7}{\nu^6}\norm{\nabla u}_{L^2_x}^{10} + \frac{\Lambda^3}{\varepsilon^2}\norm{B\psi}_{L^2_x}^2\left( 1+\norm{\nabla u}_{L^2_x}^4 \right)\norm{\Delta\psi}_{L^2_x}^4 \right]
\end{multline}\medskip

Denoting 
\begin{gather*}
    X = 1 + \norm{\Delta\psi}_{L^2_x}^2 + \nu\norm{\nabla u}_{L^2_x}^2 \\
    Y = \Lambda\norm{\nabla(B\psi)}_{L^2_x}^2 + \norm{\sqrt{\rho}\partial_t u}_{L^2_x}^2 + \frac{\nu^2}{M'}\norm{\Delta u}_{L^2_x}^2
\end{gather*}\medskip

\noindent we can rewrite \eqref{delta psi nabla u equations combined} as follows:

\begin{multline} \label{gronwall expanded form}
    \frac{dX}{dt} + Y \lesssim \left[ \frac{\Lambda^2 M'}{\nu^3}X^3 + \frac{\Lambda^4 {M'}^{3}}{\nu^7}X^5
    + \frac{\Lambda^2 M'}{\nu^5}X^5 + \frac{\Lambda}{\nu^2}X^3
    + \gamma E_0^4 X \right] \\ + \left[ \frac{{M'}^7}{\nu^{11}}X^5 + \frac{\Lambda^3}{\varepsilon^2\nu^2}\norm{B\psi}_{L^2_x}^2\left( \nu^2+X^2 \right)X^2 \right]
\end{multline}\medskip

Since $X\ge 1$, we set: 

\begin{equation} \label{huge gronwall constant}
    C = \max{\left\{ \frac{\Lambda^2 M'}{\nu^3} + \frac{\Lambda^4 {M'}^{3}}{\nu^7} +
    \frac{\Lambda^2 M'}{\nu^5} + \frac{\Lambda}{\nu^2} +
    \gamma E_0^4 + \frac{{M'}^7}{\nu^{11}} \ , \ \frac{\Lambda^3}{\varepsilon^2\nu^2}(\nu^2 + 1) \right\}}
\end{equation}\medskip

\noindent and rewrite \eqref{gronwall expanded form} in a much simpler form.

\begin{equation} \label{simple gronwall inequality}
    \frac{dX}{dt} + Y \le C\left( 1 + \norm{B\psi}_{L^2_x}^2 \right)X^5
\end{equation}\medskip

We begin by dropping the non-negative $Y$. Since $B\psi\in L^2_{[0,T]}$ (the $T$ is from Definition \ref{local existence time definition}), we can easily integrate to arrive at:

\begin{equation} \label{bound on X(t)}
    X(t) \le \frac{X_0}{\left[ 1 - C X_0^4 \left( T' + \norm{B\psi}_{L^2_{[0,T']} L^2_x}^2 \right) \right]^{\frac{1}{4}}} \qquad a.e. \ t\in [0,T']
\end{equation}\medskip

\noindent where $X_0 = X(0) = 1 + \nu\norm{\nabla u_0}_{L^2_x}^2 + \norm{\Delta\psi_0}_{L^2_x}^2 < \infty$, due to the regularity of the initial data. In addition, $T'$ is some time that is less than or equal to $T$ from \eqref{abstract definition existence time}. Of course, the RHS of \eqref{bound on X(t)} makes sense only if $T'$ is such that the denominator doesn't become non-positive. \medskip

\begin{mydef} [Updated local existence time] \label{updated local existence time}
    Consider $T'$ such that
    \begin{equation} \label{choice of T from gronwall}
        T' + \norm{B\psi}_{L^2_{[0,T']} L^2_x}^2 \le \frac{15}{16C X_0^4}
    \end{equation}\medskip
    
    Define the local existence time $= \min{\{ T' \text{ from } \eqref{choice of T from gronwall} \ , \ T \text{ from } \eqref{abstract definition existence time} \}}$. (Favoring clarity in subscripts, we will abuse notation and denote this (updated) local existence time as $T$).
\end{mydef}\medskip

With this choice of local existence time $T$, we observe that

\begin{equation} \label{X(t)<2X_0}
    X(t) \le 2 X_0 \quad a.e. \ t \in [0,T]
\end{equation}\medskip

Returning to \eqref{simple gronwall inequality}, for $t\in [0,T]$,

\begin{equation*}
    \frac{dX}{dt} + Y \le C X^5 \left( 1 + \norm{B\psi}_{L^2_x}^2 \right) \le 32C X_0^5 \left( 1 + \norm{B\psi}_{L^2_x}^2 \right)
\end{equation*}\medskip

Integrating from $0$ to $t\in [0,T]$,

\begin{gather}
    X(t) - X_0 + \int_0^t Y(\tau) \ d\tau \le 32C X_0^5 \left( \frac{15}{16C X_0^4} \right) = 30 X_0 \notag \\
    \Rightarrow \int_0^T Y(\tau) \ d\tau \le 31 X_0 \label{Y is integrable}
\end{gather}\medskip

Thus, \eqref{X(t)<2X_0} and \eqref{Y is integrable} imply the following a priori estimates:

\begin{gather} \label{higher order inclusion implication}
    \begin{split}
    \begin{gathered}
        u \in L^{\infty}_{[0,T]}H^1_{d,x} \cap L^2_{[0,T]}H^2_{d,x} \ , \ \partial_t u \in L^2_{[0,T]}L^2_{d,x} \\
        \psi \in L^{\infty}_{[0,T]}H^2_{0,x} \ , \ B\psi \in L^2_{[0,T]}H^1_x
    \end{gathered}
    \end{split}
\end{gather}\medskip

\begin{rem} \label{drop rho to get bound on partial_t u}
    Note that the actual estimate was $\sqrt{\rho} \partial_t u \in L^2_{[0,T]}L^2_x$, but in the next (even higher) a priori estimate, we will need to bound $u$ in $L^{\infty}_x$, for which we require a bound on $\partial_t u$ in $L^2_{[0,T]}L^2_x$. This is the reason for working with density that is bounded below, and is one of the differences between this work and \cite{Kim1987WeakDensity}. 
\end{rem}\medskip

\subsubsection{What does $B\psi\in L^2_{[0,T]}H^1_x$ imply for $\psi$?} \label{B psi is in H^1}

Just as in Section \ref{B psi is in L^2}, we can confirm our intuition that $B\psi \in H^1_x$ is equivalent to $\psi\in H^3_x$ (not pointwise in time, but rather in $L^2$), since $B$ is a second-order differential operator. However, this time, we use the updated a priori estimates from \eqref{higher order inclusion implication}. \\

From the definition of $B$, and using incompressibility,
\begin{align*}
    -\frac{1}{2}\nabla\Delta\psi &= \nabla(B\psi) - i\nabla\nabla\cdot(u\psi) - \nabla\left(\frac{1}{2}\abs{u}^2\psi\right) - \mu\nabla(\abs{\psi}^2\psi) \\
    \Rightarrow \norm{\psi}_{H^3_x} &\lesssim \norm{\nabla(B\psi)}_{L^2_x} + \norm{u\psi}_{H^2_x} + \norm{\nabla\left(\frac{1}{2}\abs{u}^2\psi\right)}_{L^2_x} + \mu\norm{\nabla(\abs{\psi}^2\psi)}_{L^2_x}
\end{align*}\medskip

The Sobolev spaces $H^s(\Omega)$ form an algebra for $s>\frac{3}{2}$, so this allows us to easily estimate the second term on the RHS. The third and fourth terms are managed the same way as was done in going from \eqref{schrodinger equation higher order third step} to \eqref{schrodinger equation higher order fourth step}. In all, we arrive at:

\begin{multline*}
    \norm{\psi}_{L^2_{[0,T]}H^3_x} \lesssim \norm{\nabla(B\psi)}_{L^2_{[0,T]}L^2_x} + \norm{u}_{L^2_{[0,T]}H^2_x}\norm{\psi}_{L^{\infty}_{[0,T]}H^2_x} \\ + \norm{u}_{L^{\infty}_{[0,T]}H^1_x}\norm{u}_{L^2_{[0,T]}H^2_x}\norm{\psi}_{L^{\infty}_{[0,T]}H^2_x} + \mu T^{\frac{1}{2}}\norm{\psi}_{L^{\infty}_{[0,T]}H^1_x}^2\norm{\psi}_{L^{\infty}_{[0,T]}H^2_x}
\end{multline*}\medskip

Given the a priori estimates in \eqref{higher order inclusion implication}, we see that $\psi\in L^2_{[0,T]}H^3_{0,x}$ if $B\psi\in L^2_{[0,T]}H^1_x$. We can now (slightly) modify the estimate for the time-derivative at the end of Section \ref{B psi is in L^2}, in particular, its time regularity can be increased to $L^2$. 

\begin{equation} \label{partial_t psi is in L^2_t L^2_x}
    \norm{\partial_t \psi}_{L^2_{[0,T]}L^2_x} \lesssim \norm{B\psi}_{L^2_{[0,T]}L^2_x} + T^{\frac{1}{2}}\norm{\psi}_{L^{\infty}_{[0,T]}H^2_x} + \mu T^{\frac{1}{2}} \norm{\psi}_{L^{\infty}_{[0,T]}H^1_x}^3 < \infty
\end{equation}\medskip

Once again, using Lemma \ref{aubin-lions}, we can infer the strong convergence of a subsequence of wavefunctions in $L^2_{[0,T]}H^{[0,3)}_{0,x}$.\bigskip

\subsection{The highest-order a priori estimate for $\psi$} \label{the highest-order a priori estimate for psi}

From the previous analysis, we have obtained $B\psi\in L^2_{[0,T]}H^1_x$. However, as pointed out in the discussion following Definition \ref{local existence time definition}, we seek $B\psi\in L^2_{[0,T]}L^{\infty}_x$. Taking advantage of the embedding $H^{\frac{3}{2}+\delta}(\Omega)\subset L^{\infty}(\Omega)$, we will now derive an even higher order a priori estimate (only for $\psi$). \\

We act on \eqref{NLS} by $(-\Delta)^s$, for $s\in \left( \frac{5}{4},\frac{3}{2} \right)$.

\begin{equation}
    \partial_t (-\Delta)^s\psi + \Lambda(-\Delta)^s(B\psi) = -\frac{1}{2i}\Delta(-\Delta)^s\psi + \frac{\mu}{i}(-\Delta)^s(\abs{\psi}^2\psi)
\end{equation}\medskip

As will be shown in Section \ref{local existence proof}, the semi-Galerkin scheme for the wavefunction is set up using the eigenfunctions of the penta-Laplacian in Section \ref{Constructing the wavefunction and the Dirichlet penta-Laplacian}, which have vanishing derivatives up to the 4th order on the boundary. Since the eigenfunctions are also smooth, we see that they belong to $H^5_0(\Omega)$. Just as in Sections \ref{energy estimate} and \ref{NLS higher order estimate}, we multiply by $(-\Delta)^s\Bar{\psi}$, take the real part and integrate over $\Omega$.\medskip

\begin{enumerate}
    \item 
    \begin{align*}
        \Re\int_{\Omega} (-\Delta)^s \Bar{\psi} (-\Delta)^s (B\psi) &= \Re\left\langle (-\Delta)^s \psi, (-\Delta)^s (B\psi) \right\rangle \\
        &= \Re\left\langle (-\Delta)^{s+\frac{1}{2}} \psi, (-\Delta)^{s-\frac{1}{2}} (B\psi) \right\rangle \\
        &= \Re\left\langle (-\Delta)^{s-\frac{1}{2}} (-\Delta)\psi, (-\Delta)^{s-\frac{1}{2}} (B\psi) \right\rangle
    \end{align*}\medskip
    
    Here, we have used Lemma \ref{fractional powers self-adjoint property}. But, this requires that $B\psi\in D\left((-\Delta)^s\right) = D\left(\left((-\Delta)^5\right)^{\frac{s}{5}}\right) = H^{2s}_0 = H^{\frac{5}{2}+\delta}_0$. (The characterization in terms of Sobolev spaces follows by a proof very similar to Propositions \ref{domain of L_2^0.5} and \ref{fractional bi-laplacian sobolev spaces}.)
    
    Since $B$ has a second derivative term, this is equivalent to saying $\psi\in H^{\frac{9}{2}+\delta}_0$, thus justifying our choice of boundary conditions (the extra vanishing derivative) for the basis functions of the semi-Galerkin approximation. \medskip
    
    \item 
    \begin{align*}
        \Im\int_{\Omega} (-\Delta)^s \Bar{\psi} (-\Delta)^{s+1} \psi &= \Im \left\langle (-\Delta)^s \psi, (-\Delta)^{s+1}\psi \right\rangle \\ 
        &= \Im \left\langle (-\Delta)^{s+\frac{1}{2}}\psi, (-\Delta)^{s+\frac{1}{2}}\psi \right\rangle = 0
    \end{align*}\medskip
    
    Just as in the previous term, this requires $\psi\in H^{2s+2}_0 = H^{\frac{9}{2}+\delta}_0$, which is once again satisfied given the choice of eigenfunctions in the semi-Galerkin scheme.\medskip

    \item 
    \begin{align*}
        \Im\int_{\Omega} (-\Delta)^s \Bar{\psi} (-\Delta)^{s} (\abs{\psi}^2\psi) &= \left\langle (-\Delta)^s\psi,(-\Delta)^{s} (\abs{\psi}^2\psi) \right\rangle \\
        &= \left\langle (-\Delta)^{s+\frac{1}{2}}\psi,(-\Delta)^{s-\frac{1}{2}} (\abs{\psi}^2\psi) \right\rangle \\ 
        &= \left\langle (-\Delta)^{s-\frac{1}{2}}(-\Delta)\psi,(-\Delta)^{s-\frac{1}{2}} (\abs{\psi}^2\psi) \right\rangle
    \end{align*}\medskip

    This calculation is also justified in a manner similar to the first two terms. \bigskip
    
\end{enumerate}

In the first and third terms above, we trade the negative Laplacian for the $B$ operator as done in the a priori estimates up to this point, and arrive at the following inequality\footnote{Recall that $\gamma = \mu^2 \left( \Lambda + \frac{1}{\Lambda} \right)$.}. 

\begin{multline} \label{highest order energy inequality}
    \frac{d}{dt}\norm{(-\Delta)^s\psi}_{L^2_x}^2 + \Lambda\norm{(-\Delta)^{s-\frac{1}{2}}(B\psi)}_{L^2_x}^2 \\ \lesssim \Lambda\norm{(-\Delta)^{s-\frac{1}{2}}(u\cdot\nabla\psi)}_{L^2_x}^2 + \Lambda\norm{(-\Delta)^{s-\frac{1}{2}}(\abs{u}^2\psi)}_{L^2_x}^2 + \gamma\norm{(-\Delta)^{s-\frac{1}{2}}(\abs{\psi}^2\psi)}_{L^2_x}^2
\end{multline}\medskip

Note that choosing $s=0,\frac{1}{2},1$ leads (respectively) to the mass, energy and higher-order energy estimates from before. Now, we will estimate each of the terms in the RHS of \eqref{highest order energy inequality}. Since $s-\frac{1}{2}\in \left(\frac{3}{4},1 \right)$, we can use Lemma \ref{poincare inequality fractional} to replace all the terms on the RHS (and the second term on the LHS) by appropriate Sobolev space norms.

\begin{equation} \label{highest order energy sobolev norms}
    \frac{d}{dt}\norm{(-\Delta)^s\psi}_{L^2_x}^2 + \Lambda\norm{B\psi}_{H^{2s-1}_x}^2 \lesssim \Lambda\norm{u\cdot\nabla\psi}_{H^{2s-1}_x}^2 + \Lambda\norm{\abs{u}^2\psi}_{H^{2s-1}_x}^2 + \gamma\norm{\abs{\psi}^2\psi}_{H^{2s-1}_x}^2
\end{equation}\medskip

We set $2s-1 = \frac{3}{2}+\delta$ for some $0<\delta<\frac{1}{2}$. Using the algebra property\footnote{$\norm{fg}_{H^r}\lesssim \norm{f}_{H^r}\norm{g}_{H^r}$ for $r>\frac{d}{2}$ in $d$ dimensions.} of Sobolev norms, along with Lemma \ref{poincare inequality fractional}, we estimate the RHS of \eqref{highest order energy sobolev norms}:\medskip

\begin{enumerate}
    \item $$\norm{u\cdot\nabla\psi}_{H^{2s-1}_x}^2 \lesssim \norm{u}_{H^{2s-1}_x}^2 \norm{\psi}_{H^{2s-1}_x}^2 \lesssim \norm{u}_{H^{2}_x}^2 \norm{(-\Delta)^s\psi}_{L^2_x}^2$$\medskip
    
    Since we already have $u\in L^2_{[0,T]}H^2_{d,x}$, this term is amenable to Gr\"onwall's inequality. \medskip
    
    \item $$\norm{\abs{u}^2\psi}_{H^{2s-1}_x}^2 \lesssim \norm{u}_{H^{\frac{3}{2}+\delta}_x}^4 \norm{\psi}_{H^{2s-1}_x}^2 \lesssim \norm{u}_{H^{\frac{3}{2}+\delta}_x}^4 \norm{\psi}_{H^{2s}_x}^2$$\medskip
    
    We know that $u\in L^2_{[0,T]}H^2_{d,x}$ and $\partial_t u\in L^2_{[0,T]}L^2_{x} \subset L^2_{[0,T]} \left(H^{2}_{0,x}\right)^* \subset L^2_{[0,T]} \left(H^2_{d,x}\right)^*$. So, we can use Lemma \ref{lions-magenes} to conclude that:
    
    \begin{equation} \label{u bounded in L^infty_t}
        \norm{u}_{H^{\frac{3}{2}+\delta}_x}^2(t) \le \norm{u_0}_{H^{\frac{3}{2}+\delta}_x}^2 + 2\norm{u}_{L^2_{[0,T]}H^2_x}\norm{\partial_t u}_{L^2_{[0,T]}H^{-2}_x} \lesssim \norm{u_0}_{H^{\frac{3}{2}+\delta}_x}^2 + \frac{1}{\nu}\sqrt{\frac{M'}{\varepsilon}}X_0
    \end{equation}\medskip
    
    \noindent where $X_0$ is defined immediately following \eqref{bound on X(t)}. For brevity, define 
    
    $$E_1 = \norm{u_0}_{H^{\frac{3}{2}+\delta}_x}^2 + \norm{\psi_0}_{H^{\frac{5}{2}+\delta}_x}^2$$\medskip
    
    Thus,
    
    $$\norm{\abs{u}^2\psi}_{H^{2s-1}_x}^2 \lesssim \left(\frac{M'}{\nu^2\varepsilon}X_0^2+E_1^2\right) \norm{(-\Delta)^s\psi}_{L^2_x}^2$$\medskip
    
    \item $$\norm{\abs{\psi}^2\psi}_{H^{2s-1}_x} \lesssim \norm{\psi}_{H^2_x}^6 \lesssim \norm{\psi}_{H^2_x}^4 \norm{\psi}_{H^{2s}_x}^2 \lesssim X_0^2 \norm{(-\Delta)^s\psi}_{L^2_x}^2$$
\end{enumerate}

\bigskip

From the above estimates, we have

\begin{equation} \label{highest order energy RHS estimated}
    \frac{d}{dt}\norm{(-\Delta)^s\psi}_{L^2_x}^2 + \Lambda\norm{B\psi}_{H^{2s-1}_x}^2 \lesssim \left[ \Lambda\norm{u}_{H^2_x}^2 + \left(\Lambda \frac{M'}{\nu^2\varepsilon} + \gamma\right)X_0^2 + \Lambda E_1^2 \right] \norm{(-\Delta)^s\psi}_{L^2_x}^2
\end{equation}\medskip

Using Gr\"onwall's inequality, for all $t\in [0,T]$:

\begin{equation} \label{Dirichlet form of gronwall highest order estimate}
    \norm{(-\Delta)^s\psi}_{L^2_x}^2(t) \lesssim \norm{(-\Delta)^s\psi_0}_{L^2_x}^2 e^{c\left[ \Lambda\norm{u}_{L^2_{[0,T]}H^2_x}^2 + \left(\Lambda \frac{M'}{\nu^2\varepsilon} + \gamma\right)X_0^2 T + \Lambda E_1^2 T \right]}
\end{equation}\medskip

Using Lemma \ref{poincare inequality fractional}, we can simplify \eqref{Dirichlet form of gronwall highest order estimate} to:

\begin{equation} \label{bound on psi in H^2s}
    \norm{\psi}_{H^{2s}_x}^2 (t) \lesssim \norm{\psi_0}_{H^{2s}_x}^2 e^{c Q_T} = \norm{\psi_0}_{H^{\frac{5}{2}+\delta}_x}^2 e^{c Q_T}
\end{equation}\medskip

\noindent where 

\begin{equation} \label{defining Q_T}
    Q_T = \left[ \Lambda\frac{M'}{\nu^2}X_0 + \left(\Lambda \frac{M'}{\nu^2\varepsilon} + \gamma\right)X_0^2 T + \Lambda E_1^2 T \right]
\end{equation}\medskip

More importantly, we get the sought-after ``dissipation bound":

\begin{equation} \label{Bpsi sought-after bound}
    \norm{B\psi}_{L^2_{[0,T]}H^{2s-1}_x} \lesssim \Lambda^{-\frac{1}{2}} (Q_T^{\frac{1}{2}} + 1) e^{c Q_T} \norm{\psi_0}_{H^{2s}_x} = \Lambda^{-\frac{1}{2}} (Q_T^{\frac{1}{2}} + 1) e^{c Q_T} \norm{\psi_0}_{H^{\frac{5}{2}+\delta}_x}
\end{equation}\medskip

Since $2s-1 = \frac{3}{2}+\delta$, the second embedding in Lemma \ref{sobolev embeddings} allows us to conclude that $B\psi$ is bounded in $L^2_{[0,T]}L^{\infty}_x$. This is the required estimate to ensure that the density remains bounded below. \medskip

\begin{rem} \label{why we need additional BC for velocity}
    The estimate in \eqref{u bounded in L^infty_t} was not needed in \cite{Kim1987WeakDensity}. Indeed, in that work, it was sufficient to interpolate between $L^{\infty}_t H^1_x$ and $L^2_t H^2_x$ to get a bound in $L^{\frac{4}{1+2\delta}}_t H^{\frac{3}{2}+\delta}_x$. In our preceding analysis, however, such an interpolation would not work due to the high exponent of $u$: $2$ from the $\abs{u}^2$, and another factor of $2$ from the square on the outside. Now, in trying to use Lemma \ref{lions-magenes} (Lions-Magenes), it is necessary that the spaces that $u$ and $\partial_t u$ live in be dual to each other. Therefore, this forces us to have $u\in H^2_{0,x}$ (as opposed to simply $H^2_x$), which adds the extra boundary condition of a vanishing gradient.
\end{rem}\medskip

\subsubsection{What does ``$B\psi \in L^2_{[0,T]}H^{2s-1}_x$" imply for $\psi$?}

Once again (and for the last time!), we repeat the procedure in sections \ref{B psi is in L^2} and \ref{B psi is in H^1} to deduce what regularity on the wavefunction is imparted by the above a priori estimate.

\begin{gather*}
    -\frac{1}{2}\Delta\psi = B\psi - iu\cdot\nabla\psi - \frac{1}{2}\abs{u}^2\psi - \mu\abs{\psi}^2\psi \\
    \Rightarrow \norm{\Delta\psi}_{H^{2s-1}_x} \lesssim \norm{(-\Delta)^{s-\frac{1}{2}}B\psi}_{L^2_x} + \norm{u\cdot\nabla\psi}_{H^{2s-1}_x} + \norm{\abs{u}^2\psi}_{H^{2s-1}_x} + \norm{\abs{\psi}^2\psi}_{H^{2s-1}_x}
\end{gather*}\medskip

Performing an $L^2$ integration over $[0,T]$:

\begin{multline}
    \norm{\psi}_{L^2_{[0,T]}H^{2s+1}_x} \lesssim \norm{B\psi}_{L^2_{[0,T]}H^{2s-1}_x} + T^{\frac{1}{2}}\norm{u}_{L^{\infty}_{[0,T]}H^{2s-1}_x}\norm{\psi}_{L^{\infty}_{[0,T]}H^{2s}_x} \\ + T^{\frac{1}{2}}\norm{u}^2_{L^{\infty}_{[0,T]}H^{2s-1}_x}\norm{\psi}_{L^{\infty}_{[0,T]}H^{2s-1}_x} + T^{\frac{1}{2}}\norm{\psi}^3_{L^{\infty}_{[0,T]}H^{2s-1}_x}
\end{multline}\medskip

Each of the terms on the RHS is finite (by the preceding a priori estimates). Thus, we have $\psi\in L^2_{[0,T]}H^{\frac{7}{2}+\delta}_{0,x}$. By applying Lemma \ref{aubin-lions}, we can once again conclude strong convergence of a subsequence in $L^2_{[0,T]}H^{[0,\frac{7}{2}+\delta)}_{0,x}$.\medskip

In summary, we have the following bounds:

\begin{gather} \label{final a priori bounds}
    \begin{split}
    \begin{gathered}
        u \in L^{\infty}_{[0,T]}H^{\frac{3}{2}+\delta}_{d,x} \cap L^2_{[0,T]}H^2_{d,x} \ , \ \partial_t u \in L^2_{[0,T]}L^2_{d,x} \\
        \psi \in L^{\infty}_{[0,T]}H^{\frac{5}{2}+\delta}_{0,x} \cap L^2_{[0,T]}H^{\frac{7}{2}+\delta}_{0,x} \ , \ \partial_t \psi \in L^2_{[0,T]}L^2_{x} \\
        \ B\psi \in L^2_{[0,T]}H^{\frac{3}{2}+\delta}_{x} \\
        \varepsilon \le \rho \le M' = M+m-\varepsilon \ \forall \ (t,x) \in [0,T]\times\Omega
    \end{gathered}
    \end{split}
    \end{gather}

\bigskip

\section{Local existence of weak solutions (Proof of Theorem \ref{local existence})} \label{local existence proof}

Having derived the required a priori estimates, we are now ready to establish weak solutions to the Pitaevskii model. In this section, we will use a semi-Galerkin scheme to prove local existence of solutions for a truncated form of the governing equations, before passing to the limit to arrive at the weak solutions.  \medskip

\subsection{Constructing the semi-Galerkin scheme}
Due to the boundary conditions imposed in the Pitaevskii model, the wavefunction is constructed using eigenfunctions of the fifth power of the negative Dirichlet Laplacian operator as basis, while the velocity is built from eigenfunctions of the Leray-projected Dirichlet bi-Laplacian operator. We will now describe both these operators and their properties. \medskip

\subsubsection{The truncated wavefunction} \label{Constructing the wavefunction and the Dirichlet penta-Laplacian}

Consider the following boundary value problem (where $n$ is the outward normal at the boundary, and $f\in L^2_x$):

\begin{equation} \label{dirichlet penta-laplacian}
    \begin{gathered}
        (-\Delta)^5 \xi = f \qquad \text{ in } \Omega \\
        \xi = \frac{\partial \xi}{\partial n} = \frac{\partial^2 \xi}{\partial n^2} = \frac{\partial^3 \xi}{\partial n^3} = \frac{\partial^4 \xi}{\partial n^4} = 0 \qquad \text{ on } \partial\Omega
    \end{gathered}
\end{equation}\medskip

The operator $\mathcal{L}_5 := (-\Delta)^5$, henceforth called the \textit{(Dirichlet) penta-Laplacian}, is defined on the space $D(\mathcal{L}_5) = H^{10}\cap H^5_0$ (see Corollary 2.21 in \cite{Gazzola2010PolyharmonicDomains}). It has a discrete set of strictly positive and non-decreasing eigenvalues $(0<\beta_1\le\beta_2\le\beta_3\dots\rightarrow\infty)$, and the corresponding eigenfunctions $(\{b_j\}\in C^{\infty}(\Bar{\Omega}))$ can be chosen to be orthonormal in the $L^2_x$ norm and orthogonal in the $H^5_x$ norm. The imposed boundary conditions in the Pitaevskii model indicate that this is the right basis to consider for constructing the wavefunction\footnote{Compared to the Pitaevskii model, there is one extra vanishing derivative on the boundary for these eigenfunctions. This is to ensure some that the a priori estimates work out. See the handling of the $B\psi$ term in Section \ref{the highest-order a priori estimate for psi}.}. \medskip

For $N\in\N$, we define the truncated wavefunction as:

\begin{equation} \label{truncated wavefunction definition}
    \psi^N (t,x) = \sum_{k=1}^N d^N_k(t) b_k(x)
\end{equation}\medskip

\noindent where $d^N_k(t)\in\C$. \medskip

\subsubsection{The truncated velocity}

We begin by considering the following vector-valued boundary value problem (where $n$ is the outward normal vector on the boundary, and $F\in L^2_{d,x}$ is vector-valued):

\begin{equation} \label{dirichlet bi-Stokes}
    \begin{gathered}
        (-\Delta)^2 v + \nabla p = F \qquad \text{ in } \Omega \\
        \nabla\cdot v = 0 \qquad \text{ in } \Omega \\
        v = \frac{\partial v}{\partial n} = 0 \qquad \text{ on } \partial\Omega
    \end{gathered}
\end{equation}\medskip

When the domain $\Omega$ is bounded and sufficiently smooth, this problem has (see Theorem 2.1 in \cite{Manouzi2005AHyper-dissipation}) a unique solution pair: $v\in H^4_x$ and $p\in L^2_x\setminus \R$. These functions also satisfy the following elliptic regularity estimate:

\begin{equation} \label{elliptic regularity for bi-stokes problem}
    \norm{v}_{H^4} + \norm{p}_{L^2} \lesssim \norm{F}_{L^2}
\end{equation}\medskip

\noindent implying that the map $F\mapsto v$ is bounded. This map is the inverse of the bi-Stokes operator, to be introduced next.\medskip

We will denote by $\Leray$ the Leray projector (see chapter 2 in \cite{Robinson2016TheEquations}, for instance), which maps any Hilbert space $H$ to a subspace $H_d$ consisting of only divergence-free functions. Leray-projecting the first equation of \eqref{dirichlet bi-Stokes} eliminates the gradient term, giving (recall that $F$ is already assumed to be divergence-free):

\begin{equation}
    \begin{gathered}
        \Leray (-\Delta)^2 v = F \qquad \text{ in } \Omega \\
        v = \frac{\partial v}{\partial n} = 0 \qquad \text{ on } \partial\Omega
    \end{gathered}
\end{equation}\medskip

We will refer to $\mathfrak{S}_2 := \Leray (-\Delta)^2$ as the \textit{(Dirichlet) bi-Stokes} operator, defined on the space $D(\mathfrak{S}_2) = H^4\cap H^2_d$. Just as with the Dirichlet penta-Laplacian above, it can be easily shown that the bi-Stokes operator has a discrete set of strictly positive and non-decreasing eigenvalues $(0<\alpha_1\le\alpha_2\le\alpha_3\dots\rightarrow\infty)$. The corresponding divergence-free, vector-valued eigenfunctions $(\{a_j\}\in C^{\infty}(\Bar{\Omega}))$ can be chosen to be orthonormal in the $L^2_x$ norm and orthogonal in the $H^2_x$ norm. \medskip

For $N\in\N$, we define the truncated velocity as:

\begin{equation} \label{truncated velocity definition}
    u^N (t,x) = \sum_{k=1}^N c^N_k(t) a_k(x)
\end{equation}\medskip

\noindent where $c^N_k(t)\in\R$. \medskip

\subsection{The initial conditions} \label{the initial conditions}

\subsubsection{The initial wavefunction and initial velocity}

In this section, we will discuss our choice of truncated initial conditions for each of the fields (wavefunction, velocity and density). We begin by defining $P^N$ (respectively, $Q^N$) to be the projections onto the space spanned by the first $N$ eigenfunctions of $\mathfrak{S}_2$ (respectively, $\mathcal{L}_5$). Then, we truncate the initial conditions for the velocity and wavefunction accordingly:

\begin{equation} \label{truncated initial conditions velocity and wavefunction}
    u_0^N := P^N u_0 \qquad \qquad \psi_0^N := Q^N \psi_0
\end{equation}\medskip

\noindent Since $u_0 \in H^{\frac{3}{2}+\delta}_{d}(\Omega)$ and $\psi_0 \in H^{\frac{5}{2}+\delta}_{0}(\Omega)$, it is necessary to establish that the truncated initial conditions converge to the actual ones in the relevant norms. This is indeed true, and we will now rigorously prove it. For the wavefunction, the proof is rather straightforward because of the equivalence of norms between Sobolev spaces and fractional powers of the Dirichlet Laplacian (Lemma \ref{poincare inequality fractional}).

\begin{lem} [The projection $Q_N$ is convergent] \label{Q^N is convergent}
    Let $s\in [0,5]$. If $\psi \in H^s_{0,x}$, then $Q^N \psi \xrightarrow[N\rightarrow\infty]{H^s}\psi$, and $\norm{Q^N \psi}_{H^s_x} \lesssim \norm{\psi}_{H^s_x}$.
\end{lem}\medskip

\begin{proof}
    Let $\psi\in H^s_{0,x}$ be given by:
    
    \begin{equation*}
        \psi = \sum_{k=1}^{\infty} \langle \psi,b_k \rangle b_k
    \end{equation*}\medskip
    
    Then, using Lemma \ref{poincare inequality fractional}:
    
    \begin{equation*}
        \norm{\psi}_{H^s_x}^2 \equiv \norm{(-\Delta)^{\frac{s}{2}}\psi}_{L^2_x}^2 = \sum_{k=1}^{\infty} \beta_k^{\frac{s}{5}} \abs{\langle \psi,b_k \rangle}^2 < \infty
    \end{equation*}\medskip
    
    Since the sum is finite, the sequence constituting the series must tend to zero. Thus,
    
    \begin{equation*}
        \norm{Q^N \psi - \psi}_{H^s_x}^2 = \norm{\sum_{k=N+1}^{\infty}\langle \psi,b_k \rangle b_k}_{H^s_x}^2 \equiv \sum_{k=N+1}^{\infty} \beta_k^{\frac{s}{5}} \abs{\langle \psi,b_k \rangle}^2 \xrightarrow[N\rightarrow\infty]{} 0
    \end{equation*}\medskip
    
    Finally,
    
    \begin{equation*}
        \norm{Q^N \psi}_{H^s_x}^2 \equiv \norm{(-\Delta)^{\frac{s}{2}} Q^N \psi}_{L^2_x}^2 \le \norm{(-\Delta)^{\frac{s}{2}} \psi}_{L^2_x}^2 \equiv \norm{\psi}_{H^s_x}^2
    \end{equation*}\medskip
\end{proof}

For the case of the velocity, we first have to establish a relation analogous to Lemma \ref{poincare inequality fractional} for the bi-Stokes operator. We will do this by following the approach in Section 3 of \cite{Fefferman2019SimultaneousEigenspaces}. Since the bi-Stokes operator is the Leray projection of the Dirichlet bi-Laplacian operator, we will begin with fractional powers of the latter.\medskip

Consider the following boundary value problem (with $f \in L^2_x)$. The operator of interest is the Dirichlet bi-Laplacian, denoted by $\mathcal{L}_2 = (-\Delta)^2$.

\begin{equation} \label{dirichlet bi-laplacian}
    \begin{gathered}
        (-\Delta)^2 \xi = f \qquad \text{ in } \Omega \\
        \xi = \frac{\partial \xi}{\partial n} = 0 \qquad \text{ on } \partial\Omega
    \end{gathered}
\end{equation}\medskip

Like the other operators discussed above, the bi-Laplacian (defined on $L^2$) is also positive and self-adjoint, with a compact inverse. Thus, it has a positive and non-decreasing spectrum $(0<l_1\le l_2\le l_3 \le \dots \infty)$, and smooth eigenfunctions $(\{w_j\}\in C^{\infty}(\Bar{\Omega}))$ that are orthonormal in $L^2$ and orthogonal in $H^2$. Based on elliptic regularity theory (see Corollary 2.21 in \cite{Gazzola2010PolyharmonicDomains}), the domain of the biharmonic operator is given by $D(\mathcal{L}_2) = H^4 \cap H^2_0$. Now, recalling Definition \ref{fractional operator spaces definition}, we establish the domain of the half-power of the bi-Laplacian.\medskip

\begin{prop} [Domain of half-power of the bi-Laplacian] \label{domain of L_2^0.5}
    \begin{equation*}
        D(\mathcal{L}_2^{\frac{1}{2}}) = H^2_0
    \end{equation*}\medskip
\end{prop}

\begin{proof}
    If $u\in D(\mathcal{L}_2), v \in H^2_0$, then:
    \begin{equation*}
        \langle \mathcal{L}_2 u,v\rangle = \langle (-\Delta)^2 u,v \rangle = \langle D^2 u, D^2 v \rangle
    \end{equation*}\medskip
    
    Since $l_j^{-\frac{1}{2}}w_j \in D(\mathcal{L}_2)$,
    
    \begin{equation*}
        \delta_{jk} = \langle l_j^{-\frac{1}{2}}w_j,l_k^{-\frac{1}{2}}w_k \rangle_{D(\mathcal{L}_2^{\frac{1}{2}})} = \langle \mathcal{L}_2 l_j^{-\frac{1}{2}}w_j,l_k^{-\frac{1}{2}}w_k \rangle = \langle D^2(l_j^{-\frac{1}{2}}w_j),D^2(l_k^{-\frac{1}{2}}w_k) \rangle
    \end{equation*}\medskip
    
    Since $D(\mathcal{L}_2^{\frac{1}{2}})$ is defined via an eigenfunction expansion, the above equality implies that convergence in  $D(\mathcal{L}_2^{\frac{1}{2}})$ guarantees convergence in $H^2_0$. We conclude that $D(\mathcal{L}_2^{\frac{1}{2}})$ is a closed subspace of $H^2_0$.
    
    Now, if $v\in H^2_0$ with $\langle v,u \rangle_{H^2_0} = 0 \ \forall \ u\in D(\mathcal{L}_2^{\frac{1}{2}})$, then for all $j\in\N$:
    
    \begin{equation*}
    \begin{gathered}
        0 = \langle v,w_j \rangle + \langle D^2 v, D^2 w_j \rangle = \langle v,w_j \rangle + \langle v, \mathcal{L}_2w_j \rangle = (1+l_j)\langle v,w_j \rangle \\
        \Rightarrow v = 0 \Rightarrow D(\mathcal{L}_2^{\frac{1}{2}}) = H^2_0
    \end{gathered}
    \end{equation*}\medskip
\end{proof}

We will now interpolate between $D(\mathcal{L}_2^0) = L^2$ and $D(\mathcal{L}_2^{\frac{1}{2}})$. 

%

\begin{mydef} [Fractional bi-Laplacian] \label{fractional bi-Laplacian definition}
    For $0<\theta<\frac{1}{2}$, we define the fractional bi-Laplacian as:
    \begin{equation}
        D(\mathcal{L}_2^{\theta}) = \left(D(\mathcal{L}_2^0),D(\mathcal{L}_2^{\frac{1}{2}})\right)_{\theta}
    \end{equation}\medskip
    
    \noindent where $(X,Y)_{\theta}$ is the \textbf{interpolation space} between Banach spaces $X$ and $Y$ that are both embedded in a common vector space. For details on real interpolation (and the K-method), see \cite{Fefferman2019SimultaneousEigenspaces} for a brief introduction and Chapter 7 of \cite{Adams2003SobolevSpaces} for a detailed exposition.\medskip
\end{mydef}

Having defined the fractional bi-Laplacian, what remains is to interpolate, and the final result is stated in the proposition below.\medskip

\begin{prop} \label{fractional bi-laplacian sobolev spaces}
    For $0<\theta<\frac{1}{2}$, $D(\mathcal{L}_2^{\theta}) = H^{4\theta}_0$.
\end{prop}\medskip

\begin{proof}
    \begin{equation}
    \begin{aligned}
        D(\mathcal{L}_2^{\theta}) &= \left(D(\mathcal{L}_2^0),D(\mathcal{L}_2^{\frac{1}{2}})\right)_{2\theta} \\
        &= (L^2,H^2_0)_{2\theta} \\
        &= H^{(1-2\theta)0 + (2\theta)2}_0\\
        &= H^{4\theta}_0
    \end{aligned}
    \end{equation}\medskip
    
    The first equality is just Definition \ref{fractional bi-Laplacian definition}, with the observation that the interpolation index is $2\theta$ on the RHS since $\theta$ goes from $0$ to $\frac{1}{2}$ and the interpolation always goes from $0$ to $1$. The second equality follows from Definition \ref{fractional operator spaces definition} and Proposition \ref{domain of L_2^0.5}.\medskip
    
    The penultimate equality is from the result (see Corollaries 4.7 and 4.10 in \cite{Chandler-Wilde2015InterpolationCounterexamples}) that interpolation of Sobolev spaces (on a Lipschitz domain) is a closed operation, i.e., yields another Sobolev space.
    
\end{proof}

\medskip

At this stage, we have the domain of definition of the fractional bi-Laplacian. We now return our attention to the bi-Stokes operator $(\mathfrak{S}_2)$, which involves the Leray projector acting on $\mathcal{L}_2$. Therefore, to begin with, we note that $D(\mathfrak{S}_2) = H^4 \cap H^2_0 \cap L^2_d$, where $L^2_d$ is the completion of smooth, divergence-free functions in the $L^2$ norm.\medskip

\begin{prop} \label{fractional bi-stokes operator domain proposition}
    For $0<\theta<1$,
    
    \begin{equation} \label{fractional bi-stokes operator domain}
        D(\mathfrak{S}_2^{\theta}) = D(\mathcal{L}_2^{\theta})\cap L^2_d
    \end{equation}\medskip
    
    In particular, $D(\mathfrak{S}_2^{\theta}) = H^{4\theta}_0 \cap L^2_d$ for $0<\theta<\frac{1}{2}$.
\end{prop}\medskip

\begin{proof}
    Since the bi-Stokes operator is defined on $L^2_d$, we have:
    
    \begin{equation} \label{domain of fractional bi-stokes}
        D(\mathfrak{S}_2^{\theta}) = \left(L^2_d,D(\mathcal{L}_2)\cap L^2_d\right)_{\theta} = \left(L^2\cap L^2_d,D(\mathcal{L}_2)\cap L^2_d\right)_{\theta}
    \end{equation}\medskip
    
    We would like to use the ``intersection lemma" (Lemma 3.4) of \cite{Fefferman2019SimultaneousEigenspaces} in order to commute the interpolation and intersection operations. To this end, we must construct an operator $T: L^2\rightarrow L^2_d$ such that $T\rvert_{L^2_d} = Id$ (the identity), and $T$ must also be bounded from $D(\mathcal{L}_2)$ to $D(\mathcal{L}_2)\cap L^2_d$.\medskip
    
    First, consider the operator $\Tilde{T}: D(\mathcal{L}_2) \rightarrow D(\mathfrak{S}_2)$ given by $$\Tilde{T} := \mathfrak{S}_2^{-1}\Leray\mathcal{L}_2$$
    
    From \eqref{elliptic regularity for bi-stokes problem}, $\mathfrak{S}_2^{-1}:L^2_d \mapsto D(\mathfrak{S}_2)$ is bounded, and so we have $\norm{\mathfrak{S}_2^{-1}f}_{H^4} \lesssim \norm{f}_{L^2}$. Thus, for any $f\in D(\mathcal{L}_2)$:
    
    \begin{equation*}
        \norm{\Tilde{T}f}_{D(\mathfrak{S}_2)} \lesssim \norm{\Tilde{T}f}_{H^4} = \norm{\mathfrak{S}_2^{-1}\Leray\mathcal{L}_2f}_{H^4} \lesssim \norm{\Leray\mathcal{L}_2f}_{L^2} \le \norm{\mathcal{L}_2f}_{L^2} \lesssim \norm{f}_{D(\mathcal{L}_2)}
    \end{equation*}\medskip
    
    This shows that $\Tilde{T}$ is bounded from $D(\mathcal{L}_2)$ to $D(\mathfrak{S}_2)$. Now, for $g\in L^2_d$ and $f\in D(\mathcal{L}_2)$, since $\mathfrak{S}_2,\mathcal{L}_2$ are self-adjoint and $\Leray$ is symmetric, 
    
    \begin{multline*}
        \abs{\langle g,\Tilde{T}f \rangle} = \abs{\langle g,\mathfrak{S}_2^{-1}\Leray\mathcal{L}_2 f \rangle} = \abs{\langle \mathfrak{S}_2^{-1}g,\Leray\mathcal{L}_2 f \rangle} = \abs{\langle \mathfrak{S}_2^{-1}g,\mathcal{L}_2 f \rangle} \\ 
        = \abs{\langle \mathcal{L}_2\mathfrak{S}_2^{-1}g,f \rangle} \le \norm{\mathfrak{S}_2^{-1}g}_{H^4} \norm{f}_{L^2} \lesssim \norm{g}_{L^2} \norm{f}_{L^2}
    \end{multline*}
    
    \begin{equation*}
        \Rightarrow \norm{\Tilde{T}f}_{L^2_d} \lesssim \norm{f}_{L^2}
    \end{equation*}\medskip
    
    Since $\Tilde{T}$ is linear and $D(\mathcal{L}_2)$ is dense in $L^2$, we can therefore extend $\Tilde{T}$ to $T:L^2\mapsto L^2_d$. This operator is also the identity on $L^2_d$, since $f\in L^2_d$ can be expanded in terms of the eigenfunctions of $\mathfrak{S}_2$, and noting that $f\in D(\mathfrak{S}_2)$, we have $\Leray\mathcal{L}_2 = \mathfrak{S}_2$. 
    
    \medskip
    
    Using \eqref{domain of fractional bi-stokes}, the map $T$ constructed above, and Lemma 3.4 of \cite{Fefferman2019SimultaneousEigenspaces}, we arrive at the required result.
\end{proof}

\bigskip

Finally, we are ready to prove that the truncated initial velocity is also convergent.

\begin{thm} [The projection $P_N$ is convergent] \label{P^N is convergent}
    Let $r\in [0,2]$. If $u \in H^r_{d,x}$, then $P^N u \xrightarrow[N\rightarrow\infty]{H^r}u$, and $\norm{P^N u}_{H^r_{d,x}} \lesssim \norm{u}_{H^r_{d,x}}$. 
\end{thm}\medskip

\begin{proof}
    Let $u\in H^r_{d,x}$ so that
    \begin{equation*}
        u = \sum_{k=1}^{\infty} \langle u,a_k \rangle a_k
    \end{equation*}
    
    Then, using Proposition \ref{fractional bi-stokes operator domain proposition}:
    
    \begin{equation*}
        \norm{u}_{H^r_x}^2 \equiv \norm{(\mathfrak{S}_2)^{\frac{r}{4}}u}_{L^2_x}^2 = \sum_{k=1}^{\infty} \alpha_k^{\frac{r}{2}} \abs{\langle u,a_k \rangle}^2 < \infty
    \end{equation*}\medskip
    
    Since the sum is finite, the sequence constituting the series must tend to zero. Thus,
    
    \begin{equation*}
        \norm{P^N u - u}_{H^r_x}^2 = \norm{\sum_{k=N+1}^{\infty}\langle u,a_k \rangle a_k}_{H^r_x}^2 \equiv \sum_{k=N+1}^{\infty} \alpha_k^{\frac{r}{2}} \abs{\langle u,a_k \rangle}^2 \xrightarrow[N\rightarrow\infty]{} 0
    \end{equation*}\medskip
    
    Moreover,
    
    \begin{equation*}
        \norm{P^N u}_{H^r_x}^2 \equiv \norm{(\mathfrak{S}_2)^{\frac{r}{4}} P^N u}_{L^2_x}^2 \le \norm{(\mathfrak{S}_2)^{\frac{r}{4}} u}_{L^2_x}^2 \equiv \norm{u}_{H^r_x}^2
    \end{equation*}
\end{proof}

\medskip

Given the regularity of the initial conditions, we deduce the convergence of the truncated initial conditions by applying Lemma \ref{Q^N is convergent} and Theorem \ref{P^N is convergent}.

\begin{cor} [Truncated initial conditions are convergent] \label{truncated initial conditions are convergent}
    If $\psi_0\in H^{\frac{5}{2}+\delta}_0$ and $u_0\in H^{\frac{3}{2}+\delta}_d$, then $\psi_0^N \xrightarrow[N\rightarrow \infty]{H^{\frac{5}{2}+\delta}} \psi_0$ and $u_0^N \xrightarrow[N\rightarrow \infty]{H^{\frac{3}{2}+\delta}} u_0$.
\end{cor}\medskip

\subsubsection{The initial density} \label{the initial density}

Given the initial density field $\rho_0 \in L^2_x \subset L^{\infty}_x$, we consider an approximating sequence $\rho^N_0 \in C^1_x$, such that $\rho_0^N \xrightarrow[N\rightarrow\infty]{L^2}\rho$, and $m\le \rho_0^N \le M$. (Recall that $m\le \rho_0 \le M$.) This approximating sequence may be constructed as follows. For each $N\in\N$, define $\Omega_{\frac{1}{N}} = \Omega \cup \{ x\notin \Omega : dist(x,\partial\Omega) < \frac{1}{N} \}$, so that $\Omega_{\frac{1}{N+1}} \Subset \Omega_{\frac{1}{N}}$ for all $N$. Now, extend $\rho_0$ to $\Tilde{\rho_0}$ over $\Omega_{1}$.

\begin{equation} \label{extended density}
    \renewcommand{\arraystretch}{2}
    \Tilde{\rho_0} = 
    \left\{ 
    \begin{array}{c} 
    \begin{aligned}
        \rho_0 \qquad &x\in \Omega \\ 
        m \qquad &x\in \Omega_{1}\setminus\Omega
    \end{aligned}  
    \end{array} 
    \right.
\end{equation}\medskip

Define a mollifier $\zeta:\R^3 \mapsto \R^+$. It is a smooth, compactly supported (on the unit ball), non-negative function with unit mass, i.e., $\int_{B_1}\zeta = 1$. Here, $B_r$ is a ball of radius $r$ (centered at the origin). We will now scale this mollifier in a mass-invariant way: $\zeta_{\frac{1}{N}} (x) := N^3 \zeta\left( Nx \right)$. This means that the support of $\zeta_{\frac{1}{N}}$ is in $B_{\frac{1}{N}}$. The approximating sequence is obtained through convolution with the mollifier, and restriction to $\Omega$, i.e.,

\begin{equation}
    \rho_0^N = \left. \left(\zeta_{\frac{1}{N}} * \Tilde{\rho_0}\right) \right|_{\Omega}
\end{equation}\medskip

The $\rho_0^N$ are obviously smooth since convolution upgrades regularity. They are also bounded as required, because $\Tilde{\rho_0} \in [m,M]$ and the mollifier has unit mass. Since $\Omega\Subset\Omega_1$, we have $\Tilde{\rho_0}\in L^p(\Omega) = L^p_{loc}(\Omega_1)$ for $p\in [1,\infty)$, which implies (from Theorem 6 in Appendix C of \cite{Evans2010PartialEquations}) that $\rho_0^N \xrightarrow[N\rightarrow\infty]{L^p(\Omega)} \Tilde{\rho_0}$. But, by construction, $\Tilde{\rho_0} = \rho_0$ in $\Omega$; therefore, $\rho_0^N \xrightarrow[N\rightarrow\infty]{L^p(\Omega)} \rho_0$.\medskip

\subsection{Approximate equations}

\subsubsection{The continuity equation}

Having described the (truncated) initial conditions and the semi-Galerkin scheme, we will now establish the existence of solutions to the ``approximate" equations, starting with the continuity equation.

\begin{equation} \label{approximate continuity equation}
    \begin{aligned}
        \partial_t \rho^N + u^N\cdot\nabla\rho^N &= 2\Lambda\Re (\overline{\psi^N}B^N\psi^N) \\
        \rho^N(0,x) &= \rho^N_0(x)
    \end{aligned}
\end{equation}\medskip

\noindent where $B^N = -\frac{1}{2}\Delta + \frac{1}{2}\abs{u^N}^2 + iu^N\cdot\nabla + \mu\abs{\psi^N}^2$.\medskip

Just as in \eqref{constraint to choose existence time}, we see that the constraint that fixes the local existence time $T_N$ for \eqref{approximate continuity equation} is:

\begin{equation} \label{constraint to choose approximate existence time}
    2\Lambda T_N^{\frac{1}{2}} \norm{\psi^N}_{L^{\infty}_{[0,T_N]}L^{\infty}_x} \norm{B^N\psi^N}_{L^2_{[0,T_N]}L^{\infty}_x} \le m-\varepsilon
\end{equation}\medskip

Recall that $T_N$ is also updated based on Definition \ref{updated local existence time}, to accommodate the Gr\"onwall inrequality calculation in the a priori bounds. 
Now, using Lemma \ref{sobolev embeddings}, the a priori estimate for $\psi$ in \eqref{higher order inclusion implication} and for $B\psi$ in \eqref{Bpsi sought-after bound}, and Lemma \ref{Q^N is convergent}, we can choose a local existence time that is independent of $N$.

\begin{align*}
    \norm{\psi^N}_{L^{\infty}_{[0,T_N]} L^{\infty}_x} &\lesssim \norm{\psi^N}_{L^{\infty}_{[0,T_N]} H^2_x} \\
    &\lesssim 1 + \nu\norm{u^N_0}_{H^1_x} + \norm{\psi^N_0}_{H^2_x} \\ 
    &\lesssim 1 + \nu\norm{u_0}_{H^1_x} + \norm{\psi_0}_{H^2_x}
\end{align*}\medskip

Similarly, 

\begin{align*}
    \norm{B^N\psi^N}_{L^2_{[0,T_N]}L^{\infty}_x} &\lesssim \norm{B^N\psi^N}_{L^2_{[0,T_N]} H^{\frac{3}{2}+\delta}_x} \\ 
    &\lesssim \Lambda^{-\frac{1}{2}} \left(Q_{T_N}^{\frac{1}{2}} + 1 \right) e^{cQ_{T_N}} \norm{\psi^N_0}_{H^{\frac{3}{2}+\delta}_x} \\
    &\lesssim \Lambda^{-\frac{1}{2}} \left(Q_{T_N}^{\frac{1}{2}} + 1 \right) e^{cQ_{T_N}} \norm{\psi_0}_{H^{\frac{3}{2}+\delta}_x}
\end{align*}\medskip

\noindent where $Q_{T_N}$ is defined in \eqref{defining Q_T}. Thus, substituting these estimates into \eqref{constraint to choose approximate existence time} gives:

\begin{equation} \label{final constraint to choose approximate existence time}
    2c\Lambda^{\frac{1}{2}} T_N^{\frac{1}{2}} \left( 1 + \nu\norm{u_0}_{H^1_x} + \norm{\psi_0}_{H^2_x} \right) \left(Q_{T_N}^{\frac{1}{2}} + 1 \right) e^{cQ_{T_N}} \norm{\psi_0}_{H^{\frac{3}{2}+\delta}_x} \le m-\varepsilon
\end{equation}\medskip

It is sufficient to choose $T_N$ small enough to satisfy \eqref{final constraint to choose approximate existence time}. Since no mention of the index $N$ is made in the constraint in any of the initial conditions, it is clear that $T_N$ can be chosen independent of $N$. Having arrived at the local existence time, we will now establish the analogs of Lemmas 2.2 and 2.3 from \cite{Kim1987WeakDensity}. These constitute the existence of solutions to \eqref{approximate continuity equation} and a convergence result for the same, respectively. 

\begin{lem} \label{existence of solutions to approx continuity equation}
    Let $u^N \in C^0_{[0,T]} C^1_{\Bar{\Omega}}$ and $\overline{\psi^N}B^N \psi^N\in L^1_{[0,T]} L^{\infty}_x$ (uniformly in $N$), with $u^N(t,\partial\Omega) = 0$ and $\nabla\cdot u^N (t,\Bar{\Omega})$ for $t\in [0,T]$. Then, \eqref{approximate continuity equation} has a unique solution $\rho^N \in C^1_{[0,T]}C^1_x$.
\end{lem}

\begin{proof}
    To avoid having to deal with problems of derivatives at the boundary, let us extend $u^N$ to $w^N\in C^0_{[0,T]} C^1_x$, such that:
    
    \begin{enumerate}[(i)]
        \item $w^N = u^N \ \forall \ (t,x)\in ([0,T]\times\Bar{\Omega})$; in particular, $w^N = 0$ on the boundary
        \item $w^N$ is supported on an open set $E^N$ such that $\Omega \Subset E^N$
    \end{enumerate}\medskip
    
    Consider the evolution equation for the characteristics of the flow.
    
    \begin{equation} \label{evolution of the characteristics}
        \begin{aligned}
            \frac{dx^N}{dt} &= w^N(t,x^N(t)) \\
            x^N(0) &= y^N \in \Bar{\Omega}
        \end{aligned}
    \end{equation}\medskip
    
    Since $w^N \in C^0_{[0,T]} C^1_x$, there exists a unique solution $x^N(t,y^N) \in C^1_{[0,\Tilde{T}]} C^1_{\Bar{\Omega}}$ for some $0<\Tilde{T}\le T$. Because $w^N = u^N = 0$ on $\partial\Omega$, for any $y^N\in\partial\Omega$, we have $x^N(t,y^N) = y^N\in\partial\Omega$. This implies that characteristics starting inside/on/outside the boundary, remain inside/on/outside the boundary. Thus, by uniqueness of the solution, $\Tilde{T}=T$ and $x^N(t,y^N) \in C^1_{[0,T]} C^1_{\Bar{\Omega}}$ for $u^N\in C^0_{[0,T]}C^1_{\Bar{\Omega}}$.\medskip
    
    Owing to the incompressibility of the flow $u^N$, it follows that $\det\left( \frac{\partial x^N_i}{\partial y^N_j}\right) = 1$, allowing us to conclude that the characteristics are $C^1$ diffeomorphisms and therefore, invertible:
    
    \begin{equation*}
        y^N = S^{-1}_t x^N := y^N(t,x^N)
    \end{equation*}\medskip
    
    We will now define the solution to \eqref{approximate continuity equation} along characteristics:
    
    \begin{equation} \label{solution to approx continuity eqn, along characteristics}
        \rho^N(t,x) = \rho^N_0\left(y^N(t,x)\right) + 2\Lambda\int_0^t \Re\left( \overline{\psi^N}B^N \psi^N \right) \left( \tau,y^N(t-\tau,x) \right) d\tau
    \end{equation}\medskip
    
    That \eqref{solution to approx continuity eqn, along characteristics} uniquely solves \eqref{approximate continuity equation} can be easily checked using the following property of the ``inverse-characteristics" $y(t,x)$. For any $\tau\in\R$,
    
    \begin{align*}
        \frac{\partial}{\partial t}y(t,x) &= \lim_{\Delta t\rightarrow 0} \frac{y(t-\tau+\Delta t,x) - y(t-\tau,x)}{\Delta t} \\
        &= \lim_{\Delta t\rightarrow 0} \frac{x(t+\Delta t,y) - x(t,y)}{\Delta t} \cdot \frac{y(t-\tau+\Delta t,x) - y(t-\tau,x)}{x(t+\Delta t,y) - x(t,y)} \\
        &= u(t,x) \cdot \frac{\partial_t y(t-\tau,x)}{\partial_t x(t,y)} \\
        &= -u(t,x)\cdot\nabla_{x}y(t-\tau,x)
    \end{align*}\medskip
    
    The last equality is due to Euler's chain rule (also known as the triple product rule). Furthermore, by choosing an appropriately small existence time $T$ (as before), we can ensure that for $m\le \rho^N_0 \le M$, we have $m\le \rho^N (t,\Bar{\Omega}) \le M$ for $t\in [0,T]$.
    
\end{proof}

\medskip

Now, we will consider a convergent sequence of velocities and wavefunctions that belong to the finite-dimensional subspaces spanned by the truncated Galerkin scheme. Given such a convergent sequence, we show that the sequence of density fields satisfying \eqref{approximate continuity equation} is also convergent, and this will be used to complete a contraction mapping argument later on.

\begin{lem} \label{convergence of solutions to approx continuity equation}
   For $n\in\N$, let $u^N_n \in C^0_{[0,T]} C^1_{\Bar{\Omega}}$ and $\overline{\psi^N_n}B^N_n \psi^N_n\in L^1_{[0,T]} L^{\infty}_x$ (uniformly in $n$), with $u^N_n(t,\partial\Omega) = 0$ and $\nabla\cdot u^N_n (t,\Bar{\Omega})$ for $t\in [0,T]$. Denote by $\rho^N_n \in C^1_{[0,T]}C^1_x$ the unique solution to the system:
   
   \begin{equation} \label{sequence approximate continuity equation}
       \begin{aligned}
           \partial_t \rho^N_n + u^N_n\cdot\nabla\rho^N_n &= 2\Lambda\Re (\overline{\psi^N_n}B^N_n\psi^N_n) \\
        \rho^N_n(0,x) &= \rho^N_0(x) \in C^1_x
       \end{aligned}
   \end{equation}\medskip
   
   If $u^N_n \xrightarrow[n\rightarrow\infty]{C^0_{[0,T]}C^1_{\Bar{\Omega}}} u^N$ and $\psi^N_n \xrightarrow[n\rightarrow\infty]{C^0_{[0,T]}C^3_{\Bar{\Omega}}} \psi^N$, then $\rho^N_n \xrightarrow[n\rightarrow\infty]{C^0_{[0,T]}C^0_{\Bar{\Omega}}} \rho^N$, where $\rho^N$ solves \eqref{approximate continuity equation}.

\end{lem}

\begin{proof}
    First, let us define $\Psi^N_n := 2\Lambda\Re(\Bar{\psi^N_n}B^N_n\psi^N_n)$. Since $u^N_n\in C^0_t C^1_{\Bar{\Omega}}$, there exists a sequence of characteristics $x^N_n(t,y)\in C^1_t C^1_{\Bar{\Omega}}$ corresponding to the flow, i.e., solving $\frac{dx^N_n}{dt} = u^N_n (t,x^N_n)$ with $x^N_n(0,y) = y$. The assumed convergence of $u^N_n$ allows us to conclude that $x^N_n \xrightarrow[n\rightarrow\infty]{C^1_t C^1_{\Bar{\Omega}}}x^N$. \medskip
    
    Consider the map $y\mapsto x^N_n(t,y)$ and define its inverse $y^N_n(t,x)$; this is just the inverse of the characteristic, i.e., if the flow were reversed. Due to the flow being incompressible, we know that the matrix $\frac{\partial y^N_n}{\partial x}$ is invertible. Also, as shown in the proof of the previous lemma, $\frac{\partial}{\partial t}y^N_n = -u^N_n\cdot\nabla_{x}y^N_n$. This implies that the derivatives of $y^N_n$ with respect to both space and time are bounded uniformly in $n$, $t$ and $x$. Thus, by the Arzela-Ascoli theorem, we can extract a subsequence that converges uniformly: $y^N_n\xrightarrow[n\rightarrow\infty]{C^0_t C^0_{\Bar{\Omega}}}y^N$. Just as before, we can show that the solution to \eqref{sequence approximate continuity equation} is
    
    \begin{equation} \label{solution to sequence approx continuity eqn, along characteristics}
        \rho^N_n(t,x) = \rho^N_0\left(y^N_n(t,x)\right) + \int_0^t \Psi^N_n \left( \tau,y^N_n(t-\tau,x) \right) d\tau
    \end{equation}\medskip
    
    Therefore,
    
    \begin{multline*}
        \rho^N_n(t,x) - \rho^N(t,x) = \rho^N_0\left(y^N_n(t,x)\right) + \int_0^t \Psi^N_n \left( \tau,y^N_n(t-\tau,x) \right) d\tau \\ - \rho^N_0\left(y^N(t,x)\right) - \int_0^t \Psi^N \left( \tau,y^N(t-\tau,x) \right) d\tau
    \end{multline*}\medskip
    
    \noindent which leads to
    
    \begin{align*}
        \abs{\rho^N_n - \rho^N}_{C^0_{t,x}} &\le \abs{\rho^N_0\left(y^N_n\right) - \rho^N_0\left(y^N\right)}_{C^0_{t,x}} + T \abs{\Psi^N_n \left( t,y^N_n \right) - \Psi^N \left( t,y^N \right)}_{C^0_{t,x}} \\
        &\begin{multlined}\le \norm{\nabla\rho^N_0}_{L^{\infty}_x}\abs{y^N_n - y^N}_{C^0_{t,x}} \\ + T\left[ \norm{\nabla\Psi^N_n}_{L^{\infty}_t L^{\infty}_x}\abs{y^N_n - y^N}_{C^0_{t,x}} + \abs{\Psi^N_n - \Psi^N}_{C^0_{t,x}} \right]
        \end{multlined} \\
        &\xrightarrow[n\rightarrow\infty]{} 0
    \end{align*}\medskip
    
    Given the convergence of $y^N_n$ derived above, and because $\rho^N_0\in C^1_x$, the first term on the RHS vanishes. The second and third terms vanish on account of the following argument. Note that $\Psi^N_n$ has its highest order term of the form $\psi^N_n\Delta\psi^N_n$ (second derivative), and so the  assumed convergence of $\psi^N_n$ in the $C^0_t C^3_x$ norm implies that $\Psi^N_n$ converges in $C^0_t C^1_x$. This also guarantees that $\norm{\nabla\Psi^N_n}_{L^{\infty}_{t} L^{\infty}_x}$ is finite, uniformly in $n$. 
    
\end{proof}


\subsubsection{The Navier-Stokes equation}

Suppose that the existence time has been chosen so that the density $\rho^N\in C^1_{t,x}$ is bounded below. We will now consider an ``approximate momentum equation", composed of the truncated wavefunction and velocity fields defined by \eqref{truncated wavefunction definition} and \eqref{truncated velocity definition}, respectively.

\begin{equation} \label{approximate NSE}
    \rho^N\partial_t u^N + \rho^N u^N\cdot\nabla u^N + \nabla\Tilde{p}^N - \nu\Delta u^N = - 2\Lambda\Im\left( \nabla\overline{\psi^N}B^N\psi^N \right) - 2\Lambda u^N\Re\left(\overline{\psi^N}B^N\psi^N \right)
\end{equation}\medskip

Recall that the incompressiblity condition is built-in, because the eigenfunction basis used to construct the velocity fields are divergence-free. Now, taking the $L^2$ inner product of \eqref{approximate NSE} with $a_j(x)$ for $1\le j\le N$, we arrive at a system of equations for the coefficients describing the time-dependence of the truncated velocity fields.

\begin{equation} \label{expansion of approximate NSE}
    \sum_{k=1}^N R^N_{jk}(t) \frac{d}{dt}c^N_k(t) = -\nu \sum_{k=1}^N D_{jk} c^N_k(t) - \sum_{k,l=1}^N \mathcal{N}^N_{jkl}(t) c^N_k(t) c^N_l(t) - 2\Lambda S^N_j[t,c^N]
\end{equation}\medskip

\noindent where

\begin{gather*}
    R^N_{jk}(t) = \int_{\Omega}\rho^N a_j\cdot a_k \\
    D_{jk} = \int_{\Omega} (\nabla a_j):(\nabla a_k) \\
    \mathcal{N}^N_{jkl}(t) = \int_{\Omega}\rho^N \left( a_k\cdot\nabla \right) a_l \cdot a_j \\
    S^N_j(t,c^N) = \int_{\Omega}a_j\cdot \left[ \Im\left( \nabla\overline{\psi^N}B^N(c^N)\psi^N \right) + u^N(c^N)\Re\left(\overline{\psi^N}B^N(c^N)\psi^N \right) \right]
\end{gather*}\medskip

Since we have both lower and upper bounds on the density in the chosen interval of time, we can use Lemma 2.5 in \cite{Kim1987WeakDensity} to show that the matrix $R^N(t)$ is invertible. Therefore, we arrive at the desired evolution equation (written vectorially) for the coefficients $c^N_j(t)$.

\begin{equation} \label{evolution equation for c^N}
    \frac{d}{dt}c^N = -\nu (R^N)^{-1} D\cdot c^N - (R^N)^{-1}\left[ \mathcal{N}^N : c^N \otimes c^N \right] -2\Lambda (R^N)^{-1} S^N(t,c^N)
\end{equation}\medskip

\subsubsection{The nonlinear Schr\"odinger equation}

As in the previous section, we will derive an evolution equation for the coefficients of the truncated wavefunction, by considering an ``approximate NLS".

\begin{equation} \label{approximate NLS}
    \partial_t \psi^N = -\frac{1}{2i}\Delta \psi^N - \Lambda B^N_L \psi^N - (\Lambda + i)\mu \abs{\psi^N}^2 \psi^N
\end{equation}\medskip

Recall that $B_L = B - \mu \abs{\psi}^2$, i.e., the linear (in $\psi$) part of the coupling operator. Performing an $L^2$ inner product with $b_j(x)$:

\begin{equation}
    \frac{d}{dt}d^N_j (t) = \frac{1}{2i} \beta_j^{\frac{1}{5}} d^N_j(t) - \Lambda\sum_{k=1}^N L^N_{jk} (t) d^N_k(t) - (\Lambda + i)\mu \sum_{k,l,m=1}^N G_{jklkm} d^N_k \overline{d^N_l} d^N_m (t)
\end{equation}\medskip

\noindent where

\begin{gather*}
    L^N_{jk}(t) = \int_{\Omega} b_j B^N_L b_k = \frac{1}{2}\int_{\Omega} \nabla b_j\cdot \nabla b_k + \frac{1}{2}\int_{\Omega} \abs{u^N}^2 b_j b_k + i\int_{\Omega} u^N\cdot\nabla b_k \ b_j \\
    G_{jklm} = \int_{\Omega} b_j b_k b_l b_m
\end{gather*}

Written vectorially, the evolution equation for the coefficients $d^N_j(t)$ becomes:

\begin{equation} \label{evolution equation for d^N}
    \frac{d}{dt}d^N = \frac{1}{2i} \ \mathcal{B} d^N - \Lambda \  L^N\cdot d^N - (\Lambda + i)\mu \ G :: (d^N\otimes\overline{d^N}\otimes d^N)
\end{equation}\medskip

\noindent where $\mathcal{B}_{ij} = \beta_i^{\frac{1}{5}}\delta_{ij}$. \medskip

\subsubsection{Fixed point argument for the coefficients} \label{Fixed point argument for the coefficients}

For a fixed $N$, we will now show that \eqref{evolution equation for c^N} and \eqref{evolution equation for d^N} have unique solutions that are continuous in $[0,T]$. For the remainder of this section, we will drop the superscript $N$ on the coefficients, for brevity. Furthermore, we will also use $c$ (and respectively $d$) to refer to a vector in $\R^N$ (and respectively in $\C^N$). We will not have any reason to call upon the individual coefficients $c_1, c_2, \dots c_N$ (and respectively $d_1, d_2, \dots d_N$). The subscripts used in this section will refer to the iterates used in the fixed point argument to construct solutions to \eqref{evolution equation for c^N} and \eqref{evolution equation for d^N}. \medskip

Let us start with initial vectors $c_0\in B_{r_c} \subset \R^N$ and $d_0\in B_{r_d} \subset\C^N$, where we recall that $B_r$ is a ball of radius $r$ centered at the origin. Now, we will define the iterative mild solutions to \eqref{evolution equation for c^N} and \eqref{evolution equation for d^N} as follows. For $n=1,2,3 \dots$,

\begin{align} \label{iterative mild solutions}
\begin{aligned}
    c_{n+1}(t) &= c_0 + \int_0^t RHS[c_n] d\tau \\
    d_{n+1}(t) &= d_0 + \int_0^t RHS[d_n] d\tau
\end{aligned}
\end{align}\medskip

\noindent where $RHS[c_n]$ is the RHS of \eqref{evolution equation for c^N} for the iterate $c_n$, and similarly for $RHS[d_n]$. For the inductive argument of the contraction mapping, let us assume that $\abs{c_n(t)}\le s_c$ and $\abs{d_n(t)}\le s_d$, for some positive real numbers $s_c, s_d$.\medskip

\begin{enumerate}
    \item \textit{The self-map}: From the polynomial structure of the nonlinearities on the RHS of \eqref{iterative mild solutions}, it is easy to see that:
    
    \begin{align} \label{self maps}
    \begin{aligned}
        \abs{c_{n+1}} &\le r_c + k_c T_c \left[ s_c + s_c^2 + (1+s_c)s_d^2 \left( (1+s_c)^2 + s_d^2 \right) \right] \\
        \abs{d_{n+1}} &\le r_d + k_d T_d \left[ s_d + (s_c^2 +s_c)s_d + s_d^3 \right]
    \end{aligned}
    \end{align}\medskip
    
    \noindent where $T_c$ and $T_d$ are the existence times for the iterative scheme, and $k_c, k_d$ are positive constants. We will first choose $s_i$ large enough that $r_i \le \frac{1}{2}s_i$, for $i\in \{c,d\}$. Then, we will choose $T_c, T_d$ small enough so that the second terms on the RHS of \eqref{self maps} satisfy $k_i T_i [\dots] \le r_i$, for $i\in \{c,d\}$. These choices ensure that starting from $c_n\in B_{r_c}$, we will end up at $c_{n+1}\in B_{r_c}$ (and similarly for $d_n$). \medskip
    
    \item \textit{The contraction}: Beginning from $(c_0,d_0)$, let there be two pairs of maps $(c_n,d_n)\mapsto (c_{n+1},d_{n+1})$ and $(c_n',d_n')\mapsto (c_{n+1}',d_{n+1}')$. This implies:
    
    \begin{align} \label{contractions}
    \begin{aligned}
        \abs{c_{n+1}-c_{n+1}'} &\le k_c T_c \abs{RHS[c_n] - RHS[c_n']} \\
        \abs{d_{n+1}-d_{n+1}'} &\le k_d T_d \abs{RHS[d_n] - RHS[d_n']}
    \end{aligned}
    \end{align}\medskip
    
    Since the ``RHS" terms are polynomials, one can easily make these contractions by choosing $T_c,T_d$ sufficiently small. At this stage, these times are functions of the size of the initial data $r_c,r_d$. Using a standard bootstrap argument, we can find a maximal time\footnote{This is the same as the local existence time defined earlier, due to the a priori estimates. The latter guarantee that as long as the density is bounded below, the energy of the system is bounded above, implying that the coefficients of time dependence are bounded.} and call it $T$.\medskip
\end{enumerate}

Using the Banach fixed point theorem, we conclude that the sequences $\{c_n(t),d_n(t)\}$ converge to $\{c(t),d(t)\}\in C[0,T]$, respectively.\medskip

\begin{rem} \label{restoring labels to the convergent iterates}
Recall that we dropped the superscript on the coefficients at the beginning of Section \ref{Fixed point argument for the coefficients}. Restoring this, we realize that the limits of the above iteration are actually $c^N(t)$ and $d^N(t)$, and each of them are $N-$dimensional vectors. In what follows, we will continue to use subscripts to refer to the different iterates, and not the components of the vector of coefficients. For instance, $c^N_n$ will denote the $n$-th iterate of the vector $c^N$.
\end{rem}\medskip

For a pair $(u^N_n,\psi^N_n)$, equivalently $(c^N_n,d^N_n)$, using Lemma \ref{existence of solutions to approx continuity equation}, we can find a solution $\rho^N_n$. Owing to the smoothness (in space) of the eigenfunctions used in the truncated velocity and wavefunction, and the above fixed point argument, it is easy to see that $u^N_n \xrightarrow[n\rightarrow\infty]{C^0_t C^1_x}u^N$ and $\psi^N_n \xrightarrow[n\rightarrow\infty]{C^0_t C^3_x}\psi^N$. Therefore, performing an iteration on the triplet $(c^N_n,d^N_n,\rho^N_n)$ and using Lemma \ref{convergence of solutions to approx continuity equation}, we conclude that the sequence $\rho^N_n$ converges to $\rho^N\in C^0_{[0,T]} C^0_{\Bar{\Omega}}$.\medskip

\begin{rem}
Note that by taking $\rho_0^N \in C^k_x$, we can get a solution $\rho^N\in C^0_t C^k_x$ in Lemma \ref{existence of solutions to approx continuity equation}. Similarly, in Lemma \ref{convergence of solutions to approx continuity equation}, with $\rho_0^N \in C^k_x$, we can easily show that $\rho^N_n \xrightarrow[n\rightarrow\infty]{C^0_t C^{k-1}_x}\rho^N$. The idea behind this is that $u^N,\psi^N$ have $C^{\infty}$ regularity in space, but only $C^0$ regularity in time.
\end{rem}\medskip

\subsection{Weak and strong convergences}

\subsubsection{Compactness arguments} \label{weak and strong convergences}

Let us now extract convergent subsequences from the a priori bounds in \eqref{final a priori bounds}.\medskip

\begin{enumerate}
    \item \textit{Density}: We know that $\rho^N \in C^0([0,T];C^0_x) \subset L^{\infty}(0,T;L^2_x)$. Also, from \eqref{approximate continuity equation}, 
    
    \begin{align} \label{partial_t rho^N uniformly bounded in L^2_t H^-1_x}
        \begin{aligned}
            \norm{\partial_t \rho^N}_{L^2_{[0,T]}H^{-1}_x} &\lesssim \norm{\nabla\cdot (u^N\rho^N)}_{L^2_{[0,T]}H^{-1}_x} + \norm{\Re (\overline{\psi^N}B^N\psi^N)}_{L^2_{[0,T]}H^{-1}_x} \\
            &\lesssim \norm{u^N\rho^N}_{L^2_{[0,T]}L^2_x} + \norm{ (\overline{\psi^N}B^N\psi^N)}_{L^2_{[0,T]}L^2_x} \\
            &\begin{multlined}
                \le T^{\frac{1}{2}}\norm{\sqrt{\rho^N}u^N}_{L^{\infty}_{[0,T]}L^2_x}\norm{\sqrt{\rho^N}}_{L^{\infty}_{[0,T]}L^{\infty}_x} \\ + \norm{ \psi^N}_{L^{\infty}_{[0,T]}L^{\infty}_x}\norm{ B^N\psi^N}_{L^2_{[0,T]}L^2_x}
            \end{multlined}
        \end{aligned}    
    \end{align}\medskip
    
    The second inequality is due to the (compact) embedding $L^2_x \subset H^{-1}_x$ for bounded domains. All the terms in the last line are (uniformly) finite by virtue of the a priori bounds in \eqref{final a priori bounds}. Therefore, using Lemma \ref{aubin-lions}, we conclude the following strong convergence\footnote{Refer to Section \ref{notation} for the notation used in the case of Sobolev spaces of the $x$-variable.} of a subsequence:
    
    \begin{equation} \label{strong convergence of density}
        \rho^N \xrightarrow[N\rightarrow \infty]{C^0_t H^{0^-}_x} \rho
    \end{equation}\medskip
    
    Consider a (relabeled) subsequence $\rho^N$ that strongly converges to $\rho$ in $C([0,T];H^{-1}_x)$. For a.e. $s,t\in [0,T]$ and any $\omega \in H^1_{0,x}$,
    
    \begin{multline*}
        \langle \rho^N(t)-\rho^N(s),\omega \rangle_{H^{-1}\times H^1_0} = \langle \int_s^t \partial_t \rho^N d\tau ,\omega \rangle_{H^{-1}\times H^1_0} \\ \le \int_s^t \norm{\partial_t \rho^N}_{H^{-1}_x} \norm{\omega}_{H^1_x}\le (t-s)^{\frac{1}{2}} \norm{\partial_t \rho^N}_{L^2_{[0,T]} H^{-1}_x} \norm{\omega}_{H^1_x}
    \end{multline*}\medskip
    
    \noindent showing that $\langle \rho^N(t),\omega \rangle_{H^{-1}\times H^1_0}$ is uniformly continuous in $[0,T]$, uniformly in $N$. This, along with Lemma \ref{weak-continuity in time}, allows us to conclude that $\rho^N$ is relatively compact in $C_w([0,T];L^2_x)$. \medskip
    
    Now, we will extend the strong convergence of the approximate density fields to include the strong $L^2$ topology in space, i.e., we want to show that $\rho^N$ (or an appropriate subsequence) converges to $\rho$ strongly in $C([0,T];L^2_x)$. For this, we will adapt a classical argument (see, for instance, Theorem 2.4 in \cite{Lions1996MathematicalMechanics}). First of all, we will need to perform a mollification, so we will extend the density field $\rho$, which is in $L^{\infty}([0,T]\times\Omega)$, to all of $\R^3$ by simply defining the density outside $\Omega$ to be $m$. Observe that we can also extend the velocity $u$ (and its first derivative) and the wavefunction $\psi$ (and its first three derivatives) to be identically zero outside $\Omega$. Combining this with the fact that characteristics starting inside/on/outside $\partial\Omega$ remain inside/on/outside $\partial\Omega$ (see proof of Lemma \ref{existence of solutions to approx continuity equation}) tells us that the density outside the domain remains $m$ at all times. Now, just as in Section \ref{the initial density}, we will define a sequence of mollifiers $\zeta_{h}(x) = \frac{1}{h^3} \zeta\left( \frac{x}{h} \right)$, where $h$ will eventually be taken to 0. We are now ready to establish the well-known ``renormalized solutions" of the continuity equation. Consider a weak solution $\rho$ of \eqref{continuity}, and mollify the equation to obtain:
    
    \begin{equation} \label{mollified continuity equation}
        \partial_t \rho_h + u\cdot\nabla\rho_h = \Psi_h + r_h
    \end{equation}\medskip
    
    \noindent where $g_h = g * \zeta_h$, $\Psi = 2\Lambda\Re(\overline{\psi}B\psi)$, and $r_h = u\cdot\nabla\rho_h - \left( u\cdot\nabla\rho \right)_h$. We multiply this by $\mathfrak{R}'(\rho_h)$, for a $C^1$ function $\mathfrak{R}:\R\mapsto\R$. This yields: \medskip
    
    \begin{equation} \label{mollified renormalized continuity equation}
        \partial_t \mathfrak{R}(\rho_h) + u\cdot\nabla\mathfrak{R}(\rho_h) = \mathfrak{R}'(\rho_h)[\Psi_h + r_h]
    \end{equation}\medskip
    
    From Lemma 2.3 in \cite{Lions1996MathematicalMechanics} and the boundedness of $\mathfrak{R}'$ (uniform, since $\rho$ only takes values in a compact subset of $\R$), we can see that $\mathfrak{R}'(\rho_h)r_h$ vanishes in $L^2([0,T];L^2_x)$ as $h\rightarrow 0$. Similarly, using a test function $\sigma$, we can take a distributional limit of the terms on the LHS (use properties of mollifiers $-$ see Theorem 6, Appendix C in \cite{Evans2010PartialEquations}). Lastly, as we will demonstrate shortly, $\psi$ and $B\psi$ have enough regularity so that we may pass to the limit in \eqref{mollified renormalized continuity equation}. Thus, we have shown that if $\rho$ is a weak solution of the continuity equation, then $\mathfrak{R}(\rho)$ solves (in a weak sense) \medskip 
    
    
    
    \begin{equation} \label{renormalized continuity equation}
        \partial_t \mathfrak{R}(\rho) + u\cdot\nabla\mathfrak{R}(\rho) = \mathfrak{R}'(\rho)\Psi
    \end{equation}\medskip

    Taking the difference of \eqref{mollified continuity equation} for two different parameters $h_1,h_2 >0$, we then multiply the resulting equation by $(\rho_{h_1} - \rho_{h_2})$ and integrate over $\Omega$ to obtain: \medskip
    
    \begin{align*}
        \frac{d}{dt} \frac{1}{2}\norm{\rho_{h_1} - \rho_{h_2}}_{L^2_x}^2 &= \int_{\Omega} \left( \rho_{h_1} - \rho_{h_2} \right) \left[ (\Psi_{h_1} - \Psi_{h_2}) + (r_{h_1} - r_{h_2}) \right] \\
        &\le \norm{\rho_{h_1} - \rho_{h_2}}_{L^2_x} \left[ \norm{\Psi_{h_1} - \Psi_{h_2}}_{L^2_x} + \norm{r_{h_1} - r_{h_2}}_{L^2_x} \right]
    \end{align*}\medskip
    
    \noindent which implies, by Gr\"onwall's inequality:
    
    \begin{equation*}
        \sup_{t\in [0,T]}\norm{\rho_{h_1} - \rho_{h_2}}_{L^2_x} \lesssim \norm{\rho(0)*\zeta_{h_1} - \rho(0)*\zeta_{h_2}}_{L^2_x} + T^{\frac{1}{2}}\left[ \norm{\Psi_{h_1} - \Psi_{h_2}}_{L^2_{t,x}} + \norm{r_{h_1} - r_{h_2}}_{L^2_{t,x}} \right]
    \end{equation*}\medskip
    
    All of the terms on the RHS vanish as $h_1,h_2\rightarrow 0$, thus giving us a Cauchy sequence in $C([0,T];L^2_x)$. Hence, $\rho_h$ converges to $\rho$ in $C([0,T];L^2_x)$. We have, so far, proved that our ``original approximations" of the continuity equation $\rho^N$ converge in $C_w([0,T];L^2_x)$ to $\rho$, and that the latter also belong to $C([0,T];L^2_x)$. To achieve what we set out to do, i.e., that $\rho^N$ converges strongly in $C([0,T];L^2_x)$ to $\rho$, it remains to show convergence of the norms. Explicitly, if there is a sequence of times $t^N \rightarrow t$, then we need $\rho^N(t^N)$ to converge in $L^2_x$ to $\rho(t)$. Returning to \eqref{approximate continuity equation}, we look at its renormalized version with $\mathfrak{R}(x) = x^2$, and integrate over $\Omega$ (and then from $0$ to $t^N$) to get: \medskip
    
    \begin{equation*}
        \int_{\Omega} (\rho^N(t^N))^2 = \int_{\Omega}(\rho_0^N)^2 + 2\Lambda\Re \int_0^{t^N}\int_{\Omega} \rho^N \overline{\psi^N} B^N\psi^N
    \end{equation*}\medskip
    
    Since we know that $\rho\in C^([0,T];L^2_x)$, we can do the same calculation with \eqref{continuity}, except the time integral goes from $0$ to $t$. \medskip
    
    \begin{equation*}
        \int_{\Omega} (\rho(t))^2 = \int_{\Omega}(\rho_0)^2 + 2\Lambda\Re \int_0^{t}\int_{\Omega} \rho \overline{\psi} B\psi
    \end{equation*}\medskip
    
    We now subtract the last two equations, and take the limit $N\rightarrow\infty$. Recall from Section \ref{the initial density} that $\rho_0^N \xrightarrow[]{L^2_x}\rho_0$, to cancel the first terms on the RHS. What remains is: \medskip
    
    \begin{equation*}
        \lim_{N\rightarrow\infty} \left[ \int_{\Omega} (\rho^N(t^N))^2 - \int_{\Omega} (\rho(t))^2 \right] 
        \!\begin{multlined}[t]
            = 2\Lambda\Re \lim_{N\rightarrow\infty} \left[ \int_0^{t^N}\int_{\Omega} (\rho^N-\rho) \overline{\psi^N} B^N\psi^N \right. \\ \left. + \int_0^{t^N}\int_{\Omega} \rho \left(\overline{\psi^N}-\overline{\psi}\right) B^N\psi^N + \int_0^{t^N}\int_{\Omega} \rho \overline{\psi} \left( B^N\psi^N - B\psi \right) \right. \\ \left. + \int_t^{t^N}\int_{\Omega} \rho \overline{\psi} B\psi \right] 
        \end{multlined} 
    \end{equation*}\medskip
    
    Thanks to the uniform boundedness of $\overline{\psi^N} B^N\psi^N$ in $L^1_{[0,T]} H^{\frac{3}{2}+\delta}_x$, we can use the strong convergence of $\rho^N$ to $\rho$ in $C([0,T];H^{0^-}_x)$ to handle the first term on the RHS. The second and third terms follow from simple H\"older's inequalities, once we have established the strong convergence of the wavefunction and of $B\psi$, both of which will be done later in this section. Finally, the last term is integrable on $[0,T]$, so as $t^N\rightarrow t$, it vanishes. In summary, \medskip
    
    \begin{equation} \label{strong convergence of density in C^0_t L^2_x}
        \rho^N \xrightarrow[N\rightarrow\infty]{C^0_t L^2_x} \rho
    \end{equation} \medskip

    \begin{rem}\label{C1 requirement for mathfrak(R)}
        In \cite{Lions1996MathematicalMechanics}, the well-known renormalization procedure was used to show (for the continuity equation without the source term) that $\rho^N$ converges to $\rho$ in $C([0,T];L^p_x)$ for $1\le p<\infty$. We have not pursued general $L^p_x$ norms here; only the case $p=2$ is considered. In addition, the absence of a source term in \cite{Lions1996MathematicalMechanics} meant it was sufficient to use $\mathfrak{R}\in C(\R)$. In our case, we require $\mathfrak{R}\in H^1(\R) \subset C(\R)$.
    \end{rem}\medskip

    \item \textit{Velocity}: According to the a priori estimates, $u^N\in L^{\infty}_{[0,T]}H^{\frac{3}{2}+\delta}_{d,x} \ \cap \ L^2_{[0,T]}H^2_{d,x}$ and $\partial_t u^N \in L^2_{[0,T]}L^2_{x} \subset L^2_{[0,T]}H^{0^-}_{x}$. Thus, Lemma \ref{aubin-lions} implies
    
    \begin{equation} \label{strong convergence of velocity}
        u^N \xrightarrow[C^0_t H^{\frac{3}{2}+\delta^-}_{d,x}]{L^2_t H^{2^{-}}_{d,x}} u
    \end{equation}\medskip
    
    \noindent with the convergence being strong (possibly of a subsequence). Moreover, since $u\in L^2_{[0,T]} H^{2^-}_{d,x}$ and $\partial_t u \in L^2_{[0,T]} L^2_x \subset L^2_{[0,T]}\left(H^{2^-}_{0,x} \right)^* \subset L^2_{[0,T]}\left(H^{2^-}_{d,x} \right)^*$, we can use Lemma \ref{lions-magenes} to deduce that $u\in C([0,T];H^{\frac{3}{2}+\delta}_{d,x})$. Thus, the velocity attains its initial condition in the strong sense. \medskip
    
    \item \textit{Momentum}: Based on the above results on the strong convergence of $\rho^N$ and $u^N$ (and in particular, the $L^{\infty}_{t} L^{\infty}_x$ bound on the latter), it is easy to see that $\rho^N u^N$ and $\rho^N u^N \otimes u^N$ converge in $C([0,T];L^2_x)$ to $\rho u$ and $\rho u \otimes u$, respectively. \medskip

    \begin{rem} \label{Kim's work - momentum and nonlinear terms}
        This is a good time to point out some interesting aspects of the calculations performed in \cite{Kim1987WeakDensity}. Since they did not have a positive lower bound on the density, there was no way to uniformly bound $\partial_t u^N$, i.e., a strongly convergent subsequence of $u^N$ could not be identified to manipulate the nonlinear (advective) term. The workaround this was to first show $\rho^N u^N$ converged to $\rho u$ in distribution (smooth test functions), and then use the uniform boundedness of $\rho^N u^N$ in $L^2(0,T;L^2_x)$ to extract a subsequence that converged weakly to some $g$. From the uniqueness of weak limits, it was argued that $g=\rho u$. After this, a uniform bound on $u^N \partial_t \rho^N$ was derived, and combining this with one on $\rho^N \partial_t u^N$, the quantity $\partial_t (\rho^N u^N)$ was uniformly bounded above (in an appropriate negative-order Sobolev space). Thus, compactness easily follows, to extract a strongly convergent subsequence $\rho^N u^N\rightarrow \rho u$. This allowed to show a weakly convergent subsequence for the nonlinear term $\rho^N u^N\otimes u^N$. The important takeaway from this brief detour is that the lack of a uniform bound on $\partial_t u^N$ was the main issue in \cite{Kim1987WeakDensity}; the strong convergence of the density in $C([0,T];L^2_x)$ (even in the $L^p_x$ norms) could have very well been established given the framework of their proof.
    \end{rem}\medskip

    \item \textit{Wavefunction}: We have $\psi^N\in L^{\infty}_{[0,T]}H^{\frac{5}{2}+\delta}_{0,x} \ \cap \ L^2_{[0,T]}H^{\frac{7}{2}+\delta}_{0,x}$ and $\partial_t \psi^N \in L^2_{[0,T]}L^2_{x} \subset L^2_{[0,T]}H^{0^-}_{x}$. Thus, Lemma \ref{aubin-lions} implies
    
    \begin{equation} \label{strong convergence of wavefunction}
        \psi^N \xrightarrow[C^0_t H^{\frac{5}{2}+\delta^-}_x]{L^2_t H^{\frac{7}{2}+\delta^-}_{0,x}} \psi
    \end{equation}\medskip
    
    Just as in the case of the velocity, combining $\psi\in L^2_{[0,T]} H^{\frac{7}{2}+\delta^-}_{0,x}$ and $\partial_t \psi \in L^2_{[0,T]} L^2_x \subset L^2_{[0,T]}\left(H^{\frac{7}{2}+\delta^-}_{0,x} \right)^*$, we can use Lemma \ref{lions-magenes} to get $\psi\in C([0,T];H^{\frac{5}{2}+\delta}_{0,x})$. Therefore, the wavefunction also attains its initial condition in the strong sense. Finally, using \eqref{strong convergence of velocity} and \eqref{strong convergence of wavefunction}, we can also conclude that $B^N \psi^N \xrightarrow[]{C^0_t L^2_x} B\psi$. \medskip

    \item \textit{Initial conditions}: By construction (Section \ref{the initial density}), we have $\rho_0^N \xrightarrow[]{L^2_x} \rho_0$. Also, Corollary \ref{truncated initial conditions are convergent} states that $u_0^N$ and $\psi_0^N$ converge to $u_0$ and $\psi_0$ in $H^{\frac{3}{2}+\delta}_x$ and $H^{\frac{5}{2}+\delta}_x$, respectively. For the momentum, we have:
    
    \begin{equation}
        \norm{\rho_0^N u_0^N - \rho_0 u_0}_{L^2_x} \le \norm{\rho_0^N - \rho_0}_{L^2_x} \norm{u_0^N}_{L^{\infty}_x} + \norm{\rho_0}_{L^{\infty}_x} \norm{u_0^N - u_0}_{L^2_x}
    \end{equation}\medskip
    
    Using the embedding $H^{\frac{3}{2}+\delta}_x \subset L^{\infty}_x$ to handle the velocity norm in the first term of the RHS, it is easy to see that the initial momentum converges in the $L^2_{x}$ norm. \medskip

    \item \textit{Time derivatives}: From \eqref{final a priori bounds} and \eqref{partial_t rho^N uniformly bounded in L^2_t H^-1_x}, we know that $\partial_t \psi^N,\partial_t u^N \in L^2_{[0,T]}L^2_x \subset L^2_{[0,T]}H^{-1}_x$, and $\partial_t \rho^N \in L^2_{[0,T]}H^{-1}_x$ (all uniformly in $N$). In other words, all the fields are in $H^1_{[0,T]}H^{-1}_x$, and a weakly convergent subsequence can be extracted in the same Hilbert space. Thus, the final convergent fields also belong to $H^1_{[0,T]}H^{-1}_x$, implying that they can be used as test functions, which is justified by interpreting the integral over $\Omega$ as an inner product between functions from a Hilbert space and its dual.\medskip

\end{enumerate}

\subsubsection{Passing to the limit}

We finally return to the weak solution of the Pitaevskii model, as defined in \eqref{weak solution wavefunction}, \eqref{weak solution velocity} and \eqref{weak solution density}. First, observe that \eqref{expansion of approximate NSE} is obtained by taking the $L^2_x$ inner product of \eqref{approximate NSE} with $a_j$ (eigenfunctions of the bi-Stokes operator), for each $j \in \{ 1,2,\dots N \}$. Therefore, we can multiply \eqref{expansion of approximate NSE} by some $\eta^u_j \in C^{\infty}([0,T])$ and sum over $j$ from $1$ to some $N'\le N$. Combining this with \eqref{approximate continuity equation}, and integrating over $[0,T]$:

\begin{equation} \label{weak solution velocity - finite test function}
\begin{multlined}
    -\int_0^T \int_{\Omega} \left[ \rho^N u^N\cdot \partial_t \left( \sum_{j=1}^{N'} \eta^u_j a_j \right) + \rho^N u^N\otimes u^N:\nabla\left( \sum_{j=1}^{N'} \eta^u_j a_j \right) \right. \\ \left. - \nu\nabla u^N:\nabla\left( \sum_{j=1}^{N'} \eta^u_j a_j \right) - 2\Lambda\left( \sum_{j=1}^{N'} \eta^u_j a_j \right)\cdot\Im(\nabla\overline{\psi^N}B^N\psi^N) \right] dx \ dt \\ = \int_{\Omega} \left[ \rho^N_0 u^N_0 \left( \sum_{j=1}^{N'} \eta^u_j(t=0) a_j \right) - \rho^N(T)u^N(T)\left( \sum_{j=1}^{N'} \eta^u_j(T) a_j \right) \right] dx
\end{multlined}
\end{equation}\medskip

A similar procedure with \eqref{approximate NLS} and \eqref{approximate continuity equation} yields

\begin{equation} \label{weak solution wavefunction - finite test function}
\begin{multlined} 
    -\int_0^T \int_{\Omega} \left[ \psi^N\partial_t \left( \sum_{j=1}^{N'} \overline{\eta^w_j} b_j \right) + \frac{1}{2i}\nabla\psi^N\cdot\nabla\left( \sum_{j=1}^{N'} \overline{\eta^w_j} b_j \right) \right. \\ \left. - \Lambda\left( \sum_{j=1}^{N'} \overline{\eta^w_j} b_j \right)B^N\psi^N - i\mu\left( \sum_{j=1}^{N'} \overline{\eta^w_j} b_j \right)\abs{ \psi^N}^2\psi^N \right] dx \ dt \\ = \int_{\Omega} \left[ \psi^N_0\left( \sum_{j=1}^{N'} \overline{\eta^w_j}(t=0) b_j \right) - \psi^N(T)\left( \sum_{j=1}^{N'} \overline{\eta^w_j}(T) b_j \right) \right] dx
\end{multlined}
\end{equation}\medskip

\noindent and
    
\begin{equation} \label{weak solution density - finite test function}
\begin{multlined} 
    -\int_0^T \int_{\Omega} \left[ \rho^N \partial_t \left( \sum_{j=1}^{N'} \eta^d_j v_j \right) + \rho^N u^N\cdot\nabla\left( \sum_{j=1}^{N'} \eta^d_j v_j \right) \right. \\ \left. + 2\Lambda \left( \sum_{j=1}^{N'} \eta^d_j v_j \right) \Re\left(\overline{\psi^N}B^N\psi^N\right) \right] dx \ dt \\ = \int_{\Omega} \left[ \rho^N_0 \left( \sum_{j=1}^{N'} \eta^d_j(0) v_j \right) - \rho^N(T)\left( \sum_{j=1}^{N'} \eta^d_j(T) v_j \right) \right] dx
\end{multlined}
\end{equation}\medskip

\noindent where $\eta^w_j$ and $\eta^d_j$ also belong to $C^{\infty}([0,T])$, except that the former takes complex values and the latter, real values. The functions $b_j$ are the eigenfunctions of the penta-Laplacian, used earlier for setting up the semi-Galerkin truncated wavefunction. Finally, the sequence $v_j\in C^{\infty}(\overline{\Omega})$.\medskip

The linear combinations just considered, like $\sum_{j=1}^{N'} \eta^u_j a_j$, fit the criteria of test functions in Definition \ref{definition of weak solutions}. Therefore, for a fixed $N'$, with all the weak and strong convergences in Section \ref{weak and strong convergences}, we can pass to the limit $N\rightarrow\infty$ in \eqref{weak solution velocity - finite test function}, \eqref{weak solution wavefunction - finite test function} and \eqref{weak solution density - finite test function}. This leads us back to \eqref{weak solution velocity}, \eqref{weak solution wavefunction} and \eqref{weak solution density} respectively, with the caveat that the test functions are still smooth in space and time. The test functions appearing in Definition \ref{definition of weak solutions} had space-time regularities that were in Sobolev spaces, and can thus they can be approximated by smooth functions (with the appropriate boundary conditions, which explains the choice of the bases used in the semi-Galerkin truncation). Hence, we can now pass to the limit $N'\rightarrow\infty$ to obtain the required regularities of the test functions and in turn, the weak solutions we seek.\medskip

\subsection{The energy equality}

The smooth approximations to the weak solutions satisfy an energy equality, given by \eqref{energy bound E0} and \eqref{E0 definition}.

\begin{equation} \label{energy equality for smooth approximations}
\begin{multlined}
    \left( \frac{1}{2}\norm{\sqrt{\rho^N}u^N}_{L^2_x}^2 + \frac{1}{2}\norm{\nabla\psi^N}_{L^2_x}^2 + \frac{\mu}{2}\norm{\psi^N}_{L^4_x}^4 \right)(t) + \nu\norm{\nabla u^N}_{L^2_{[0,t]}L^2_x}^2 + 2\Lambda\norm{B^N\psi^N}_{L^2_{[0,t]}L^2_x}^2 \\ = \frac{1}{2}\norm{\sqrt{\rho_0^N}u_0^N}_{L^2_x}^2 + \frac{1}{2}\norm{\nabla\psi_0^N}_{L^2_x}^2 + \frac{\mu}{2}\norm{\psi_0^N}_{L^4_x}^4 \quad a.e. \ t\in [0,T]
\end{multlined}
\end{equation}\medskip

From our choice of the initial conditions and their approximations (see Section \ref{the initial conditions}), we can ensure that as $N\rightarrow\infty$, the RHS converges to the initial energy $E_0$ defined in \eqref{E0 definition}. Indeed, \medskip

\begin{equation} \label{convergence of initial kinetic energy}
    \begin{aligned} 
        \abs{\norm{\sqrt{\rho_0^N}u_0^N}_{L^2_x}^2 - \norm{\sqrt{\rho_0}u_0}_{L^2_x}^2} &= \abs{\int_{\Omega} \rho_0^N \abs{u_0^N}^2 - \rho_0 \abs{u_0}^2} \\
        &\lesssim \norm{\rho_0^N - \rho_0}_{L^2_x} \norm{u_0^N}_{L^6_x}^2 + \norm{\rho_0}_{L^{\infty}_x} \norm{u_0^N + u_0}_{L^2_x} \norm{u_0^N - u_0}_{L^2_x} \\
        &\xrightarrow[N\rightarrow\infty]{} 0
\end{aligned}
\end{equation}\medskip

Moreover, based on the results of Section \ref{weak and strong convergences}, we can conclude that all the terms on the LHS converge strongly to the corresponding terms with the approximate solutions replaced by the weak solution. (The first term on the LHS can be dealt with the same way as the first term on the RHS in \eqref{convergence of initial kinetic energy}, by simply including a $\sup_t$ outside the absolute values.) \medskip

\qed

\medskip

This completes the proof of Theorem \ref{local existence} - local existence of weak solutions and the energy equality. We will now give a quick proof of Proposition \ref{energy inequality - Kim}. \medskip

\section{Proof of Proposition \ref{energy inequality - Kim}} \label{energy inequality proof - Kim}

Using the approximate solutions, it is easy to see that the corresponding energy equation in this case is given by:

\begin{equation} \label{energy equation approximate - Kim}
    \frac{1}{2}\norm{\sqrt{\rho^N}u^N}_{L^2_x}^2 (t) + \nu\norm{\nabla u^N}_{L^2_{[0,t]}L^2_x}^2 = \frac{1}{2}\norm{\sqrt{\rho_0^N}u_0^N}_{L^2_x}^2 \quad a.e. \ t\in [0,T]
\end{equation}\medskip

The bounds (uniform in $N$) on the approximations, and their convergence properties are as follows. \medskip

\begin{equation} \label{uniform a priori bounds - Kim}
    \begin{gathered}
        \norm{\sqrt{\rho^N}\partial_t u^N}_{L^2_{[0,T]} L^2_x} , \norm{u^N}_{L^{\infty}_{[0,T]} H^1_x} , \norm{u^N}_{L^2_{[0,T]} H^2_x} , \norm{\partial_t \rho^N}_{L^{\infty}_{[0,T]} H^{-1}_x} \le C \\
        \frac{1}{N} \le \rho^N(t,x) \le M+\frac{1}{N} \quad a.e. \ (t,x) \in (0,T\times \Omega) \\
        \rho_0^N \xrightarrow[]{L^2_x} \rho_0 \quad , \quad u_0^N \xrightarrow[]{L^2_x} u_0
    \end{gathered}
\end{equation}\medskip

\noindent where the constant $C$ depends on the initial conditions, the time $T$. As explained in Remark \ref{Kim's work - momentum and nonlinear terms}, compactness arguments were used to extract some strongly convergent subsequences (relabeled). \medskip

\begin{equation} \label{strong convergence - Kim}
    \begin{gathered}
        \rho^N \xrightarrow[]{L^2_t H^{-1}_x} \rho \\
        \rho^N u^N \xrightarrow[]{L^2_t H^{-1}_x} \rho u \\
    \end{gathered}
\end{equation}\medskip

Given these estimates, the RHS of \eqref{energy equation approximate - Kim} can be shown to converge to $\frac{1}{2}\norm{\sqrt{\rho_0}u_0}_{L^2_x}^2$ in exactly the same way as \eqref{convergence of initial kinetic energy}. Due to the weak convergence of $u^N$ in $L^2(0,T;H^1_{d,x})$, and the lower semicontinuity of the norm, we have $\norm{\nabla u}_{L^2_{[0,t]}L^2_x}^2 \le  \liminf_{N\rightarrow\infty} \norm{\nabla u^N}_{L^2_{[0,t]}L^2_x}^2$. What remains is the first term on the LHS, and we will proceed as follows: \medskip

\begin{equation} \label{kinetic energy in L^2_t L^2_x - Kim}
    \begin{aligned}
        \int_0^T \int_{\Omega} \left[ \rho^N\abs{u^N}^2 - \rho\abs{u}^2 \right] &= \int_0^T \int_{\Omega} (\rho^N u^N - \rho u)\cdot u^N + \int_0^T \int_{\Omega} \rho u\cdot(u^N - u) \\
        &\le T^{\frac{1}{2}} \norm{\rho^N u^N - \rho u}_{L^2_t H^{-1}_x} \norm{u^N}_{L^{\infty}_t H^1_x} + \int_0^T \langle u^N - u,\rho u \rangle_{H^1_0 \times H^{-1}}
    \end{aligned}
\end{equation}\medskip

\noindent The first term vanishes due to the strong convergence of $\rho^N u^N$, while the second term goes to zero due to the weak-* convergence of $u^N$. Thus, we conclude that $\norm{\sqrt{\rho^N}u^N}_{L^2_{[0,T]} L^2_x}$ converges to $\norm{\sqrt{\rho}u}_{L^2_{[0,T]} L^2_x}$. Let us now define \medskip

\begin{equation} \label{defining f^N(t)}
    f^N(t) := \norm{\sqrt{\rho^N}u^N}_{L^2_x}(t) \quad , \quad f^N(t) := \norm{\sqrt{\rho}u}_{L^2_x}(t)
\end{equation}\medskip

We have shown that $\norm{f^N}_{L^2_t} \rightarrow \norm{f}_{L^2_t}$. From a calculation mirroring that in \eqref{kinetic energy in L^2_t L^2_x - Kim}, it is easy to see that $f^N \xrightharpoonup[]{\mathfrak{D}^*_t} f$, i.e., for all $\Theta \in C^{\infty}_c[0,T]$, $\lim_{N\rightarrow\infty}\int_0^T (f^N - f)\Theta = 0$. Furthermore, since the RHS of \eqref{energy equation approximate - Kim} converges strongly, it is bounded; therefore, $\norm{f^N}_{L^2_t}$ is also bounded\footnote{It should be mentioned here that $T$ is finite, depending only on the initial data and size of the domain.} uniformly in $N$. We can extract a subsequence (relabeled) that is weakly convergent in $L^2(0,T)$, i.e., $f^N \xrightharpoonup[]{L^2_t} g$, where $g\in L^2(0,T)$. Combining this with the convergence in distribution $(\mathfrak{D}^*_t)$ and the uniqueness of weak limits, one deduces that indeed, $g=f$ a.e. In summary, we have shown that $f^N$ converges to $f$ weakly in $L^2(0,T)$, while $\norm{f^N}_{L^2_t}$ converges to $\norm{f}_{L^2_t}$. This implies the strong convergence of $f^N$ to $f$ in $L^2(0,T)$, which in turn means that we can select a subsequence that converges a.e. Consequently, we have shown that the energy inequality\footnote{It is worth noting that the obstacle to an energy \textit{equality} in the work by Kim was the lack of strong convergence of the dissipative term; yet again, this boils down to the fact that there is no uniform bound on $\partial_t u$ (since the density is not bounded below).} holds along a subsequence, for a.e. $t\in [0,T]$. \medskip

\qed

\bigskip



\bibliographystyle{alpha}
\bibliography{references}

\newcommand{\etalchar}[1]{$^{#1}$}
\begin{thebibliography}{BRMFC03}

\bibitem[AF03]{Adams2003SobolevSpaces}
Robert~A Adams and John J~F Fournier.
\newblock {\em {Sobolev Spaces}}.
\newblock Elsevier Science, Netherlands, second edition, 2003.

\bibitem[AJ38]{Allen1938New2}
J.~F. Allen and H.~Jones.
\newblock {New phenomena connected with heat flow in helium II [2]}.
\newblock {\em Nature}, 141(3562):243--244, 1938.

\bibitem[AM38]{Allen1938FlowII}
J.~F. Allen and A.~D. Misener.
\newblock {Flow Phenomena in Liquid Helium II}.
\newblock {\em Nature}, 142(3597):643--644, 1938.

\bibitem[AM09]{Antonelli2009OnDynamics}
Paolo Antonelli and Pierangelo Marcati.
\newblock {On the Finite Energy Weak Solutions to a System in Quantum Fluid
  Dynamics}.
\newblock {\em Communications in Mathematical Physics}, 287:657--686, 2009.

\bibitem[AM12]{Antonelli2012TheDimensions}
Paolo Antonelli and Pierangelo Marcati.
\newblock {The Quantum Hydrodynamics System in Two Space Dimensions}.
\newblock {\em Archive for Rational Mechanics and Analysis}, 203(2):499--527,
  2012.

\bibitem[AM15]{AntonelliPaolo2015FiniteSuperfluidity}
{Antonelli, Paolo} and {Marcati, Pierangelo}.
\newblock {Finite energy global solutions to a two-fluid model arising in
  superfluidity}.
\newblock {\em Bulletin of the Institute of Mathematics Academia Sinica},
  10(3):349--373, 2015.

\bibitem[BD04]{Bresch2004QuelquesKorteweg}
Didier Bresch and Benoît Desjardins.
\newblock {Quelques mod{\`{e}}les diffusifs capillaires de type Korteweg}.
\newblock {\em Comptes Rendus M{\'{e}}canique}, 332:881--886, 2004.

\bibitem[BLS06]{Bewley2006VisualisationVortices}
Gregory~P Bewley, Daniel~P Lathrop, and Katepalli~R. Sreenivasan.
\newblock {Visualisation of quantised vortices}.
\newblock {\em Nature}, 441(7093):588, 2006.

\bibitem[BRMFC03]{Boldrini2003Semi-GalerkinFluids}
José~L. Boldrini, Marko~A. Rojas-Medar, and Enrique Fern{\'{a}}ndez-Cara.
\newblock {Semi-Galerkin approximation and strong solutions to the equations of
  the nonhomogeneous asymmetric fluids}.
\newblock {\em Journal des Mathematiques Pures et Appliquees},
  82(11):1499--1525, 2003.

\bibitem[CDS12]{Carles2012MadelungKorteweg}
Rémi Carles, Raphaël Danchin, and Jean-Claude Saut.
\newblock {Madelung, Gross-Pitaevskii and Korteweg}.
\newblock {\em Nonlinearity}, 25(10):2843--2873, 2012.

\bibitem[CK03]{Choe2003StrongFluids}
Hi~Jun Choe and Hyunseok Kim.
\newblock {Strong solutions of the Navier-Stokes equations for nonhomogeneous
  incompressible fluids}.
\newblock {\em Communications in Partial Differential Equations},
  28(5-6):1183--1201, 2003.

\bibitem[CKS{\etalchar{+}}]{CollianderWell-posednessEquations}
Jim Colliander, Mark Keel, Gigliola Staffilani, Hideo Takaoka, and Terry Tao.
\newblock {Well-posedness for non-linear dispersive and wave equations}.

\bibitem[CMS08]{Carles2008OnGases}
Rémi Carles, Peter~A. Markowich, and Christof Sparber.
\newblock {On the Gross-Pitaevskii equation for trapped dipolar quantum gases}.
\newblock {\em Nonlinearity}, 21(11):2569--2590, 11 2008.

\bibitem[CWHM15]{Chandler-Wilde2015InterpolationCounterexamples}
Simon~N. Chandler-Wilde, David~P. Hewett, and Andrea Moiola.
\newblock {Interpolation of Hilbert and Sobolev Spaces: Quantitative Estimates
  and Counterexamples}.
\newblock {\em Mathematika}, 61:414--443, 4 2015.

\bibitem[DFM15]{Donatelli2015Well/IllProblems}
Donatella Donatelli, Eduard Feireisl, and Pierangelo Marcati.
\newblock {Well/Ill Posedness for the Euler-Korteweg-Poisson System and Related
  Problems}.
\newblock {\em Communications in Partial Differential Equations},
  40(7):1314--1335, 2015.

\bibitem[Dod16]{Dodson2016Global2}
Benjamin Dodson.
\newblock {Global well-posedness and scattering for the defocusing,
  L{\^{}}2-critical, nonlinear Schr{\"{o}}dinger equation when d = 2}.
\newblock {\em Duke Mathematical Journal}, 165(18):3435--3516, 2016.

\bibitem[Eva10]{Evans2010PartialEquations}
Lawrence~C. Evans.
\newblock {\em {Partial Differential Equations}}.
\newblock American Mathematical Society, Providence, Rhode Island, second
  edition, 2010.

\bibitem[Fei04]{Feireisl2004DynamicsFluids}
Eduard Feireisl.
\newblock {\em {Dynamics of Viscous Compressible Fluids}}.
\newblock Oxford University Press, 2004.

\bibitem[Fey55]{Feynman1955ApplicationHelium}
Richard~P. Feynman.
\newblock {Application of Quantum Mechanics to Liquid Helium}.
\newblock In {\em Progress in Low Temperature Physics}, volume~1, chapter~2,
  pages 17--53. Elsevier, 1955.

\bibitem[FHR19]{Fefferman2019SimultaneousEigenspaces}
Charles~L. Fefferman, Karol~W. Hajduk, and James~C. Robinson.
\newblock {Simultaneous approximation in Lebesgue and Sobolev norms via
  eigenspaces}.
\newblock {\em arXiv preprint}, 2019.

\bibitem[GGS10]{Gazzola2010PolyharmonicDomains}
Filippo Gazzola, Hans-Christoph Grunau, and Guido Sweers.
\newblock {\em {Polyharmonic Boundary Value Problems: A monograph on positivity
  preserving and nonlinear higher order elliptic equations in bounded
  domains}}.
\newblock Springer Berlin Heidelberg, 2010.

\bibitem[GS11]{Guermond2011APowers}
Jean-Luc Guermond and Abner Salgado.
\newblock {A note on the Stokes operator and its powers}.
\newblock {\em Journal of Applied Mathematics and Computing}, 36:241--250,
  2011.

\bibitem[HL94]{Hattori1994SolutionsType}
Harumi Hattori and Dening Li.
\newblock {Solutions for two-dimensional system for materials of Korteweg
  type}.
\newblock {\em SIAM J. Math. Anal.}, 25(1):85--98, 1994.

\bibitem[HL96]{HattoriHarumi1996GlobalMaterials}
{Hattori, Harumi} and {Li, Dening}.
\newblock {Global solutions of a high-dimensional system for Korteweg
  materials}.
\newblock {\em Journal of Mathematical Analysis and Applications}, 198:84--97,
  1996.

\bibitem[JL04]{JungelAnsgar2004QuantumDecay}
{J{\"{u}}ngel, Ansgar} and Hailiang Li.
\newblock {Quantum Euler-Poisson systems: Global existence and exponential
  decay}.
\newblock {\em Quarterly of Applied Mathematics}, 62(3):569--600, 2004.

\bibitem[JMR02]{JungelAnsgar2002LocalEquations}
{J{\"{u}}ngel, Ansgar}, {Mariano, Maria Cristina}, and {Rial, Diego}.
\newblock {Local existence of solutions to the transient quantum hydrodynamics
  equations}.
\newblock {\em Mathematical Models and Methods in Applied Sciences},
  12(4):485--495, 2002.

\bibitem[JT21a]{Jayanti2021GlobalEquations}
Pranava~Chaitanya Jayanti and Konstantina Trivisa.
\newblock {Global Regularity of the 2D HVBK equations}.
\newblock {\em J Nonlinear Sci}, 31(2), 2021.

\bibitem[JT21b]{Jayanti2021UniquenessSuperfluidity}
Pranava~Chaitanya Jayanti and Konstantina Trivisa.
\newblock {Uniqueness in a Navier-Stokes-nonlinear-Schr{\"{o}}dinger model of
  superfluidity}.
\newblock {\em arXiv:2109.14083}, 2021.

\bibitem[J{\"{u}}n10]{Jungel2010GlobalFluids}
Ansgar J{\"{u}}ngel.
\newblock {Global Weak Solutions to Compressible Navier–Stokes Equations for
  Quantum Fluids}.
\newblock {\em SIAM J. Math. Anal.}, 42(3):1025--1045, 2010.

\bibitem[Kap38]{Kapitza1938Viscosity-Point}
P~Kapitza.
\newblock {Viscosity of Liquid Helium below the {$\lambda$}-Point}.
\newblock {\em Nature}, 141(3558):74, 1938.

\bibitem[Kaz74]{Kazhikov1974Fluid}
A.~V. Kazhikov.
\newblock { Solvability of the initial and boundary value problem for the
  equations of motion of an inhomogeneous viscous incompressible fluid}.
\newblock {\em Soviet Phys Dokl}, 19(6):331--332, 1974.

\bibitem[Kim87]{Kim1987WeakDensity}
Jong~Uhn Kim.
\newblock {Weak solutions of an initial boundary value problem for an
  incompressible viscous fluid with non-negative density}.
\newblock {\em SIAM J. Math. Anal.}, 18(1):89--96, 1987.

\bibitem[Lan41]{Landau1941TheoryII}
Lev Landau.
\newblock {Theory of the superfluidity of helium II}.
\newblock {\em Physical Review}, 60(4):356--358, 1941.

\bibitem[Lio96a]{Lions1996MathematicalMechanicsb}
Pierre-Louis Lions.
\newblock {\em {Mathematical Topics in Fluid Mechanics}}, volume~2.
\newblock Oxford University Press, 1996.

\bibitem[Lio96b]{Lions1996MathematicalMechanics}
Pierre-Louis Lions.
\newblock {\em {Mathematical Topics in Fluid Mechanics}}, volume~1.
\newblock Oxford University Press, 1996.

\bibitem[LS78]{Ladyzhenskaya1978UniqueFluids}
O.~A. Ladyzhenskaya and V.~A. Solonnikov.
\newblock {Unique solvability of an initial- and boundary-value problem for
  viscous incompressible nonhomogeneous fluids}.
\newblock {\em Journal of Soviet Mathematics}, 9(5):697--749, 1978.

\bibitem[Man05]{Manouzi2005AHyper-dissipation}
H.~Manouzi.
\newblock {A mixed variational formulation for the Navier-Stokes problem with
  hyper-dissipation}.
\newblock {\em Computers {\&} Mathematics with Applications},
  50(10-12):1639--1646, 2005.

\bibitem[MB02]{Majda2002VorticityFlow}
Andrew~J Majda and Andrea~L Bertozzi.
\newblock {\em {Vorticity and Incompressible Flow}}.
\newblock Cambridge University Press, Cambridge, UK, first edition, 2002.

\bibitem[Ons53]{Onsager1953IntroductoryTalk}
Lars Onsager.
\newblock {Introductory talk}.
\newblock In {\em Proc. Int. Conf. Theor. Physics}, pages 887--880, Tokyo,
  1953.

\bibitem[Pit59]{Pitaevskii1959PhenomenologicalPoint}
Lev~P Pitaevskii.
\newblock {Phenomenological theory of superfluidity near the Lambda point}.
\newblock {\em Soviet Physics JETP}, 35(8):282--287, 1959.

\bibitem[RRS16]{Robinson2016TheEquations}
James~C. Robinson, Jose~L. Rodrigo, and Witold Sadowski.
\newblock {\em {The Three-Dimensional Navier-Stokes Equations}}.
\newblock Cambridge University Press, 2016.

\bibitem[Sim86]{Simon1986CompactLpOTB}
Jacques Simon.
\newblock {Compact sets in the space {\{}{\$}L{\^{}}p(O,T;B){\$}{\}}}.
\newblock {\em Annali di Matematica Pura ed Applicata}, 146:65--96, 1986.

\bibitem[Sim90]{Simon1990NonhomogeneousPressure}
Jacques Simon.
\newblock {Nonhomogeneous Viscous Incompressible Fluids: Existence of Velocity,
  Density, and Pressure}.
\newblock {\em SIAM Journal on Mathematical Analysis}, 21(5):1093--1117, 1990.

\bibitem[Soh11]{Sohinger2011BoundsEquations}
Vedran Sohinger.
\newblock {\em {Bounds on the growth of high Sobolev norms of solutions to
  nonlinear Schr{\"{o}}dinger equations}}.
\newblock PhD thesis, Massachusetts Institute of Technology, 2011.

\bibitem[Tao06]{Tao2006NonlinearAnalysis}
Terence Tao.
\newblock {\em {Nonlinear dispersive equations: local and global analysis}}.
\newblock American Mathematical Society, 2006.

\bibitem[Tem77]{Temam1977Navier-StokesAnalysis}
Roger Temam.
\newblock {\em {Navier-Stokes Equations: Theory and Numerical Analysis}}.
\newblock North-Holland Publishing Company, first edition, 1977.

\bibitem[Vin04]{Vinen:808382}
W~F Vinen.
\newblock {The physics of superfluid helium}.
\newblock Technical report, CERN, 2004.

\bibitem[VY16]{Vasseur2016GlobalDamping}
Alexis~F. Vasseur and Cheng Yu.
\newblock {Global weak solutions to the compressible quantum Navier-Stokes
  equations with damping}.
\newblock {\em SIAM J. Math. Anal.}, 48(2):1489--1511, 2016.

\bibitem[WG20]{Wang2020AModel}
Guangwu Wang and Boling Guo.
\newblock {A blow-up criterion of strong solutions to the quantum hydrodynamic
  model}.
\newblock {\em Acta Mathematica Scientia}, 40(3):795--804, 2020.

\bibitem[WG21]{Wang2021AModel}
Guangwu Wang and Boling Guo.
\newblock {A new blow-up criterion of the strong solution to the quantum
  hydrodynamic model}.
\newblock {\em Applied Mathematics Letters}, 119:107045, 2021.

\end{thebibliography}

%
%

%
%

\end{document}